\documentclass[leqno,11pt]{amsart}
\usepackage{fancyhdr}
\usepackage{amsmath}
\usepackage{mathtools}
\usepackage{amsfonts}
\usepackage{amsthm}
\usepackage{amssymb}
\usepackage{dsfont}
\usepackage{graphicx}
\usepackage[dvipsnames]{xcolor}
\usepackage{extarrows}
\usepackage{marginnote}
\usepackage{pifont}
\usepackage{enumerate}
\usepackage[shortlabels]{enumitem}
\usepackage[english]{babel}
\usepackage[latin1]{inputenc}
\usepackage[T1]{fontenc}
\usepackage[colorlinks]{hyperref}
\usepackage[margin=0.9in]{geometry}
\usepackage{amsthm}
\usepackage{todonotes}
\usepackage{pgf,tikz}
\usetikzlibrary{arrows}

\allowdisplaybreaks
\makeatletter
\long\def\unmarkedfootnote#1{{\long\def\@makefntext##1{##1}\footnotetext{#1}}}
\makeatother

\AtBeginDocument{%
\hypersetup{%
linkcolor=blue,%
citecolor=blue,%
}%
}%
\newcommand{\barint}{
         \rule[.036in]{.12in}{.009in}\kern-.16in
          \displaystyle\int  }
     \newcommand{\om}{\Omega}
     \newcommand{\ep}{{\varepsilon}}
\newcommand{\R}{{\mathbb{R}}}
\newcommand{\N}{{\mathbb{N}}}
\newcommand{\rn}{{\mathbb{R}^{n}}}
\newcommand{\Rm}{{\mathbb{R}^{m}}}

\newcommand{\rp}{{[0,\infty)}}
\newcommand{\ve}{{\varepsilon}}
\newcommand{\wt}{\widetilde}
\newcommand{\vt}{\vartheta}

\newcommand{\iO}{\int_{\Omega}}

\newcommand{\vp}{\varphi}
\newcommand{\bu}{{\bar{u}}}

\newcommand{\dv}{\mathrm{div}}

\newcommand{\PD}{{\Phi_\Diamond}}
\newcommand{\T}{{T_t}}

\allowdisplaybreaks
\newtheorem{definition}{Definition}[section]
\newtheorem{lemma}[definition]{Lemma}
\newtheorem{theorem}[definition]{Theorem}
\newtheorem{proposition}[definition]{Proposition}

\newtheorem{remark}[definition]{Remark}
\newtheorem{ex}{Example}

\newcommand{\nocontentsline}[3]{}
\newcommand{\tocless}[2]{\bgroup\let\addcontentsline=\nocontentsline#1{#2}\egroup}

\begin{document}

\title%Paper:\\
%{\color{red} Weak and approximable solutions \\ to fully anisotropic elliptic problems
%\\ or \\
[Fully anisotropic elliptic problems]{Fully anisotropic elliptic problems\\ with minimally integrable data
%\\ or \\
%Existence and regularity of solutions \\ to fully anisotropic elliptic problems
%\\ or \\
%  Fully anisotropic elliptic problems with $L^1$ or measure data
}

\frenchspacing
\numberwithin{equation}{section}

\author{Angela Alberico}\address{Angela Alberico\\  Istituto per le  Applicazioni del Calcolo ``M.
Picone'', Consiglio Nazionale delle Ricerche, Via Castellino 111,
80131 Napoli, Italy}
\email{\texttt{a.alberico@iac.cnr.it}}

\author{Iwona Chlebicka} \address{Iwona Chlebicka\\Institute of Applied
Mathematics and Mechanics, University of Warsaw\\ul. Banacha 2, 02-097  Warsaw, Poland} \email{\texttt{i.chlebicka@mimuw.edu.pl}}

\author{Andrea Cianchi}
\address{Andrea Cianchi\\ Dipartimento di Matematica e Informatica  ``U. Dini", Universit\`a di Firenze,  Viale Morgagni 67/A,  50134 Firenze, Italy  {\rm and } Istituto per le Applicazioni del  Calcolo ``M. Picone'', Consiglio Nazionale delle Ricerche, Via  Castellino 111, 80131 Napoli, Italy}  \email{\texttt{e-mail:andrea.cianchi@unifi.it}}

\author{Anna Zatorska-Goldstein} \address{Anna Zatorska-Goldstein\\  Institute of Applied Mathematics and Mechanics, University of Warsaw\\ul. Banacha 2, 02-097 Warsaw, Poland} \email{\texttt{azator@mimuw.edu.pl}}

\pagestyle{myheadings} \thispagestyle{plain}

\iffalse

\author[1]{Angela Alberico}%{\footnote{a.alberico@na.iac.cnr.it}}}
\author[2]{Iwona Chlebicka}%{\footnote{i.chlebicka@mimuw.edu.pl}}}
\author[3,1]{Andrea Cianchi}%{\footnote{andrea.cianchi@unifi.it}}}
\author[2]{Anna Zatorska-Goldstein}%{\footnote{azator@mimuw.edu.pl}}}

\affil[1]{\small Istituto per le Applicazioni del Calcolo ``M.
Picone'', Consiglio Nazionale delle Ricerche, Via Castellino 111,
80131 Napoli, Italy}
\affil[2]{{\small Institute of Applied
Mathematics and Mechanics, University of Warsaw\\ ul. Banacha 2,
02-097 Warsaw, Poland}}
\affil[3]{\small Dipartimento di Matematica e Informatica ``U. Dini''\\
Universit\`a di Firenze\\
Viale Morgagni 67/A\\
50134 Firenze, Italy}

\fi
%%%%%%%%%%%%%%%%%%%%%%%%%%%%%%%%%%%%%%%%%%%%%%%%%%%%%%%%%%%%%%%%%%%%%%

\maketitle

\begin{abstract}
 We investigate   nonlinear elliptic
Dirichlet problems whose growth is driven by a general anisotropic $N$-function, which is not necessarily of power type and need not satisfy the $\Delta_2$ nor the $\nabla _2$-condition. Fully anisotropic, non-reflexive Orlicz-Sobolev spaces provide a natural functional framework associated with these problems.  Minimal integrability assumptions are detected on the datum on the right-hand side of the equation ensuring existence and uniqueness of weak solutions.  When merely integrable, or even measure, data are allowed, existence of suitably further generalized solutions -- in~the approximable sense -- is established. Their maximal regularity in Marcinkiewicz--type spaces is  exhibited as well. Uniqueness of approximable solutions is also  proved  in case of $L^1$--data.
\end{abstract}

%%%%%%%%%%%%%%%%%%%%%%%%%%%%%%%%%%%%%%%%%%%%%%%%%%%%%%%%%%%%%%%%%%%%%%%%%%%%%%%%%%%%%%%%%%%%%%%%%%%%%%%%%
\setcounter{tocdepth}{1}
\tableofcontents

\unmarkedfootnote {
}

\section{Introduction}\label{intro}

This paper concerns  Dirichlet problems for elliptic equations of the form
\begin{equation}\label{eq:main:f}
\begin{cases}
-\dv \, a(x,\nabla u) = f &\qquad \hbox{in $\Omega$}\\
u=0 &\qquad \hbox{on $\partial\Omega$\,,}
\end{cases}
\end{equation}
where $\Omega$ is a bounded open set
%{\color{blue} Lipschitz domain}
in $\rn$, $n \geq 2$,  $a: \om \times \rn\to \rn$ is a Carath\'eodory function and the function $f : \Omega \to \mathbb R$ is assigned.
\par
Second-order elliptic equations, in  divergence form, are a very classical theme in the theory of~partial differential equations, and have been extensively investigated in the literature. The punctum of the present contribution is in that, besides the standard monotonicity assumption
\begin{align}\label{A3}
(a(x,\xi) - a(x, \eta)) \cdot (\xi-\eta)> 0\qquad \hbox{for every
$\xi,\eta\in\rn$ such that $\xi\neq\eta$,}
\end{align}
for  a.e. $x\in\Omega$, the function $a$ is subject to very general coercivity and growth conditions, that embrace and considerably extend  customary instances. The leading hypotheses on $a$ amount to requiring that there exists   a (possibly fully anisotropic)  $N$-function
$\Phi:\rn\to\rp$  such that, for  a.e. $x\in\Omega$,
\begin{align}\label{A2'}
a(x,\xi)\cdot \xi \geq \Phi(\xi) \quad  \hbox{for every  $\xi\in\rn$,}
\end{align}
and
\begin{align}\label{A2''}
\widetilde{\Phi}(c_\Phi a(x, \xi))\leq  \Phi(\xi) +h(x) \quad \hbox{for every  $\xi\in\rn$}
\end{align}
% \begin{align}\label{A2''}
% c_\Phi\widetilde{\Phi}(a(x, \xi))\leq  \Phi(\xi) +h(x) \quad \hbox{for every  $\xi\in\rn$,}
% \end{align}
for some positive constant $c_{\Phi}$ and some nonnegative  function $h\in L^1(\Omega)$. Here,  $\widetilde{\Phi}$ denotes the Young conjugate of $\Phi$. Of course, there is no loss of generality in assuming that $c_\Phi  \in (0,1)$. In particular,
% replacing $\widetilde{\Phi}(c_\Phi a(x, \xi))$ by $c_\Phi\widetilde{\Phi}(a(x, \xi))$ results in a condition stronger than \eqref{A2''}.
condition \eqref{A2''} is  fulfilled if $a(x, \xi)$ satisfies the stronger inequality obtained on replacing  the left-hand side of \eqref{A2''}  by  $c_\Phi\widetilde{\Phi}(a(x, \xi))$.
\par
An $N$-function is an even convex function, vanishing at zero, 
decaying faster than linearly near zero and growing faster than linearly
near infinity. Its Young conjugate   is also an $N$-function and comes into play in an H\"older--type inequality for the Orlicz norm defined in terms of $\Phi$.
Precise definitions of $N$-function and Young conjugate   can be found in the next section, where a number of notions and properties concerning the unconventional functional framework associated with our analysis are recalled or proved.
\par
Let us just stress here that  $\Phi (\xi)$ does not have to depend on $\xi$ just through its length $|\xi|$, thus allowing for full anisotropy in the differential operator.  Moreover, in contrast to the assumptions imposed \mbox{on~$p$-Lap}lace-type equations,
$\Phi$ need not have a polynomial growth.
%, as it is instead assumed when dealing with $p$-Laplace type equations.
In fact, $\Phi$ is not even supposed to~fulfill the so-called $\Delta_2$-condition, nor the $\nabla_2$-condition, that are usually required as a~replacement for {homogeneity}  of $\Phi$.  The lack of these conditions on $\Phi$ results in non-reflexivity and non-separability of the Orlicz-Sobolev space $W^1_0L^\Phi(\om)$ built upon $\Phi$, a natural function space associated with problem~\eqref{eq:main:f}. \par
We are concerned with existence, uniqueness and regularity of solutions to the Dirichlet problem~\eqref{eq:main:f}. Our analysis initiates by discussing weak solutions to~\eqref{eq:main:f}, namely solutions $u$ that belong to the Orlicz--Sobolev space $W^1_0L^\Phi(\om)$, or, more precisely, to the corresponding Orlicz--Sobolev class. Due to the generality of the situation under consideration and, specifically, to the anisotropy and non-reflexivity of the involved function spaces, standard methods do not apply. Our approach combines various techniques, including approximation via isotropic operators, comparison with solutions to symmetrized problems, the use of sharp embedding theorems for Orlicz--Sobolev spaces. This enables us to exhibit  an optimal integrability assumption on the datum $f$, depending on the growth of $\Phi$ near infinity, for the existence of a (unique) weak solution to problem~\eqref{eq:main:f}. The relevant optimal assumption on $f$ amounts to its membership in a space of Orlicz--Lorentz type, which arises as an associate space of the optimal rearrangement-invariant target space in an anisotropic Orlicz-Sobolev embedding.
This is the content of Theorem~\ref{theo:boundex}.
Let us emphasize that this result is new even in the isotropic case, that is when $\Phi$ is a radial function.
\par
When $f$ is affected by poor integrability properties, existence of weak solutions to problem~\eqref{eq:main:f} is not guaranteed. This is well known even in the linear situation when the differential operator is the Laplacian. In particular, solutions that do exist in a yet weaker sense -- for instance, merely distributional solutions -- typically do not belong to the pertaining Sobolev space. Also, they need not be unique, as shown in \cite{serrin}.
\par
In this connection, after disposing  the issue of existence of weak solutions,  we drop any extra regularity  on $f$ besides plain integrability in $\om$, and address the question of existence of solutions \mbox{to the Di}richlet problem~\eqref{eq:main:f} in a suitably generalized sense. Our result with this regard is stated in~Theorem~\ref{theo:main-f}. Under the mere assumption that $f\in L^1(\om)$, it asserts the existence and uniqueness of~solutions, called  {\em approximable solutions} throughout, that are limits of~weak solutions to approximating problems with regular right-hand sides. Importantly, Theorem \ref{theo:main-f} also provides us with  maximal regularity of the solution $u$ and of its gradient $\nabla u$.  Such a regularity is properly described  in terms of~Marcinkiewicz--type spaces,  depending on $\Phi$. An anisotropic Orlicz--Sobolev embedding, with optimal Orlicz target space, is critical in dictating the form of these Marcinkiewicz--type spaces.
\par
Our approach to problem \eqref{eq:main:f}  with  right-hand side in $L^1(\om)$ carries over, in fact, to the case when $f$ is replaced by a measure with finite total variation in $\om$. The relevant result is stated in~Theorem~\ref{theo:main-mu}. Let us point out that, though existence and regularity of solutions hold exactly  under the same conditions as for  data in $L^1(\om)$, their  uniqueness is uncertain.  As  far as we know, this is an open problem even in case of standard isotropic nonlinear operators, such as the $p$-Laplacian.
\par
The literature on  elliptic equations, %problems as in \eqref{eq:main:f},
under such a broad  ellipticity condition as that  defined in terms of  $N$-functions $\Phi$, is quite limited -- see e.g.  \cite{A, AdBF1, AdBF3, AC, Ci-local, pgisazg1, gwiazda-ren-ell, pgisazg1}.
Our results answer some questions in their general theory, and  provide a unified framework for  results available for functions $\Phi$ of special forms.
% dealing with generalized solutions, obtained by approximation, for nonlinear operators.
\par
So-called operators  with $p$-growth,  modelled upon the $p$-Laplacian, correspond to the choice
\begin{equation}\label{Phi=p}
\Phi (\xi) = |\xi|^p \qquad \hbox{for $\xi \in \rn$,}
\end{equation}
with $p>1$.
% {\color{red} In this case, assumption \eqref{A2''} agrees with the classical growth condition $|a(x,\xi)|\leq c(|\xi|^{p-1} + g(x))$ for some function $g \in L^{p'}(\Omega)$ and some constant $c$.}
The theory of equations governed by this kind of nonlinearity has been thoroughly developed since the sixties of the last century. The analysis of solutions that are  well suited to allow for right-hand sides in $L^1$ is more recent. Their systematic  study was initiated with the papers
\cite{bgSOLA} and \cite{LionsMurat}. Other contributions in this direction  include \cite{AFT, AM, BBGGPV, DMOP, dall, FS}.
\par
Existence and sharp regularity results for equations with non-polynomial growth and $L^1$ or measure data, but still in the isotropic and reflexive setting where
\begin{equation}\label{Phi=A}
\Phi (\xi) = A(|\xi|) \qquad \hbox{for $\xi \in \rn$}
\end{equation}
for some classical $N$-function of one variable satisfying both the $\Delta _2$ and $\nabla _2$-condition, are presented in~\cite{CiMa}. Previous researches along this direction can be found in \cite{renel3,Dong-Fang}. Results concerning  this kind of ellipticity, but involving more regular operators $a$,  or right-hand sides $f$ enjoying stronger integrability properties, are the subject of \cite{Baroni, BeckMingione, IC-gradest, IC-gradest2, Ci-bound, Donaldson, Gossez2, Gossez, Korolev, Lieberman, Mustonen, Ta79}.
\par
Elliptic  problems with  growth of the form
\begin{equation}\label{Phi=pi}
\Phi (\xi) = \sum _{i=1}^n |\xi_i|^{p_i} \qquad \hbox{for $\xi \in \rn$,}
\end{equation}
where $\xi = (\xi_1, \dots ,{\xi}_n)$, {$1<p_i<\infty,$ $i=1,\dots,n$,} provide a basic framework for physical models in the presence of anisotropy. They are the topic of diverse contributions, including  \cite{CaPa, CKP, ELM, IL, Marc, Stroff, Ve1}. The case of $L^1$ right-hand sides was considered in \cite{BGM} under some restrictions on the exponents $p_i$.  Note that functions as in \eqref{Phi=pi} are  particular examples of those given by
\begin{equation}\label{Phi=Ai}
\Phi (\xi) = \sum _{i=1}^n A_i(|\xi_i|) \qquad \hbox{for $\xi \in \rn$,}
\end{equation}
where $A_i$ are $N$-functions of one variable, which fall within the  frames of the present discussion.
\par
As an application of Theorems \ref{theo:boundex}, \ref{theo:main-f}  and \ref{theo:main-mu}, stated in  Section \ref{main}, optimal results are  offered in  the specific instances mentioned above. However,  let us again emphasize that our discussion covers more general situations than those described above and, importantly, allows for functions $\Phi$ that do not necessarily admit the split form \eqref{Phi=Ai}.   Examples which generalize  one from \cite{Tr} are provided by $N$-functions $\Phi$
of the form
\begin{equation}\label{Phi=full}
\Phi (\xi) = \sum _{k=1}^K A_k\Big(\Big|\sum_{i=1}^n \alpha _i^k\xi_i\Big|\Big) \qquad \hbox{for $\xi \in \rn$,}
\end{equation}
where $A_k$ are $N$-functions of one variable, $K \in \mathbb N$ and the coefficients $\alpha _i^k \in \mathbb R$ are arbitrary. A possible instance, when $n=2$, include the function
%$\Phi: \mathbb R^2 \to [0, \infty)$ given by
\begin{equation}
\label{trud}
\Phi (\xi) = |\xi_1 -\xi_2|^p + |\xi_1|^q\log (c+ |\xi _1|)^\alpha \quad \hbox{for $\xi \in \mathbb R^2$,  }
\end{equation}
where either $q\geq 1$ and $\alpha >0$, or $q=1$ and $\alpha >0$, the exponent  $p>1$, and $c$ is a sufficiently large constant for $\Phi$ to be convex. Another example   amounts to the function
\begin{equation}
\label{trud1}
\Phi (\xi) = |\xi_1 + 3 \xi_2|^p + e^{|2\xi_1-\xi_2|^\beta} -1  \quad \hbox{for $\xi \in \mathbb R^2$,  }
\end{equation}
with $p>1$ and $\beta >1$.
%
% Of course,
%$n$-dimensional versions of \eqref{trud} can be easily exhibited.

\section{Function spaces}\label{back}

\iffalse
{\noindent \color{teal} Notation. Let us denote the characteristic function of the set $E$ by $\chi_E$, while the Lebesgue measure of the unit ball in~$\rn$ is denoted by $\omega_n$.}
\fi

Assume that $\Omega$ is a measurable subset of $\rn$, with $n\geq 1$,
having finite Lebesgue measure $|\om|$. Given $m \in \mathbb N$, we set
$$\mathcal M(\om {{;}} \R^m)=\{U:   \; \hbox{$U$ is a   measurable function from $\om$ into $\mathbb R^m$}\}\,.$$
%$$\mathcal M(\om {{;}} \R^m)=\{U: \om \to \R^m \;\; \hbox{such that $U$ is a measurable function}\}\,.$$
When $m=1$,
we shall make  use of the abridged notation $\mathcal M(\om)$ for $\mathcal M(\om; \R)$. An analogous simplification will be employed in the notation of other function spaces without further mentioning.
\\ Given  $u \in \mathcal M(\om)$, we define the \textit{distribution
function} $\mu_{u} : [0, \infty) \to [0, \infty)$ as
\begin{equation} \label{muu} \mu_{u}(t)=|\{x \in \Omega:
|u(x)|>t\}| \quad \hbox{for $t\geq 0$,}
\end{equation}
and the  \textit{decreasing rearrangement} $u^*: [0, \infty)
\rightarrow [0, \infty]$ as
\begin{equation}\label{decreasing rearrangement}
u^*(s)= \inf\{t \geq 0:\ \ \mu_u(t)   \leq s\}  \qquad\text{for}\ s
\geq 0.
\end{equation}
The function $u^*$ is equimeasurable with $u$ and right-continuous.
The  function $u^{**} :(0, \infty) \to [0, \infty ]$,   called the {\em maximal rearrangement} of $u^*$ and given by
\begin{equation}\label{u^**}
u^{**}(s)\, = \, {\frac 1 s}\, \int_0^{s}u^*(r) dr \, \qquad
\hbox{for} \ s >0 \,,
\end{equation}
is non-increasing, and satisfies  $u^* \leq u^{**}$.
%Moreover,
%\begin{equation}\label{**}
%(u+v)^{\ast \ast}  \, \leq \, u^{\ast \ast}  + v^{\ast \ast}
%\end{equation}
%for every $u, v \in L^0_+ (E)$.
\\ A \textit{Banach function space} $X(\om)$ (in the sense of Luxemburg \cite{BS}) of
functions in $\mathcal{M}(\om)$  is called a~\textit{rearrangement-invariant}
space if its norm $\|\cdot\|_{X(\om)}$ satisfies
\begin{equation}
\label{N6} \|u\|_{{X(\om)}} \, = \, \| v\|_{{X(\om)}} \,\,\,   \
\text{whenever}  \,\,\, u^\ast = v^\ast\,.
\end{equation}
 If  $X(\om)$ is a rearrangement-invariant space,  then
\begin{equation}\label{imm}
L^\infty (\om)     \to  X(\om)    \to L^1(\om),
\end{equation}
where $\to$ stands for a continuous embedding.
\\
Let  $X(\om)$ be a rearrangement-invariant space. Its
\textit{associate space}   is the rearrangement-invariant space
$X'(\om)$ equipped with the norm given by
\begin{equation} \label{n.assoc.}
\|u\|_{X'(\om)}=  \sup \bigg\{\int _{\om} |u(x)v(x)|\, dx:\   \| v
\|_{{X(\om)}} \leq 1 \bigg\}.
\end{equation}
The space $X'(\om)$  is contained in
 the topological dual of $X(\om)$, denoted  by
$X(\om)^*$, but need not coincide with the latter.
\iffalse

\noindent The \textit{fundamental function} $\varphi_X :
[0,|\om|]\to [0, \infty)$ of a rearrangement-invariant space
$X(\om)$ is defined as
\begin{equation}
\label{fundam} \varphi _X(s) = \|\chi_E\|_{X(\om)} \quad \hbox{for
$s \in[0, |\om|]$,}
\end{equation}
where $E$ is any measurable subset of $\om$ such that $|E|=s$. On
replacing $\|\cdot\|_{X(\om)}$, if necessary, by an equivalent norm,
one can assume, without loss of generality, that the fundamental
function $\varphi_X$ is concave {and increasing}.
\\
 The norm  in the \textit{Lorentz space} $\Lambda _X(\om)$  associated with $X(\om)$ is defined by
\begin{equation}\label{Lambda}
\|u\|_{\Lambda _X(\om)} = u^*(0)\varphi_X(0^+) +\int _0^{|\om|}
u^*(s) \varphi_X '(s)\, ds
\end{equation}
for $u\in \mathcal M(\om)$. The space $\Lambda _X(\om)$ is the
smallest rearrangement-invariant space whose fundamental function
agrees (up to equivalence) with $ \varphi_X$, namely,
\begin{equation}\label{Lambdaincl}
\Lambda _X(\om) \to X(\om)
\end{equation}
for every rearrangement-invariant space $X(\om)$.

\fi
\\
Let $\varrho : (0, |\om|) \to (0, \infty)$ be a continuous
increasing function. We denote by $L^{\varrho(\cdot),\infty}(\om)$ the
\textit{Marcinkiewicz--type space  associated with
$\varrho$}, and defined as
%of those measurable functions $u$ in $\om$ such that
$$L^{\varrho(\cdot),\infty}(\om) = \bigg\{u\in \mathcal M(\om):\ \
\hbox{there exists $\lambda >0$ such that}\,\, \sup _{s \in (0,
|\om|)} \tfrac {u^{*}(s)}{\varrho ^{-1}(\lambda/s)}< \infty\bigg\}.$$
%for some  $\lambda >0$.
 Note that $L^{\varrho(\cdot),\infty}(\om)$  is not always a normed space.
%{\color{blue} Need for references!}
Special choices of the function
$\varrho$ recover standard spaces of weak type. For instance, if
$\varrho (s) = s ^q$ for some $q >0$, then $L^{\varrho(\cdot),\infty}(\om)= L^{q, \infty}(\Omega)$, the~customary
weak--$L^q(\Omega)$ space. When $\varrho (s)$ behaves like $s^q (\log
s)^\beta$ near infinity for some $q >0$ and $\beta \in \mathbb R$,
we shall adopt the notation $L^{q, \infty}(\log L)^\beta (\om)$ for
$L^{\varrho(\cdot),\infty}(\om)$. The meaning of the notation $L^{q ,
\infty}(\log L)^\beta (\log \log L)^{-1}(\om)$ is analogous.
\\
Orlicz and Orlicz--Lorenz spaces generalize Lebesgue and Lorentz spaces,
respectively, and are classical instances of rearrangement-invariant
spaces. Together with their anisotropic counterpart and with the
associated Sobolev type spaces, they play a critical role in our
discussion. Their definitions and basic properties are recalled in
what follows.

\subsection{Orlicz, Orlicz-Lorentz and Orlicz-Sobolev spaces}

We say that a function $A: [0, \infty) \to [0, \infty]$ is a {\em
Young function} if it is convex, vanishes at $0$, and is neither
identically equal to $0$, nor to~infinity. A Young function $A$
which is finite-valued, vanishes only at $0$ and satisfies the
additional growth conditions
\begin{equation}\label{limA}\lim _{t \to
0}\frac{A (t)}{t}=0 \qquad \hbox{and} \qquad \lim _{t \to \infty
}\frac{A (t)}{t}=\infty \,,
\end{equation}
is called an {\em $N$-function}.
\\
The {\em Young conjugate} of a Young function $A$ is the Young
function $\widetilde{A}$ defined by
$$\widetilde{A}(t)=\sup\{s t - A(s): s\geq 0\} \qquad  \hbox{for $t\geq 0$}\,.$$
Hence,
\begin{equation}
\label{young}
st \leq A(s) + \widetilde A(t) \qquad \hbox{for $s,t \geq 0$.}
\end{equation}
Note that $\widetilde{({\widetilde A})}   = A$ for any Young
function $A$. 
The class of $N$-functions is closed under the operation of Young conjugation. One has that
\begin{equation}\label{AAtilde}
t\leq  \widetilde{A}^{-1} (t) A^{-1}(t) \leq 2t \qquad \hbox{for
$t\geq 0$}\,,
\end{equation}
where $A^{-1}$ stands for the (generalized) left-continuous inverse
of $A$. Hence,
\begin{equation}\label{AAtilde'}
\frac{A(t)}t\leq \widetilde{A}^{-1} (A(t))\leq 2 \, \frac{A(t)}t\qquad
\hbox{for $t\geq 0$}\,.
\end{equation}
%For any Young function, the following inequality holds\footnote{IC: it's direct from convexity, it does not have label and we don't refer to it - do we need this line?}
%$$\lambda A(t)\leq A(\lambda t) \qquad \hbox{for $ \lambda \geq 1$
%and $t\geq 0$}\,.$$
 A Young function $A$ fulfils the \emph{$\Delta_2$-condition near infinity}  if $A$ is finite--valued and there exist constants $c>0$ and $t_0\geq 0$ such that $A(2t)\leq cA(t)$ for $t\geq t_0$.
% \begin{equation}\label{Delta2}
% A(2t)\leq CA(t) \qquad \hbox{for $t\geq t_0$}\,.
% \end{equation}
\\
A function $A$  is said to satisfy the  \emph{$\nabla_2$-condition near infinity} if
 there exist  constants $c>2$ and $t_0\geq 0$ such that $A(2t)\geq cA(t) $ for $t\geq t_0$.
%\begin{equation}\label{nabla2}
% A(2t)\geq CA(t) \qquad \hbox{for  $t\geq t_0$}\,.
% \end{equation}
 \\ We shall also write \lq\lq{}$A\in \Delta_2$ near infinity\rq\rq{}  and \lq\lq{}$A\in \nabla_2$ near infinity\rq\rq{} to denote these properties.\\
%
%to denote that $A$ satisfies~\eqref{Delta2}, and use the expression \lq\lq $A\in \nabla_2$ near infinity" with an analogous meaning.
%{\color{magenta} when  $A$ fulfills \eqref{nabla2}}.
One has that $A\in \Delta_2$ near infinity if and only if  $\widetilde{A}\in \nabla_2$ near infinity.
\\
 We say that a Young function $A$\emph{ dominates} another Young function
 $B$ near infinity, if there exist constants $c>0$ and $t_0\geq 0$ such
 that $B(t)\leq A(ct)$  if $t\geq t_0$.
% \begin{equation}\label{dominate}
% B(t)\leq A(ct)\qquad \hbox{if $t\geq t_0$}\,.
% \end{equation}
%We say that \emph{$A$ dominates $B$ globally} if this inequality %\eqref{dominate}
%holds with $t_0=0$.
 If two Young functions $A$ and $B$ dominate
each other near infinity, then we say that they are
\emph{equivalent near infinity}.
\\
A Young function $A$ is said to \emph{increase essentially faster}
than $B$ near infinity,  if
\begin{equation}\label{ess_slow2}
\lim_{t\to +\infty} \frac{A^{-1}(t)}{B^{-1}(t)}= 0 \,.
\end{equation}
%If $A$ and $B$ are finite-valued, condition \eqref{ess_slow2} is equivalent to
%\begin{equation}\label{ess_slow1}
%\lim_{t\to +\infty} \frac{B(\lambda t)}{A(t)}= 0 \qquad \hbox{for
%every $\lambda>0$}\,.
%\end{equation}

\smallskip
\par

Let $\Omega$ be a measurable set  in $\rn$, $n\geq 1$, with $|\om|< \infty$, and let $A$ be a Young function.  The {\em Orlicz class} $\mathcal L ^A(\om)$ is defined as
\begin{equation}\label{Orlicz_class}
\mathcal L ^A(\om) = \bigg\{u \in \mathcal M(\om): \int_{\Omega } A\left
(|u|\right)\; dx< \infty\bigg\}.
\end{equation}
The set $\mathcal L ^A(\om)$ is convex, but it is not a linear space in general.
The {\emph{Orlicz space}} $L^A(\Omega)$  is the set of~all functions $u\in
\mathcal M(\om)$ such that the Luxemburg norm
\begin{equation}\label{Orlicz_norm}
\|u\|_{L^A(\Omega)} = \inf\bigg\{\lambda>0: \int_{\Omega } A\left
(\tfrac{1}{\lambda}|u|\right)\; dx \leq 1\bigg\}
\end{equation}
is finite. The space $L^A(\Omega)$ equipped with this norm is a
Banach space. It is the smallest vector space containing $\mathcal L^A(\om)$. In particular, one has that $L^A(\Omega) =
L^p(\Omega)$ if $A(t) =t^p$ for some $p\in[1, \infty)$, and
$L^A(\Omega) = L^\infty(\Omega)$ if $A(t) = \infty\chi_{(1,
\infty)}(t)$. Here, and in what follows, $\chi_E$ stands for the characteristic function of a set $E$.
\\
A H\"{o}lder--type inequality in the Orlicz setting reads
\begin{equation}\label{Holder_ineq_A}
\int_{\Omega} |uv|\; dx \leq 2 \|u\|_{L^A(\Omega)}
\|v\|_{L^{\widetilde{A}}(\Omega)}
\end{equation}
for every $u\in L^A(\Omega)$ and $v\in L^{\widetilde{A}}(\Omega)$.
\\
 Let $A$ and $B$ be two Young functions. Then
\begin{equation}\label{embedd}
L^A(\Omega)\rightarrow L^B(\Omega)\qquad \hbox{if and only
if}\qquad\hbox{$A$ dominates $B$ near infinity}\,.
\end{equation}
Here, the arrow  \lq\lq{}$\rightarrow$\rq\rq{} stands for continuous embedding.
In particular, $L^A(\Omega)\rightarrow L^1(\Omega)$ for any Young function $A$.
%\begin{equation}\label{embedd_L1}
%L^A(\Omega)\rightarrow L^1(\Omega)
%\end{equation}
%for any Young function $A$.
Hence,
\begin{equation}\label{embedd1}
L^A(\Omega)=L^B(\Omega)\qquad \hbox{if and only
if}\qquad\hbox{$A$ is equivalent to $B$ near infinity,}
\end{equation}
where the equality has to be interpreted  up to equivalent norms.
\\
 Let us next set
\begin{equation}\label{Eiso}E^A(\om) = \bigg\{u\in {\mathcal M} (\om):\ \  \int_{\Omega } A\left
(\tfrac{1}{\lambda}|u|\right)\; dx < \infty \,\,\hbox{for every $\lambda >0$}\bigg\}\,.
\end{equation}
The space
$E^A(\Omega)$ agrees with the closure in $L^A(\Omega)$, in the norm topology, of the
space of functions which are bounded in $\Omega$ and
have bounded  support. Trivially,
\begin{equation}\label{A4}
E^A(\Omega)\subset\mathcal{L}^A(\Omega)
\subset L^A(\Omega)\,.
\end{equation}
Both inclusions hold as equalities in \eqref{A4} if and only if $A$
satisfies the $\Delta_2$-condition near infinity.
\\
If  $A$ increases essentially faster than $B$  near infinity, then
\begin{equation}\label{emb_AB}
L^A(\Omega)\rightarrow E^B(\Omega)\,.
\end{equation}
The alternative notation $A(L)(\om)$ will also be employed, when convenient, to denote the  Orlicz space associated with any Young function equivalent to $A$ near infinity. For instance, if $\alpha >0$, then  $\exp L^\alpha  (\om)$ stands for the Orlicz space built upon a~Young function equivalent to $e^{t^\alpha}$ near infinity. Moreover, if~either $p >1$ and $\alpha \in \mathbb R$, or $p=1$ and $\alpha \geq 0$, then the space $L^p \log^\alpha L (\om)$ denotes the Orlicz space associated with a~Young function equivalent to $t^p \log^\alpha t$ near infinity.
\\
Given a Young function $A$ and  {\color{blue} $r\in (-\infty, \infty]\setminus \{0\}$}, we denote
by  $L[A,r](\om)$  the Orlicz-Lorentz-type space of those {functions}
 $u \in \mathcal M(\om)$ such that the quantity
\begin{equation}\label{dualoptnorm}
\|u\|_{L[A,r](\om)}= \|s^{\frac{1}{r}}u^{**}(s)\|_{L^A(0, |\om|)}
\end{equation}
is finite. Here, and in what follows, we use the convention that $\tfrac 1\infty=0$. The space $L[A,r](\om)$ is  a rearrangement-invariant
space. It is non-trivial, namely it contains functions that do
not vanish identically, if $ \|s^{\frac{1}{r}}\|_{L^A(0, |\om|)}<
\infty$. In analogy with $E^A(\om)$, we define
\begin{equation}\label{E}
E[A,r](\om) = \bigg\{u\in \mathcal M(\om):
\int_0^{|\om|}A\Big(\tfrac{1}{\lambda}{s^{\frac{1}{r}}u^{**}(s)}
\Big)\, ds <\infty \;\; \hbox{for every $\lambda >0$}\bigg\}.
\end{equation}
Similarly,  we denote by  $L(A,r)(\om)$  the {Orlicz-Lorentz-type
space} of {all functions} $u \in \mathcal M(\om)$ for which the expression
\begin{equation}\label{optnorm}
\|u\|_{L(A,r)(\om)}= \|s^{\frac{1}{r}}u^*(s)\|_{L^A(0, |\om|)}
\end{equation}
is finite. The space $E(A,r)(\om)$ is defined accordingly.
  Under suitable assumptions on $A$ and $r$ the functional defined by
\eqref{optnorm} is a norm, and, consequently, $L(A,r)(\om)$ is a
rearrangement-invariant space equipped with this norm --
 see \cite{cianchi_ibero}. {In particular, for any Young function $A$, formula
\eqref{optnorm} defines a norm provided that  $r<-1$.
\\ The special instance corresponding to  $L^A(0, |\om|)= L^q (0, |\om|)$
%with $q \in [1, \infty]$, and $\tfrac 1r = \tfrac 1p  - \tfrac 1q$ for some $q \in[1, \infty)$,
yields
\begin{equation}\label{lorentz}
%L(A,r)(\om) = E(A,r)(\om)= L^{p,q}(\om) \qquad \hbox{and} \qquad
 L(A,r)(\om)=  E(A,r)(\om)=L^{p,q}(\om)\,,
\end{equation}
up to equivalent norms, provided that $p$ and $r$ are properly chosen. Here, $L^{p,q}(\om)$ denotes the customary Lorentz space of those functions $u \in \mathcal M (\om)$ making the quantity
\begin{equation}\label{lorentznorm}
\|u\|_{{L^{p,q}}(\om)}= \big\|s^{\frac 1p- \frac 1q}u^*(s)\big\|_{L^q(0, |\om|)}
%
%
%=\bigg( \int _0^{|\om|}\big(s^{\frac 1\gamma} u^*(s)\big)^{\sigma}\, \frac {ds}s\bigg)^{\frac 1\sigma}
\end{equation}
 finite.
Also, with a proper choice of $p$ and $r$,
\begin{equation}\label{lorentz**}
%L(A,r)(\om) = E(A,r)(\om)= L^{p,q}(\om) \qquad \hbox{and} \qquad
 L[A,r](\om)=  E[A,r](\om)=L^{[p,q]}(\om)\,,
\end{equation}
up to equivalent norms, where $L^{[p,q]}(\om)$ denotes the   Lorentz space equipped with the norm given by
\begin{equation}\label{lorentznorm**}
\|u\|_{L^{[p,q]}(\om)}=\big\|s^{\frac 1p- \frac 1q}u^{**}(s)\big\|_{L^q (0, |\om|)}
%
%
%=\bigg( \int _0^{|\om|}\big(s^{\frac 1\gamma} u^{**}(s)\big)^{\sigma}\, \frac {ds}s\bigg)^{\frac 1\sigma}\
\end{equation}
 for $u \in \mathcal M (\om)$.
}
\\
When $L^A(0, |\om|)= L^q \log^\alpha L(0, |\om|)$,
%$A(t) \approx t^\sigma (\log t)^\alpha$  near infinity,
where  either $q \in (1, \infty]$ and $\alpha \geq 0$, or $q=1$ and $\alpha \geq 0$,
%either $\alpha \in \mathbb R$ or $\alpha \geq 0$, according \todo{IC: I don't get; do we want to say either ($q>1$\&$\alpha \in \mathbb R$) or  ($q=1$\&$\alpha \geq 0$)??} the whether $q>1$ or $q=1$,
%or $\sigma=1$ and $\alpha >0$,
%and $\tfrac 1r = \tfrac 1p - \tfrac 1q$ for some $q \in [1, \infty)$,
one has that
\begin{equation}\label{lorentzzug}
 L[A,r](\om)=  E[A,r](\om)=L^{[p,q]}(\log L)^\alpha (\om)\,,
\end{equation}
up to equivalent norms, again with a suitable choice of $p$ and $r$ -- see e.g.  \cite[Lemma~6.12, Chapter~4]{BS}. Here, $L^{[p,q]}(\log L)^\alpha (\om)$ denotes the
Lorentz-Zygmund space equipped with the norm defined as
%\todo[inline]{parameters to be checked/changed}
\begin{equation}\label{LZ}
\|u\|_{L^{[p,q]}(\log L)^\alpha(\om)} = \big\|s^{\frac 1p- \frac 1q}\log^{\frac \alpha q} \big(1+\tfrac{|\om|}{s}\big)u^{**}(s)\big\|_{L^q (0, |\om|)}
%
%
%=\bigg( \int _0^{|\om|}\big(s^{\frac 1\gamma} u^{**}(s)\big)^{\sigma}\log^\alpha \Big(1 +\frac {|\Omega|}s\Big)\, \frac {ds}s\bigg)^{\frac 1\sigma}
\end{equation}
for $u \in \mathcal M (\om)$. An analogous relation links the spaces
 $L(A,r)(\om)$, $E(A,r)(\om)$ and $L^{p,q}(\log L)^\alpha (\om)$, where the latter is defined as the set of all functions $u \in \mathcal M (\om)$ which render the right-hand side of~equation \eqref{LZ}, with $u^{**}$ replaced by $u^*$, finite.
%
%\todo[inline]{Include here all definitions of norms}
%\smallskip

\smallskip
\par
Assume now that $\Omega$ is an open set in $\rn$,  $n\geq 2$, with $|\om|<\infty$.
%The\textit{ Orlicz-Sobolev} space $W^1L^A(\Omega)$ is defined as
%\begin{equation}\label{def_W^1L^A}
%W^1L^A(\Omega) =\left\{ u\in L^A(\Omega):  \nabla u\in L^A(\Omega;
%\rn)\right\}\,.
%\end{equation}
%The space $W^1E^A(\Omega)$ is defined  analogously. Both
%$W^1L^A(\Omega)$ and $W^1E^A(\Omega)$ can be identified as subspace
%of the product of $L^A\times L^A$ and are Banach spaces endowed with
%the norm
%\begin{equation}\label{sobolev-norm}
%\|u\|_{W^1L^A(\Omega)} = \|u\|_{L^A(\Omega)} + \| \nabla
%u\|_{L^A(\Omega; \rn)}\,.
%\end{equation}
 We define the {\em Orlicz--Sobolev  class}
\begin{align}\label{defWclass}
{W}^1_0 \mathcal L^A (\Omega) = \{  u \in \mathcal M(\Omega) :\   &\, \hbox{ the
continuation of $u$ by $0$ outside  $\Omega$ is} \\ \nonumber &
\hbox{ weakly differentiable and $|\nabla u|\in\mathcal L^A(\Omega)$}\}.
 \end{align}
The {\em Orlicz--Sobolev space} ${W}^1_0 L^A (\Omega)$  is defined analogously, on replacing $\mathcal L^A(\Omega)$ by $L^A(\Omega)$ on the right-hand side of definition \eqref{defWclass}.
%
%
%\begin{align}\label{defW0}
%{W}^1_0 L^A (\Omega) = \{  u \in \mathcal M(\Omega) :  &\, \hbox{the
%continuation of $u$ by $0$ outside  $\Omega$ is} \cr &
%\hbox{weakly differentiable and $|\nabla u|\in L^A(\Omega)$}\}\,.
% \end{align}
 The space $W^1_0 L^A (\Omega)$, endowed with the norm
\begin{equation}\label{sobolev_zero-norm}
\|u\|_{W^1_0L^A(\Omega)} = \|\,|\nabla u|\,\|_{L^A(\Omega)}\,,
\end{equation}
is a Banach space.  %Here, and in similar situations in what follows, the notation $ \| \nabla u\|_{L^A(\Omega)}$ is an abbreviation for $ \|\, | \nabla u|\,\|_{L^A(\Omega)}$.\\
 Note that, thanks to a Poincar\'{e}-type inequality -- see \cite[Lemma 3]{Ta90} --
the norm defined by~\eqref{sobolev_zero-norm} is equivalent to the norm given by $\|u\|_{L^A(\Omega)} + \| \,|\nabla u|\,\|_{L^A(\Omega)}$.
%\begin{equation}\label{sobolev-norm}
%\|u\|_{W^1L^A(\Omega)} = \|u\|_{L^A(\Omega)} + \| \nabla
%u\|_{L^A(\Omega; \rn)}\,,
%\end{equation}
%where $W^1L^A(\Omega)$  stands for the space of all functions
%$u\in L^A(\om)$ such that  $\nabla u\in L^A(\om; \rn)$.
\\
 In the case when $L^A(\om) =L^p(\om)$ for some $p \in [1, \infty)$ and $\partial \Omega$ is regular enough, the above
  definition  of~${W}^1_0 L^A (\Omega)$ reproduces the usual space
  $W^{1,p}_0 (\Omega)$ defined as the closure in $W^{1,p}(\Omega)$ of the space
  $ {C }_0^\infty (\Omega)$ of~smooth compactly supported functions in $\Omega$.
  On the other hand, the set of~smooth bounded functions is dense in
$L^A(\Omega)$  if and only if $A$ satisfies the $\Delta_2$-condition,
%(just near
%infinity when $|\Omega|<\infty$),
and hence, for arbitrary $A$, our
definition of~${W}^1_0 L^A (\Omega)$ yields a space which can be
larger than the closure of $ {C }_0^\infty (\Omega)$ with respect to  the norm in~\eqref{sobolev_zero-norm}. A systematic study of Orlicz-Sobolev spaces was initiated in \cite{DoTr}. An account of~more recent developments can be found in \cite{rao-ren, rao-ren'}.

\subsection{Anisotropic Orlicz and Orlicz--Sobolev spaces}

 A function $\Phi:\rn\to [0,\infty]$ is called an \emph{$n$-dimensional Young function}
 if it is convex, $\Phi(0) = 0$, $\Phi(\xi)=\Phi(-\xi)$ for
$\xi\in \rn$, and  $\{\xi \in \rn: \Phi(\xi) \leq t\}$ is a compact set containing $0$ in its interior for every $t>0$.
\\ The function $\Phi$ is called an
\emph{$n$-dimensional $N$-function} if, in addition, $\Phi$ is finite--valued,
vanishes only at $0$, and
\begin{equation}\label{lim}\lim _{\xi \to
0}\frac{\Phi (\xi)}{|\xi|}=0 \qquad \hbox{and} \qquad \lim _{|\xi| \to
\infty }\frac{\Phi (\xi)}{|\xi|}=\infty\,.
\end{equation}
Notice that, for technical reasons and  ease of presentation, in the case when $n=1$ we are distinguishing  Young functions or $N$-functions, as defined on $[0, \infty)$ as in the previous subsection, from $1$-dimensional Young functions or $1$-dimensional $N$-functions defined on the whole of $\mathbb R$ here. However, extending a
Young function to an~even function on the entire $\mathbb R$
results in a $1$-dimensional Young function; conversely, the
restriction of a $1$-dimensional Young function to $[0, \infty)$ is
a Young function. Thus, any definition or result concerning Young
functions or $N$-functions translates into a corresponding definition or result for
$1$-dimensional Young functions or $N$-functions, and viceversa.
\\  In what follows, Young or $N$-functions will be denoted by latin capital letters, whereas $n$-dimensional Young or $N$-functions will be denoted by greek capital letters. Thus, there will be no ambiguity if we simply write  Young function or $N$-function when referring to an $n$-dimensional function.
%
%
%{Note that an $N$-function $A:[0,\infty)\to[0,\infty]$   can be
%extended into an even function to the whole $\R$ producing $1$-dimensional $N-$function.
%Conversely, the restriction of a $1$-dimensional $N-$function is an $N$-function. }\\
%{\color{blue} When $A$ be a finite valued Young function fulfilling \eqref{limA} and vanishing only in the origin, then function $\rn\ni\xi\mapsto A(|\xi|)$ is an $n$-dimensional $N$-function. On the other hand, if $\Phi$ is an $n$-dimensional $N$-function and $\xi\in \rn$, then function $[0, \infty)\ni t\mapsto \Phi(t|\xi|)$ represents an $N$-function.}
\\
\emph{The Young conjugate} of a Young function $\Phi$ is the  Young function $\wt \Phi$
defined as
$$\wt \Phi (\xi) = \sup\{\eta \cdot \xi - \Phi(\eta):\ \eta \in \rn\}\quad \hbox{for $\xi \in \rn$\,.}
$$
Here, the dot $\lq\lq \cdot "$ denotes scalar product in $\rn$.
% Clearly,
%\begin{equation}\label{Young-Phi-inq}
%\xi\eta\leq  \Phi(\xi)+\widetilde{\Phi}(\eta) \qquad \hbox{for
%$\xi,\eta\in\rn$}\,.
%\end{equation}
One has  that $\widetilde{({\widetilde \Phi})}   = \Phi$ for any
   Young function
$\Phi$. The class of $N$-functions is closed under the operation of Young conjugation. \\
 A  Young function $\Phi$ is said to satisfy the \emph{$\Delta _2$-condition near infinity}, briefly $\Phi \in \Delta _2$ near infinity,  if it is finite--valued and there
exist positive constants $c$ and $M$ such that $\Phi (2\xi) \leq c \Phi(\xi)$   if $|\xi|\geq M$.
%\begin{equation}\label{delta2-a}
%\Phi (2\xi) \leq c \Phi(\xi) \quad \hbox{if $|\xi|\geq M$\,.}
%\end{equation}
\\ A Young function $\Phi$ is said to satisfy the \emph{$\nabla
_2$-condition near infinity}, briefly $\Phi \in \nabla _2$ near
infinity,  if~there exist  constants $c>2$ and $M>0$ such that $\Phi (2\xi) \geq c \Phi(\xi)$   if $|\xi|\geq M$.
%\begin{equation}\label{nabla2-a}
%\Phi (2\xi) \geq c \Phi(\xi) \quad \hbox{if $|\xi|\geq M$\,.}
%\end{equation}
\\
The  Young function $\Phi$ is said to \emph{dominate} another
Young function $\Psi$ \emph{near infinity} if there exist  positive constants
$c$ and $M$ such that $\Psi (\xi) \leq \Phi (c\xi)$  if $|\xi| \geq M$.
%\begin{equation}\label{dominate-a}
%\Phi (\xi) \leq \Psi (c\xi) \quad \hbox{if $|\xi| \geq M$\,.}
%\end{equation}
%We say that  \emph{$\Phi$  dominates $\Psi$ globally} if this  inequality
% holds for every $\xi \in \rn$.
 Equivalence of  Young functions is defined accordingly.
\par
%%%%%%%%%%%%%%%%%%%%%%%%%%%%%%%%%%%%%%% anisotropic Orlicz spaces %%%%%%%%%%%%%%%%%%%%%%%%%%%%%%%%%%%
Let $\Omega$ be a measurable set in $\rn$,  $n\geq 1$, with $|\om|<\infty$, and let $\Phi$ be an $n$-dimensional Young function.  The  \textit{anisotropic Orlicz space}
$L^\Phi(\om;\rn)$ is the set of all vector-valued functions $U\in \mathcal M(\om; \rn)$ such that
the norm
$$\|U\|_{L^\Phi (\om;\rn)} = \inf \bigg\{\lambda >0: \int _\om\Phi\big(\tfrac 1\lambda U \big)\, dx \leq 1\bigg\}$$
is finite. The space  $L^\Phi (\om;\rn)$, equipped with this norm,
is a Banach space. The   Orlicz class $\mathcal L^\Phi(\om, \rn)$ and the space $E^\Phi(\Omega;\rn)$ are defined in analogy with definitions \eqref{Orlicz_class} and  \eqref{Eiso}, respectively. One has that
 $E^\Phi(\Omega;\rn)$ agrees with the
closure in $L^\Phi(\Omega;\rn)$ of the space of bounded functions in $\om$ with bounded support.
Clearly,
\begin{equation} \label{EPP}
E^\Phi(\Omega;\rn) \subset  \mathcal L^\Phi(\Omega;\rn) \subset L^\Phi(\Omega;\rn),
\end{equation}
and both inclusions hold as equalities if
and only if $ \Phi\in\Delta_2$ near infinity.
%
%Moreover,
%\begin{equation} \label{EPhiLPhi}
%L^\Phi(\Omega;\rn) = \mathcal L^\Phi(\Omega;\rn)=E^\Phi(\Omega;\rn)\qquad\text{if
%and only if}\qquad \Phi\in\Delta_2\ \text{near infinity.}
%\end{equation}
The H\"older-type inequality
\begin{equation}\label{holder}
\int _\om |U \cdot   V|\,dx \leq 2\|U\|_{L^\Phi (\om;\rn)}
\|V\|_{L^{\widetilde\Phi} (\om;\rn)}
\end{equation}
holds for every $U \in L^\Phi (\om;\rn)$ and $V \in
L^{\widetilde\Phi} (\om;\rn)$.
\\
If $\Phi$ and $\Psi$ are  Young functions,   then
\begin{equation*}
   L^\Phi (\om;\rn)\to L^\Psi (\om;\rn) \quad \hbox{if and only if} \quad \hbox{$\Phi$
dominates $\Psi$ near infinity.}
\end{equation*}
 In
particular, $ L^\Phi (\om;\rn)\to L^1 (\om;\rn)$
%
%for any measurable set $G\subset \mathbb R^N$, with finite
%measure. In particular
%\begin{equation}\label{emb}
%   L^\Phi (\om;\rn)\to L^1 (\om;\rn)
%\end{equation}
for any   Young function $\Phi$. Moreover,
\begin{equation*}
   L^\Phi (\om;\rn)=L^\Psi (\om;\rn) \quad \hbox{if and only if} \quad \hbox{$\Phi$ and $\Psi$ are equivalent near infinity.}
\end{equation*}
%up to equivalent norms, if and only if $\Phi$ and $\Psi$ are equivalent near infinity.
By  \cite[Corollary 7.2]{Sch}, given any $N$-function $\Phi$, the space $L^\Phi (\om; \rn)$ is reflexive if and only  if $\Phi \in \Delta _2 \cap \nabla _2$ near
infinity.
% (or as a
%consequence of duality and \eqref{EPhiLPhi}),
%\begin{equation}\label{reflex}
% \hbox{  $L^\Phi (\om; \rn)$ is reflexive if and only  if $\Phi \in \Delta _2 \cap \nabla _2$ near
%infinity.} \end{equation}
In general,  if $\Phi$ is an arbitrary $n$-dimensional  $N$-function, then
\begin{equation}\label{Dual}
\hbox{the dual
of $E^{\Phi}(\Omega;\rn)$ is isomorphic and homeomorphic to
$L^{\widetilde{\Phi}}(\Omega;\rn)$,}
\end{equation}
%, and the duality pairing is given
%by%
%$$
%< V,U > = {\displaystyle\int_{\Omega}} V(x)\cdot U(x)dx$$
%for
% $V\in L^{\widetilde{\Phi}}(\Omega;\rn)$ and $U\in E^{\Phi}
%(\Omega;\rn)$
see \cite[Proposition 2.3]{AdBF1}.
%
%\par
%In the next proposition, we announce  the dual space of
%$E^{\Phi}(\Omega;\rn)$ when $\Phi$ does not fulfil either
%$\Delta_2$-condition nor $\nabla_2$-condition near infinity (see
%{\color{magenta}\cite[Proposition 2.3]{AdBF1}}).
%
%\begin{proposition}[Duality $(E^\Phi)^*=L^{\wt{\Phi}}$]\label{thm_dual}
%Suppose that $\Phi $ is an $n$-dimensional $N$-function. The dual space
%of $E^{\Phi}(\Omega;\rn)$ is isomorphic and homeomorphic to
%$L^{\widetilde{\Phi}}(\Omega;\rn)$ and the duality pairing is given
%by%
%\[
%< V,U > = {\displaystyle\int_{\Omega}} V(x)\cdot U(x)dx\qquad\text{for }
% \ V\in L^{\widetilde{\Phi}}(\Omega;\rn)\ \text{ and }\ U\in E^{\Phi}%
%(\Omega;\rn).\]
%\end{proposition}
Orlicz spaces of vector-valued functions are studied in detail in \cite{Sk1, Sk2}, as special cases of more general Musielak-Orlicz spaces; the analysis of the paper  \cite{Sch} also includes Orlicz spaces of functions defined on infinite dimensional spaces.
\smallskip
\par
Assume that $\Omega$ is an open set in $\rn$,  $n\geq 2$, with $|\om|<\infty$. Let $\Phi$ be  an $n$-dimensional  Young function. The \emph{anisotropic Orlicz-Sobolev class} is
defined as
\begin{align}\label{aniso_Wclass}
{W}^{1}_0\mathcal L^\Phi(\Omega)= \{u\in{\mathcal M}(\Omega)\,:\ & \hbox{ the continuation
of $u$ by $0$ outside $\Omega$} \\ \nonumber  & \hbox{ is weakly differentiable
in $\rn$ and $\nabla u \in \mathcal L^\Phi (\Omega{; \rn})$}\}.
\end{align}
The \emph{anisotropic Orlicz-Sobolev space} ${W}^{1}_0L^\Phi(\Omega)$ is defined accordingly,
on replacing $\mathcal L^\Phi (\Omega{; \rn})$ by~$ L^\Phi (\Omega{; \rn})$ on the right-hand side of equation \eqref{aniso_Wclass}.
%
%\begin{align*}{W}^{1}_0L^\Phi(\Omega)= \{u: \Omega \to \R &\,\, \hbox{such that the continuation
%of $u$ by $0$ outside $\Omega$} \\ & \hbox{ is weakly differentiable
%in $\rn$, and $\nabla u \in L^\Phi (\Omega{; \rn})$}\},
%\end{align*}
%respectively.
One has that $W^{1}_0L^\Phi(\Omega)$, equipped with the norm
$$\|u\|_{{W}^{1}_0L^\Phi(\Omega)} = \|\nabla u\|_{L^\Phi (\Omega{; \rn})},$$ is a Banach space.
%Let $\Omega$ be an open set in $\rn$ such that $|\Omega |< \infty$.
The Orlicz-Sobolev space ${W}^{1}_0L^\Phi(\Omega)$ is reflexive if and only if $\Phi\in \Delta _2 \cap \nabla _2$ near infinity.  Classical contributions on Orlicz-Sobolev spaces  are \cite{Tr} and \cite{Klimov76}.
\\
The use of sets of
functions, whose truncations belong to an Orlicz-Sobolev space, is crucial in dealing with approximable solutions.  Given any
$t > 0$, let $T_t : \R\rightarrow \R$ denote the function defined by
\begin{equation}\label{trunc}
 T_t(s) = \begin{cases} s & \quad \hbox{if
$\ |s|\leq t$\,, }
\\
t\,{\rm sign}\, (s) & \quad \hbox{if $\ |s|> t$\,.}
\end{cases}
\end{equation}
We set \begin{equation}\label{def_Trun_Space} \mathcal{T}_0^{1,
\Phi}(\Omega) = \{ u \in \mathcal M(\om) :\  T_t(u)\in
{W_0^1L^\Phi}(\Omega) \;\; \hbox{for every $t>0$}\}\,.
\end{equation}
The space $\mathcal{T}_0^{1, \Phi}(\Omega)$ is the anisotropic counterpart of   the space introduced in \cite{BBGGPV} and associated with the standard Sobolev space $W^{1,p}_0(\om)$ corresponding to the choice $\Phi(\xi)=|\xi|^p$.
%\footnote{IC: ???}
\\
 If $u\in\mathcal{T}_0^{1,
\Phi}(\Omega)$, then there exists a (unique) measurable function
$Z_u : \Omega \to \rn$ such that
\begin{equation}\label{gengrad} \nabla T_t(u) = \chi_{\{|u|<t\}} Z_u\qquad \hbox{
a.e. in $\Omega$}
\end{equation}
for every $t > 0$. This is a consequence of \cite[Lemma~2.1]{BBGGPV}. One has that $u\in W_0^1L^\Phi(\Omega)$ if and only
if $u\in\mathcal{T}_0^{1, \Phi}(\Omega)$
% \todo{check if $L^1$ is needed}
and
$Z_u\in L^\Phi(\Omega; \rn)$.
In the latter case, $Z_u = \nabla u$ a.e.
in $\Omega$. With an abuse of~notation, for every $u\in
\mathcal{T}_0^{1, \Phi}(\Omega)$ we denote $Z_u$ simply by $\nabla
u$ throughout.

\subsection{Auxiliary  functions associated with $\Phi$}
%{\color{blue} Changed order!}

Let $\Phi$ be an $n$-dimensional  Young function.
By
 $\Phi_\circ : [0, \infty ) \to [0, \infty)$ we denote the Young function obeying
\begin{equation}\label{phistar}
|\{\xi \in \rn: \Phi_\circ (|\xi|) \leq t\}|  =|\{\xi \in \rn: \Phi
(\xi)\leq t\}|  \quad \hbox{for $t\geq 0$.}
\end{equation}
The function $\rn \ni \xi \mapsto \Phi _\circ(|\xi|)$ can be regarded as a kind of  \lq\lq average in measure\rq\rq  \, of $\Phi$. It
 can be used to define  the radially increasing symmetral $\Phi _\bigstar : \rn
\to [0, \infty)$ of $\Phi$ by
$$\Phi _\bigstar (\xi) = \Phi_\circ (|\xi|) \qquad \hbox{for $\xi \in \rn$.}$$
%One has that $\Phi _\bigstar$ is an $n$-dimensional Young function, then $\Phi_\circ$ is an  Young function and $\Phi _\bigstar$ is an $n$-dimensional Young function. %it was repeating with info below (2.42)
Since $\Phi_\bigstar$ is radially symmetric, the function $\Phi_\Diamond : [0,\infty)\to [0,\infty)$, defined by
\begin{equation} \label{phidiamond}
\Phi_\Diamond (|\xi|) =
\widetilde{\big({\widetilde\Phi}_\bigstar\big)}(\xi) \quad \hbox{for $\xi \in \rn$,}
\end{equation}
is a Young function. Moreover, the function $\Phi_\Diamond$  is equivalent to $\Phi_\circ$, and  there exist constants $c_1=c_1(n)$
and $c_2=c_2(n)$ such that
\begin{equation}\label{equivdiamond}
\Phi_\circ (c_1 t) \leq \Phi_\Diamond (t) \leq \Phi_\circ (c_2 t)
\quad \hbox{for $t\geq 0$,}
\end{equation}
see~\cite[Lemma~7]{Klimov76}.
Note that  if $\Phi$ is an  $n$-dimensional  $N$-function, then the functions $\Phi_\circ$ and $\Phi_\Diamond$ are $1$-dimensional $N$-functions and $\Phi_\bigstar$ is  an $n$-dimensional  $N$-function.
%Let us note that if $\Phi$ is an $n$-dimensional $N$-function, then
%$\Phi_\circ$ and $\Phi_\Diamond$ are $N$-functions and
%$\Phi_\bigstar$ is an $n$-dimensional $N$-function.
\\
Two more functions associated with $\Phi$, denoted by $\Phi_n$ and  $\widehat{\Phi_\circ}$, will be introduced in the next section in connection with Orlicz-Sobolev type embeddings.
 %{\color{blue} This is temporarily given twice: here and in the next section!
%
%\textit{Remark.} We note that Orlicz-Lorentz space $L[\widetilde
%{\Phi_\circ},n](\om)$, defined as in \eqref{dualoptnorm}, is the
%space  associate to $L(\widehat {\Phi_\circ}, -n)(\om)$ (up to
%equivalent norms). In particular,
%\begin{equation}
%\label{holderopt} \int_\om |uv|\, dx \leq C \|u\|_{L(\widehat
%{\Phi_\circ}, -n)(\om)}\|v\|_{L[\widetilde {\Phi_\circ},n](\om)}
%\end{equation}
%for some constant $C=C(n)$  and for every  $u, v\in
%\mathcal M(\om)$. Moreover,
%\begin{equation}
%\label{holderint} \int_\om |uv|\, dx \leq \int _0^{|\om|} u^*(s)
%v^*(s)\, ds \leq C \bigg(\int_0^{|\om|}\widehat
%{\Phi_\circ}\big(s^{-\frac{1}{n}}u^*(s)\big)\, ds +
%\int_0^{|\om|}\widetilde
%{\Phi_\circ}\big(s^{\frac{1}{n}}v^{**}(s)\big)\, ds\bigg)
%\end{equation}
%for some constant $C=C(n)$, for every  $u, v\in \mathcal M(\om)$.}
%\par
%To describe the behavior of the above transformations of $\Phi$ we
%shall consider some auxiliary functions.
\\
Some auxiliary functions depending on $\Phi$ will still be needed.
 We denote by $\Psi_\circ : [0,
\infty) \to [0, \infty)$   the increasing function given by
\begin{equation}\label{psicirc}
\Psi_\circ (t) = \frac{\Phi_\circ(t)}{t}
\qquad \hbox{for $t>0$ \quad
and $\quad \Psi_\circ (0)=0$.}
\end{equation}
Also, we call $\Psi _\Diamond : [0, \infty) \to [0,\infty)$  the
increasing function given by
\begin{equation}
\label{Psi} \Psi _\Diamond (t) = \frac{\Phi_\Diamond (t)}{t} \qquad
\hbox{for $t>0\quad$ and $\quad \Psi_\Diamond (0)=0$.}
\end{equation}
The function $\Theta:\rn\to\rp$ is defined as
\begin{equation}
\label{andrea3}\Theta(\xi) = \widetilde
{\Phi_{\Diamond}}^{-1}(\Phi(\xi)) \quad \hbox{for $\xi \in \rn$,}
\end{equation}
and the function $\Theta_\Diamond:\rp\to\rp$   as
\begin{equation}
\label{andrea4} \Theta_\Diamond  (t) = \widetilde
{\Phi_{\Diamond}}^{-1}(\Phi_\Diamond(t))\quad \hbox{for $t \geq 0$.}
\end{equation}
%One can verify that
%\begin{equation}\label{june26}
%\Phi _\Diamond \big(\Psi _\Diamond ^{-1}(t/2)\big) \leq
%\widetilde{\Phi_\Diamond} (t) \leq \Phi _\Diamond \big(\Psi
%_\Diamond ^{-1}(t)\big) \quad \hbox{for $t \geq 0$.}
%\end{equation}
Relations among the functions introduced above are the subject of the following lemma.

\begin{lemma}\label{lem:aux-aniso-est}
Let $\Phi : \rn \to [0, \infty)$ be an $n$-dimensional $N$-function, and let
$\Phi_\Diamond$, $\Psi_\Diamond$, $\Theta$, and  $\Theta_\Diamond$
be the functions associated with $\Phi$ as in \eqref{phidiamond},
\eqref{Psi},  \eqref{andrea3} and \eqref{andrea4}, respectively.
Then
\begin{itemize}
\item[(i)]  $\Phi_\Diamond \circ \Theta_\Diamond^{-1}= \widetilde {\Phi_\Diamond}$,
% and, consequently, it is  a convex function,
\item[(ii)]  $\Phi_{\Diamond}\circ \Theta_\Diamond^{-1}\circ \Theta = \Phi$,
\item[(iii)] $\Phi_\Diamond^{-1}\big(t\Psi_\Diamond ^{-1}(t)\big) = \Psi_\Diamond ^{-1}(t)$ \quad {\rm for} $t \geq 0$,
\item[(iv)] $\Theta_\Diamond \big(\Psi_\Diamond^{-1}\big)(t)\leq 2t$\quad
{\rm for} $t \geq 0$.
 \item[(v)]
$\Phi _\Diamond \big(\Psi _\Diamond ^{-1}(t/2)\big) \leq
\widetilde{\Phi_\Diamond} (t) \leq \Phi _\Diamond \big(\Psi
_\Diamond ^{-1}(t)\big)$ \quad {\rm for} $t \geq 0$.
\end{itemize}
\end{lemma}
\begin{proof} Equations {\it (i)} and {\it (ii)} are straightforward consequences of definitions \eqref{andrea3} and \eqref{andrea4}.
Equation {\it (iii)} easily follows on replacing $t$ by $\Psi_\Diamond ^{-1}(t)$ in the definition  of $\Psi_\Diamond$. As for  inequality {\it {(iv)}}, recall that, since $\Phi_\Diamond$ is a Young function, then, by \eqref{AAtilde},
\[ t \leq \Phi_\Diamond^{-1}(t) \widetilde
{\Phi_\Diamond}^{-1}(t)\leq 2t \quad \hbox{for $t\geq 0$.}
\]
By{\it (iii)} and the second inequality above we get
\begin{align}
\label{andrea6} \Theta_\Diamond \big(\Psi_\Diamond^{-1}(t)\big) & =
\widetilde
{\Phi_\Diamond}^{-1}\big(\Phi_\Diamond\big(\Psi_\Diamond^{-1}(t)\big)\big)=
\widetilde
{\Phi_\Diamond}^{-1}\big(\Phi_\Diamond\big(\Phi_\Diamond^{-1}\big(t\Psi_\Diamond
^{-1} (t)\big)\big)\big)=  \widetilde
{\Phi_\Diamond}^{-1}\big(t\Psi_\Diamond ^{-1} (t)\big)\\ \nonumber &
\leq \frac{2t\Psi_\Diamond ^{-1}
(t)}{\Phi_\Diamond^{-1}(t\Psi_\Diamond ^{-1} (t))}=
\frac{2t\Psi_\Diamond ^{-1} (t)}{\Psi_\Diamond ^{-1} (t)} = 2t \quad
\hbox{for $t\geq 0$.}
\end{align}
\\ Finally, property {\it (v)} follows via equation \eqref{AAtilde'} applied with $A$ replaced by $\Psi_\Diamond$.
 \end{proof}

\subsection{Sobolev embeddings}

The sharp  embeddings for anisotropic Orlicz--Sobolev spaces collected in~this subsection
are pivotal in our analysis.
\\ Let $\Phi$ be an $n$-dimensional  Young function. A basic anisotropic Poincar\'e-type inequality
 tells us that there exists a
constant $\kappa_1=\kappa_1(n)$ such that
\begin{equation}\label{anisopoinc}
\int_\om \Phi_\circ (\kappa_1
|\om|^{-\frac 1n} |u|)\, dx \leq \int _\om \Phi( \nabla u)\, dx\,,
\end{equation}
for every $ u
\in W_0^{1}\mathcal L^\Phi (\Omega )$, and
\begin{equation}\label{anisopoincnorm}
\|u\|_{L^{\Phi_\circ}(\om)} \leq \kappa_1^{-1}|\om|^{\frac 1n}\|\nabla
u\|_{L^\Phi(\om)}
\end{equation}
for every $ u
\in W_0^{1}L^\Phi (\Omega )$ -- see \cite[Proposition 3.2]{barlettacianchi}.
\\ The statement of optimal anisotropic Sobolev inequalities requires some further definitions.  Assume that
\begin{equation}\label{conv0}
\int _0\bigg(\frac t{\Phi _\circ (t)}\bigg)^{\frac 1{n-1}}\, dt <
\infty\,.
\end{equation}
If
\begin{equation}
\label{intdiv}
\int^\infty\left(\frac{t}{\Phi_\circ(t)}\right)^\frac{1}{n-1}\,dt
=\infty\,,
\end{equation}
then we denote by $\Phi _n : [0, \infty ) \to [0, \infty]$ the
Sobolev conjugate of $\Phi$ introduced in \cite{Cfully}. Namely, $\Phi_n$ is the Young function   defined as
\begin{equation}\label{sobconj}
\Phi_n (t)= \Phi _\circ (H^{-1}(t)) \quad \hbox{for $t \geq 0$,}
\end{equation}
where
 $H : [0, \infty ) \to [0, \infty)$ is given by
\begin{equation}\label{H1}
    H(t)= \bigg(\int _0^t \bigg(\frac \tau{\Phi _\circ (\tau)}\bigg)^{\frac 1{n-1}}\, d\tau\bigg)^{\frac {n-1}{n}} \quad \hbox{for $t \geq 0$.}
\end{equation}
Here, $H^{-1}$ denotes the
generalized left-continuous inverse of~$H$.
\\
By \cite[Theorem 1 and Remark 1]{Cfully}, there exists a constant $\kappa_2=\kappa_2 (n)$ such that
\begin{equation}\label{B-W}
\int_{\Omega}\Phi _n\left(\frac{ |u|}{\kappa_2\,
(\int_{\Omega}\Phi(\nabla u)dy)^{\frac 1{n}}}\right)dx \leq
\int_{\Omega}\Phi(\nabla u)\, dx
\end{equation}
for every $u \in W_0^{1}\mathcal L^\Phi (\Omega )$, and
\begin{equation}\label{B-Wbis}
 \|u \|_{L^{\Phi _n}(\Omega )} \leq \kappa_2 \|\nabla u \|_{L^\Phi (\Omega )}
\end{equation}
for every $u \in W_0^{1}L^\Phi (\Omega )$.  Moreover, $L^{\Phi _n}(\Omega )$ is the optimal, i.e. the smallest possible, Orlicz space which renders \eqref{B-Wbis} true for all $n$-dimensional  Young functions $\Phi$ with prescribed $\Phi _\circ$.
\\
This result can be still improved if embeddings of $W_0^{1}L^\Phi (\Omega )$ into the broader class of rearrangement-invariant target spaces are considered. Indeed,
denote by $\phi_\circ : [0, \infty) \to [0, \infty)$ the
non-decreasing, left-continuous function such that
$$
\Phi_\circ (t) = \int_0^t\phi_\circ(\tau)\, d\tau \quad \hbox{for $t\geq 0$,}
$$
and let $\widehat {\Phi_\circ}$ be the Young function given by
\begin{equation}
\label{Bphi} \widehat {\Phi_\circ} (t)= \int_0^t\widehat
{\phi_\circ}(\tau)\, d\tau \quad \hbox{for $t\geq 0$,}
\end{equation}
where $\widehat {\phi_\circ} :[0, \infty) \to [0, \infty)$ is the
non-decreasing, left-continuous function defined via
\begin{equation}\label{2,-2}
( \widehat {\phi_\circ})^{-1}(t) = \bigg(\int
_{\phi_\circ^{-1}(t)}^{\infty}\bigg(\int
_0^r\bigg(\frac{1}{\phi_\circ(t)}\bigg)^{\frac{1}{n-1}}
 dt\bigg)^{-n}\frac{dr}{\phi_\circ(r)^{\frac{n}{n-1}}}\bigg)^{\frac{1}{1-n}}\,\,\,\quad{\rm for}\,\,\,
 t\geq 0\,,
\end{equation}
and $\phi_\circ^{-1}$ and $\widehat {\phi_\circ}^{-1}$ are the
(generalized) left-continuous inverses of $\phi_\circ$ and $\widehat
{\phi_\circ}$, respectively. \iffalse Then we define $L(\widehat
\Phi, -n)(\om)$ as the rearrangement-invariant space of those
 $u\in \mathcal M(\om)$ for which the norm
\begin{equation}\label{optnorm}
\|u\|_{L(B_\Phi, -n)(\om)}=
\|s^{-\frac{1}{n}}u^*(s)\|_{L^{B_\Phi}(0, |\om|)}
\end{equation}
is finite. We also call $E(B_\Phi, -n)(\om)$ the subspace of
$L(B_\Phi, -n)(\om)$ defined as
$$E(B_\Phi, -n)(\om) = \bigg\{u:
\int_0^{|\om|}B_\Phi\big(\lambda s^{-\frac{1}{n}}u^*(s)\big)\, ds <
\infty \;\;\hbox{for every $\lambda>0$}\bigg\}.$$

\fi
\\
Let $L(\widehat {\Phi_\circ}, -n)(\om)$ be the Orlicz-Lorentz type space defined as in \eqref{optnorm}. By \cite{cianchi_ibero}, there exists a constant $\kappa_3=\kappa_3 (n)$ such that
\begin{equation}
\label{optint} \int_0^{|\om|}\widehat
{\Phi_\circ}\big(\kappa_3^{-1} s^{-\frac{1}{n}}u^*(s)\big)\, ds \leq
\int_\om \Phi (\nabla u)\, dx
\end{equation}
for every
$ u
\in W_0^{1}\mathcal L^\Phi (\Omega )$,
and
\begin{equation}
\label{optemb} \|u\|_{L(\widehat {\Phi_\circ}, -n)(\om)} \leq
\kappa_3 \|\nabla u\|_{L^\Phi(\om)}
\end{equation}
for every $ u
\in W_0^{1}L^\Phi (\Omega )$.
Moreover, $L(\widehat
{\Phi_\circ}, -n)(\om)$ is the optimal, i.e. the smallest possible,
rearrangement-invariant space which renders inequality
\eqref{optnorm} true for all $n$-dimensional  Young functions $\Phi$ with prescribed
$\Phi_\circ$.
\\
{ Let us notice that the Orlicz-Lorentz-type space $L[\widetilde {\Phi_\circ},n](\om)$,
defined as in \eqref{dualoptnorm}, is the associate space of
$L(\widehat {\Phi_\circ}, -n)(\om)$ (up to equivalent norms).
%In
%particular,
%\begin{equation}
%\label{holderopt} \int_\om |uv|\, dx \leq C \|u\|_{L(\widehat
%{\Phi_\circ}, -n)(\om)}\|v\|_{L[\widetilde {\Phi_\circ},n](\om)},
%\end{equation}
%for some constant $C=C(n)$  and for every  $u\in \mathcal M(\om)$.
Moreover, as shown in \cite[Inequality (4.46)]{cianchi_ibero},
\begin{equation}
\label{holderint} \int_\om |uv|\, dx \leq \int _0^{|\om|} u^*(s)
v^*(s)\, ds \leq C \bigg(\int_0^{|\om|}\widehat
{\Phi_\circ}\big(s^{-\frac{1}{n}}u^*(s)\big)\, ds +
\int_0^{|\om|}\widetilde
{\Phi_\circ}\big(s^{\frac{1}{n}}v^{**}(s)\big)\, ds\bigg)
\end{equation}
for some constant $C=C(n)$, and for every  $u, v\in \mathcal M(\om)$.}
\\ When $\Phi_\circ$ grows so fast near infinity  that condition \eqref{intdiv} fails, namely
\begin{equation}\label{intconv}
\int^\infty\left(\frac{t}{\Phi_\circ
(t)}\right)^\frac{1}{n-1}\,dt<\infty\,,
\end{equation}
then   there exists a constant $\kappa_4=\kappa_4(\Phi, n, |\om|)$ such that
\begin{equation}\label{B-Wter}
 \|u \|_{L^{\infty}(\Omega )} \leq \kappa_4 \|\nabla u \|_{L^\Phi (\Omega )}
\end{equation}
for every $ u
\in W_0^{1}L^\Phi (\Omega )$.

\subsection{Modular approximation}\label{subapprox}

One  obstacle to be faced when dealing with Orlicz and Orlicz-Sobolev spaces built upon Young functions that do not %not!!!
satisfy the $\Delta_2$-condition is the lack of separability of~these spaces. In particular, functions in these spaces cannot be approximated in norm by smooth functions. Substitutes for this property are certain approximation results in integral form, usually referred to~as~\lq\lq{}modular approximability\rq\rq{} in the theory of Orlicz spaces,  which are well fitted for~applications to~partial differential equations. This kind of approximation is well known  for  isotropic Orlicz and  Orlicz--Sobolev spaces, and goes back to~\cite{Gossez}. On the other hand, a counterpart in the more general anisotropic framework seems not to be completely settled yet. In this subsection, we recall a~few definitions %in this connection
and state the approximation properties that are needed in view of our main results. Their proofs present some additional difficulty comparing to the isotropic case, and are given in Section~\ref{approxproof}.

\smallskip
\par Let $\Phi$ be an $n$-dimensional  Young function and let $\om$ be a measurable set in $\rn$ with $|\om|<\infty$. A~sequence $\{U_k\}\subset L^\Phi(\om{;} \rn)$ is said to converge
modularly to $U$ in $L^\Phi(\Omega{;} \rn)$ if there exists
$\lambda>0$ such that
\begin{equation}
\label{july41} \lim _{k\to
\infty}\int_{\Omega}\Phi\left(\frac{U_k-U}{\lambda}\right)\, dx= 0.
\end{equation}
Note  that if $U_k \to U$ {modularly}, then $ U_k\to U$ in measure.

The following proposition links modular convergence to a kind of weak convergence against test functions in the associate space.

\begin{proposition}\label{prop:conv:mod-weak}  Let $\Phi$ be an $n$-dimensional  $N$-function and let $\om$ be a measurable set in $\rn$ with $|\om|<\infty$. Let $U \in \mathcal M(\om{;} \rn)$. Assume that    the sequence $\{U_k\} \subset \mathcal M(\om{;}\rn)$ and
that $U_k \to U$ modularly in $L^\Phi(\om{;} \rn)$.  Then there exists a subsequence of $\{U_k\}$, still indexed by $k$, such that
\begin{equation}
\label{july37} \lim _{k \to \infty}\int_\Omega U_k \cdot V\,dx
=  \int_\Omega  U \cdot V \,dx\qquad \text{for every }\quad V \in L^{\widetilde  \Phi}(\om{;}
\rn).
\end{equation}
\end{proposition}

The next result concerns the modular density  of simple functions  in anisotropic Orlicz spaces.

\begin{proposition}\label{prop:modulardensity}
 Let $\Phi$ be an $n$-dimensional  $N$-function and let $\om$ be a measurable set in $\rn$ with $|\om|<\infty$. Assume that  $U\in L^\Phi(\Omega;\rn)$. Then there exists  a sequence of simple
functions $\{U_k\}$ such that $U_k \to U$ modularly in $L^\Phi(\om , \rn)$.
\end{proposition}

We conclude with a modular smooth approximation property in anisotropic Orlicz-Sobolev spaces on bounded Lipschitz domains. Recall that an open set  $\Omega$ is called a \textit{Lipschitz domain} if each point of~$\partial \Omega$ has a~neighborhood
$\mathcal U $ such that $\Omega \cap \mathcal U$ is  the
subgraph of a Lipschitz continuous function of $n-1$ variables.

\begin{proposition}\label{prop:approx} Let $\Phi$ be an $n$-dimensional  $N$-function and  let  $\Omega$ be  a bounded Lipschitz domain in $\rn$.  Assume that  $u\in W_0^1L^\Phi(\Omega)\cap L^\infty(\Omega)$. Then there exists  a constant $C=C(\om)$ and a sequence $\{u_k\} \subset C_0^\infty(\Omega)$ such that
\begin{equation}\label{sep30}
u_k \to  u \quad \hbox{a.e. in $\om$,}
\end{equation}
\begin{equation}\label{july36}
\|u_k\|_{L^\infty(\Omega)}\leq C\|u\|_{L^\infty(\Omega)} \quad \hbox{for every $k\in\N$,}
\end{equation}
\begin{equation}\label{angela1grad} \nabla u_k \to
\nabla u \quad \hbox{modularly in $L^\Phi(\Omega;\rn)$.}
\end{equation}
\end{proposition}
\begin{remark}\label{approx-iso}
{\rm
In the isotropic case, namely  when   $\Phi(\xi) = A(|\xi|)$ for $\xi \in \rn$, for some
$N$-function $A$, properties \eqref{sep30} and~\eqref{angela1grad} in Proposition~\ref{prop:approx}   are known to hold even if the assumption $u \in L^\infty(\Omega)$ is dropped -- see  \cite[Theorem~4]{Gossez}. }
\end{remark}

\subsection{Some classical theorems of functional analysis}

We conclude this section  by recalling a~few well--known results of functional analysis, formulated in the anisotropic Orlicz space framework.   In~their  statements, $\om$  is assumed to be a measurable set in $\rn$ with $|\om|<\infty$.

\begin{theorem}{\rm {\bf [Vitali]}} \label{theo:VitConv}
%Suppose that $\om$ is a measurable set in $\rn$ with $|\om|<\infty$.
Assume that the sequence
$\{U_k\} \subset \mathcal M(\om; \rn)$ is uniformly integrable in $\om$, and there exists a
function $U: \om \to \rn$ such that $\lim_{k   \to \infty} U_k =U$
 a.e. in $\om$ \color{black} and  $|U|<\infty$ a.e. in
$\om$. Then $U \in L^1(\om{;} \rn)$ and $\lim _{k\to \infty}U_k = U$
in  $L^1(\om; \rn)$.
\end{theorem}

\begin{theorem}{\rm {\bf  [Dunford-Pettis]}}\label{theo:dunf-pet}
%Let $\om$ be a measurable set in $\rn$ with $|\om|<\infty$.
A family $\{U_\sigma\}_{\sigma \in \Sigma}$ of functions in $\mathcal M(\om; \rn)$
 is uniformly integrable in $L^1(\Omega;
\rn)$ if and only if it is relatively compact in the weak
topology.
\end{theorem}

\begin{theorem}{\rm {\bf [Anisotropic De La Vall\'ee Poussin]}} \label{theo:delaVP}
 Let $\Phi$ be an $n$-dimensional  $N$-function.
 %and let $\om$ be a measurable set in $\rn$ with $|\om|<\infty$.
 Assume that  $\{U_\sigma\}_{\sigma \in
\Sigma}$ is a family of  functions in $\mathcal M(\om; \rn)$
such that $\sup_{\sigma \in \Sigma}\int_\Omega \Phi(U_\sigma)\,
dx<\infty$. Then the family $\{U_\sigma\}$ is uniformly integrable.
\end{theorem}

The next result follows from the customary  version of the Banach-Alaoglu theorem, owing to~property \eqref{Dual} applied to $\Phi$ and $\widetilde \Phi$.  
Notice that, in view of that property, a sequence   $\{U_k\} \subset L^\Phi (\om, \rn)$   weakly-$*$ converges to   $U \in L^\Phi (\om, \rn)$  in $L^\Phi (\om, \rn)$    if
$$\lim_{k\to \infty} \int _\om U_k \cdot V \, dx = \int _\om U \cdot V \, dx$$
for every $V \in E^{\widetilde \Phi} (\om, \rn)$. Weak-$*$ convergence in $L^{\widetilde \Phi} (\om, \rn)$ can be characterized on exchanging the~roles of~$\Phi$ and $\widetilde \Phi$.

\begin{theorem}{\rm {\bf [Banach-Alaoglu in anisotropic Orlicz spaces]}}\label{theo:Banach-Alaoglu}
 Let $\Phi$ be an $n$-dimensional  $N$-function.
 %and let $\om$ be a measurable set in $\rn$ with $|\om|<\infty$.
 Then
the closed unit ball  in $L^\Phi(\om; \rn)$ and the closed unit ball  in $L^{\widetilde \Phi}(\om; \rn)$
 are
 weakly-$*$ compact in the respective spaces.
\end{theorem}

\section{Main results}\label{main}

%{\color{teal}
%We study the problem~\eqref{eq:main:f} under conditions~\eqref{A3}, \eqref{A2'}, and~\eqref{A2''} with some anisotropic $N$-function $\Phi$, nonnegative $h\in L^1(\Omega)$, and $c_\Phi\in(0,1]$.}

This is a  section of the paper where definitions of solutions to the Dirichlet problem \eqref{eq:main:f} are introduced and the pertaining existence, uniqueness, and regularity results are stated. In what follows, when referring to assumptions \eqref{A3}, \eqref{A2'}, and~\eqref{A2''}, we mean that they are fulfilled for some $N$-function $\Phi$, some function $h\in L^1(\Omega)$, and some constant  $c_\Phi\in(0,1]$.

\subsection{Weak solutions}\label{sec:weaksol}

Our first purpose is to detect  a minimal integrability condition on the datum  $f$ for a weak solution to  problem \eqref{eq:main:f} to exist.  In order to allow for the largest possible class of~admissible functions $f$, in the definition of weak solution that will be adopted the function $f$ is a~priori assumed to~be just integrable in $\om$. The class of test functions  is thus accordingly chosen for the  weak formulation of the problem to be well posed for any such $f$.

\begin{definition}\label{weaksol} {\rm {\bf [Weak solution]}} Let $f \in L^1(\Omega)$.  Under assumptions \eqref{A3}--\eqref{A2''}, a function $u \in W_0^{1}\mathcal L^{\Phi} (\Omega)$ is called a weak
solution to the Dirichlet problem \eqref{eq:main:f}  if
%$a(x,\nabla
%u) \in L^{\widetilde \Phi}(\Omega ;\rn)$, and
\begin{equation}
\label{weak-bound-e} \int_\Omega a(x, \nabla u) \cdot \nabla
\varphi\, dx = \int_\Omega f\varphi\,dx
\end{equation}
for every    $\varphi \in W_0^{1}\mathcal L^\Phi (\Omega) \cap
L^{\infty}(\Omega)$.
\end{definition}

Observe that both sides of equality \eqref{weak-bound-e} are well defined if $f$, $u$ and $\varphi$ are as in definition~\eqref{weaksol}. In particular, the integral on the left-hand side of \eqref{weak-bound-e}  is convergent by the H\"older inequality~\eqref{holder}, since, owing to  assumption \eqref{A2''}, $a(x,\nabla
u) \in L^{\widetilde \Phi}(\Omega ;\rn)$ provided that $u \in W_0^{1}\mathcal L^{\Phi} (\Omega)$.

%\noindent {\color{teal} Recall $\Phi_\circ$ given by~\eqref{phistar} and consider its two possible complementary asymptotic types of behaviour, that is~\eqref{intdiv} and \eqref{intconv}. We have the following existence result.}

Our main result about weak solutions is contained in Theorem~\ref{theo:boundex}. Its assumptions in connection with the existence (and uniqueness) of these solutions take a form of an alternative, depending on~a~threshold on the growth near infinity of the function $\Phi$. More precisely, what is relevant is the growth of its \lq\lq{}average\rq\rq{} $\Phi_\circ$, defined as in \eqref{phistar}, and the alternative corresponds to the two complementary conditions~\eqref{intdiv} and~\eqref{intconv}. Indeed, if $\Phi_\circ$ grows fast enough near infinity for the latter condition to hold, then any integrable function $f$ is admissible. On the other hand, if \eqref{intconv}  fails, and hence the former condition is in force, then a proper degree of integrability has to be imposed on $f$.
A natural ambient space for $f$ is the largest rearrangement-invariant space ensuring that the integral on the right-hand side of equation \eqref{weak-bound-e} is convergent for every test function $\varphi \in W^1_0L^\Phi(\om)$. This corresponds to the associate space $L[\widetilde {\Phi_\circ}, n](\om)$ of the optimal rearrangement-invariant target space $L(\widehat {\Phi_\circ}, -n)(\om)$
for embeddings of $W^1_0L^\Phi(\om)$ -- see \eqref{optemb}. Theorem \ref{theo:boundex} asserts that  the Dirichlet problem \eqref{eq:main:f} does actually admit a~unique weak solution  provided that $f$ belongs to the separable counterpart $E[\widetilde {\Phi_\circ}, n](\om)$ of $L[\widetilde {\Phi_\circ}, n](\om)$, defined as in \eqref{E}.
\\ As will be clear from Example \ref{ex-plap} in the next section, in the classical case of $p$-Laplacian-type problems, the two alternatives discussed above correspond to the situations when $p\leq n$ or $p>n$.   In the former case, our assumption  amounts to requiring that $f$ belongs to the  Lorentz space $L^{[\frac{np}{np+p-n}, p']}(\om)$, where $p'=\tfrac{p}{p-1}$, thus weakening the customary condition that $f \in L^{\frac{np}{np+p-n}}(\om)$.

\begin{theorem}
\label{theo:boundex} {\rm {\bf [Existence of weak solutions]}} Let
$\Omega$ be a bounded Lipschitz domain in $\rn$. Assume that
conditions  {\eqref{A3}--\eqref{A2''}} are in force, and  let $\Phi_\circ $ be the function associated with $\Phi$ as in \eqref{phistar}.
If~either
\begin{equation}\label{hpf}
\Phi_\circ \hbox{\quad grows so slowly that \eqref{intdiv} \ \
holds \  and  \ \  $f \in E[\widetilde {\Phi_\circ}, n](\om)$,}
\end{equation}
or
\begin{equation}\label{hpfbis}
\Phi_\circ \hbox{\quad grows so fast that \eqref{intconv}  \ \
holds \  and  \ \ $f \in L^1(\om)$\,,}
\end{equation}
 then there exists a unique weak solution $u\in W^1_0\mathcal L^{\Phi}(\om)$
to the Dirichlet problem \eqref{eq:main:f}.
%{Moreover,
%\begin{equation}
%\label{warsaw1}\int_\Omega \Phi(\nabla u)\,dx<\infty.
%\end{equation}}
\end{theorem}

%
%\begin{remark}\label{stronger}
%{\rm The piece of information   provided by  \eqref{warsaw1} is somewhat stronger than just mebership of $\nabla u$ in $L^{\Phi}(\om)$. Indeed, the latter property only ensures that equation  \eqref{warsaw1} is fufilled with $\nabla u$ replaced by $\lambda \nabla u$ for some $\lambda >0$.
%}
%\end{remark}

In some applications, we need to make use of the solution $u$ itself as a test function $\varphi$  in equation~\eqref{weak-bound-e} in the  definition of weak solution to problem \eqref{eq:main:f}. This requires $u$ to be bounded. An optimal condition on $f$ for this property to hold is exhibited in the next result.

\begin{proposition}
\label{boundedsol} {\rm {\bf [Boundedness of weak solutions]}}
Assume, in addition to the assumptions of~Theorem~\ref{theo:boundex}, that
 \begin{equation}
 \label{boundcond}
 \int _0^{|\om|}s^{-\frac{1}{n'}}\Psi_\circ^{-1}\big(\lambda s^{\frac{1}{n}}f^{**}(s)\big)\, ds < \infty\,
 \end{equation}
 for every $\lambda >0$, where $\Psi_\circ$ is defined as in \eqref{psicirc}.
 Then $u \in L^{\infty}(\Omega)$, and
there exists a constant $C=C(n)$ such that
 \begin{equation}\label{boundcond1}\|u\|_{L^\infty(\om)} \leq C \int _0^{|\om|}s^{-\frac{1}{n'}}\Psi_\circ^{-1}\big(C s^{\frac{1}{n}}f^{**}(s)\big)\, ds\,.
 \end{equation}
\end{proposition}

%\begin{remark}\label{utest}
%{\rm Under the assumptions of {Proposition
%\ref{boundedsol}}, the solution $u$ to the Dirichlet problem
%{\eqref{eq:main:f}} can be used as a test function $\varphi$ in equation \eqref{weak-bound-e} in the  definition of weak solution.}
%\end{remark}

\begin{remark}\label{bounddiamond}
{\rm Owing to equation \eqref{equivdiamond}, condition
\eqref{boundcond} can be equivalently formulated with $\Psi_\circ$
replaced by the function $\Psi_\Diamond$ defined by \eqref{Psi}. In
fact, the use of the latter function allows for an explicit sharp
value of the constant $\lambda$ in corresponding condition. Actually, the
weak solution $u$  to the Dirichlet problem  {\eqref{eq:main:f}--\eqref{A2''}} is
bounded provided that
 \begin{equation}\label{boundconddiamond}
\int _0^{|\om|}s^{-\frac{1}{n'}}\Psi_\Diamond^{-1}\bigg(\frac{
s^{\frac{1}{n}}}{n \omega_n^{1/n}}f^{**}(s)\bigg)\, ds < \infty\,.
 \end{equation}
 Moreover,
  \begin{equation}\label{boundconddiamond1}
\|u\|_{L^\infty(\om)} \leq \frac{1}{n \omega_n^{1/n}}\int
_0^{|\om|}s^{-\frac{1}{n'}}\Psi_\Diamond^{-1}\bigg(\frac{
s^{\frac{1}{n}}}{n \omega_n^{1/n}}f^{**}(s)\bigg)\, ds\,.
 \end{equation}
 Both condition \eqref{boundconddiamond} and the bound given by \eqref{boundconddiamond1} are sharp.
 The sufficiency of condition \eqref{boundconddiamond}, and the validity of estimate \eqref{boundconddiamond1} are apparent from a close inspection of the proof of Proposition \ref{boundedsol}. Their sharpness is due to the fact that equality holds in \eqref{boundconddiamond1} if $u$ is the solution to a suitable symmetric problem in a ball, which is stated in   equation \eqref{symmetrized} below.
 }
\end{remark}

\begin{remark}\label{weakregular}
{\rm
If condition   \eqref{boundconddiamond}, or even \eqref{boundcond}, is dropped,  boundedness of the weak solution  $u$ to~problem~\eqref{eq:main:f} $u$ is not guaranteed. In this case, sharp integrability properties of  $u$  can be derived via \cite[Proposition~3.7]{Ci-sym}.
 }
\end{remark}

\subsection{Approximable solutions}\label{sec:approxsol}

When neither of conditions \eqref{hpf} and  \eqref{hpfbis} holds,   weak solutions to~problem  \eqref{eq:main:f} do not necessarily exist. This calls for the use of some notion of solution, still weaker than that of weak solution, which enables to deal with  arbitrary right-hand sides $f \in L^1(\om)$, and yet with measure data, whatever $\Phi$ is. Merely distributional solutions are not satisfactory,  since even for linear equations this class of solutions does not guarantee uniqueness and permits well-known pathologies \cite{serrin}. These drawbacks can be overcome if, instead, solutions obtained as limits of solutions to approximating problems with regularized right-hand sides are introduced. Such a notion of solution has been extensively exploited, more or less explicitly, for nonlinear problems with isotropic growth -- see e.g.  \cite{bgSOLA, BBGGPV, dall, DMOP, min07, min-math-ann}. It restores uniqueness and, importantly, is well suited to analyze regularity.

\medskip

\par Approximable solutions to problem \eqref{eq:main:f} under the present assumptions on the differential operator, and with right-hand side in $L^1(\om)$, can be defined as follows.

%
%
%In particular, as recalled in Section \ref{intro},distributional so
%
%
%When neither condition \eqref{hpf} nor \eqref{hpfbis} holds,   weak solutions to problem  \eqref{eq:main:f} do not necessarily exist. {\color{teal} However, one can consider a notion} of solution to problem \eqref{eq:main:f} %can be
% given as the limit of solutions $u_k$ to the problems
%\begin{equation}\label{prob:trunc} \begin{cases}
%-\dv (a(x,\nabla u_k))= f_k &\qquad \mathrm{ in}\qquad  \Omega\\
%u_k =0 &\qquad \mathrm{  on}\qquad \partial\Omega,
%\end{cases}
%\end{equation}
%where $\{f_k\}$ is a sequence of smooth functions approximating $f$.

\begin{definition}\label{def:as:f}  {\rm {\bf [Approximable solution with $L^1$ data
]}}  Let $f \in L^1(\Omega)$.   Under assumptions \eqref{A3}--\eqref{A2''}, a~function $u\in {\mathcal T}^{1,\Phi}_0(\Omega)$ is called an approximable solution to problem \eqref{eq:main:f} if there exists a~sequence
 $\{f_k \}\subset L^\infty (\Omega)$ such that $f_k\to f$ in $L^1 (\Omega)$, and
the sequence of~weak solutions $\{u_k\}\subset W^{1}_0\mathcal L^\Phi
(\Omega)$ to problems
\begin{equation}\label{prob:trunc} \begin{cases}
-\dv \, a(x,\nabla u_k) = f_k &\quad \mathrm{ in}\quad  \Omega\\
u_k =0 &\quad \mathrm{  on}\quad \partial\Omega,
\end{cases}
\end{equation}
%\eqref{prob:trunc}
%
%
%
%~\eqref{eq:main:f}, with $f$ replaced by $f_k$,
satisfies
\begin{equation}
 \label{ae}
 u_k\to u\qquad  \text{a.e. in }\Omega.
 \end{equation}
\end{definition}

%As shown by our
%
%
%{\color{teal} To justify calling $u$ a `solution' to~\eqref{eq:main:f}, we note that in fact we prove that in our case it also holds that $\nabla u_k\to   \nabla u $ a.e.
%in $\Omega$, and consequently $a(x,\nabla u_k)\to a(x,\nabla u)$ a.e. in
%$\Omega$. These solutions turn to exist and to be unique in the sense that any sequence  $\{f_k \}\subset L^\infty (\Omega)$ convergent  to $f$ in $L^1 (\Omega)$ generates the sequence of solutions convergent to the same limit function $u$.}

Despite its apparent mildness, this definition gives grounds for an adequate generalized notion of~solution $u$ to problem~\eqref{eq:main:f}. Indeed, although the function $u$ is a priori assumed only  to be the pointwise limit of the solutions $u_k$ to the approximating problems \eqref{prob:trunc}, its \lq\lq{}surrogate gradient\rq\rq{} $\nabla u$, in the sense of~\eqref{gengrad}, turns out to be the pointwise limit of the weak gradients $\nabla u_k$, and hence $a(x,\nabla u_k)\to a(x,\nabla u)$ a.e. in $\Omega$ as well.
\par  This fact, together with the uniqueness of the approximable solution $u$ and its regularity, are the subject of the next theorem. Information about regularity amounts to   membership of $u$ and $\nabla u$ in~Marcinkiewicz-type spaces associated with the functions $\vt _n, \varrho_n : (0,\infty) \to (0,\infty)$ defined by
 \begin{equation}
 \vt_n(t)=\frac{\Phi_n(t^{1/n'})}{t}\qquad \text{  and }\qquad  \varrho_n(t)=\frac{t}{\Phi^{-1}_n(t)^{n'}} \qquad \hbox{for}\quad t > 0,
 \label{vt1vt2}
 \end{equation}
respectively. Here, $\Phi_n$ denotes the Sobolev conjugate of $\Phi$ given by~\eqref{sobconj}.

\begin{theorem}\label{theo:main-f}  {\rm {\bf [Well-posedness and regularity with $L^1$ data]}} Let $\Omega$ be a bounded Lipschitz domain in $\rn$ and let $f \in L^{1}(\om)$. Assume that conditions \eqref{A3}--\eqref{A2''} and \eqref{intdiv} are in force.
 Then there exists a unique approximable solution $u\in {\mathcal T}^{1,\Phi}_0(\Omega)$ to the Dirichlet problem~\eqref{eq:main:f}.   If $\{u_k\}$ is any sequence as in the definition of approximable solution, then $\nabla u_k\to   \nabla u $ a.e. in $\Omega$, where $\nabla u$ has to be understood in the sense of equation \eqref{gengrad}.
Moreover,
\begin{equation}\label{umarc}
u \in L^{\vt_n(\cdot),\infty}(\om) \qquad \hbox{and} \qquad  \Phi(\nabla u) \in L^{\varrho_n(\cdot),\infty}(\om),
\end{equation}
%and
%\begin{equation}\label{gradmarc}
%\Phi(\nabla u) \in L^{\varrho_n(\cdot),\infty}(\om),
%\end{equation}
where $\vt_n$ and $\varrho_n$ are the functions defined as in
\eqref{vt1vt2}.
\end{theorem}

\begin{remark}\label{remapprox} {\rm
Theorem \ref{theo:main-f}  is relevant, and therefore stated, only under assumption \eqref{intdiv}. Actually, if~$\Phi_\circ$ grows so fast near infinity that \eqref{intdiv}  is violated, and hence  \eqref{intconv}
is satisfied, then a weak solution certainly exists by Theorem \ref{theo:boundex}, and, by their uniqueness, it agrees with the  approximable one.}
\end{remark}

%
%{\color{teal} If $\Phi$ is increasing so fast that \eqref{intdiv} is violated, that is when \eqref{intconv}
%is satisfied, then the approximable solutions are always weak
%solutions. The gradient regularity therein is obviously better than the one the we can
%obtain via this approach.}

We conclude this section by considering
the still more general situation when the function $f$ in~problem~\eqref{eq:main:f} is replaced by a signed Radon  measure
$\mu$ with finite total variation $\|\mu\|(\om)$.
Approximable solutions to the corresponding Dirichlet problem
\begin{equation}\label{eqmeas}
\begin{cases}
- {\rm div}  \, a(x, \nabla u)  = \mu  &\quad {\rm in}\quad \om \\
 u =0  &\quad {\rm on}\quad
\partial \om \,
\end{cases}
\end{equation}
can be defined in analogy with Definition \ref{def:as:f}, provided that
convergence of the approximating sequence $\{f_k\}$ to $f$ in $L^1(\om)$ is replaced  by   weak-$\ast$ convergence in the space of
measures.  Recall that a sequence of functions $\{f_k\} \subset L^1(\om)$ is said to weak-$\ast$ converge to $\mu$ in the space of measures if
\begin{equation}\label{meas1}
\lim _{k \to \infty} \int _\om \varphi f_k \, dx = \int _\om \varphi
\,d\mu
\end{equation}
for every function $\varphi \in C _0(\om)$. Here, $C _0(\om)$ denotes the
space of continuous functions with compact support in $\om$.

\begin{definition}\label{def:as:mu}  {\rm {\bf [Approximable solution with measure data]}}  Let $\mu$ be a signed Radon  measure with finite total variation on $\om$.   Under assumptions \eqref{A3}--\eqref{A2''},
 a~function $u\in {\mathcal T}^{1,\Phi}_0(\Omega)$ is called an~approximable solution to problem \eqref{eqmeas} if there exists a sequence
 $\{f_k \}\subset L^\infty (\Omega)$  weakly-$*$ converging to $\mu$ in the space of measures,  such that the sequence of~weak solutions $\{u_k\}\subset W^{1}_0\mathcal L^\Phi
(\Omega)$ to problems~\eqref{prob:trunc} satisfies $$ u_k\to u  \quad \hbox{
a.e. in  $\Omega$.}$$
\end{definition}

Apart from uniqueness, an   analogue to Theorem \ref{theo:main-f} for approximable solutions $u$ with measure data can be established via essentially the same proof. In particular,  a.e. convergence of gradients, and hence of~the nonlinear  coefficient of the differential operator, as well as regularity of $u$ and $\nabla u$ hold exactly as in the case of data in $L^1(\om)$.

%
%
%{\color{teal}
%\noindent As in the $L^1$-data case calling $u$ a `solution' is justified, since in our case it also holds that $\nabla u_k\to   \nabla u $ a.e. in $\Omega$, and consequently $a(x,\nabla u_k)\to a(x,\nabla u)$ a.e. in $\Omega$.}

\begin{theorem}\label{theo:main-mu}  {\rm {\bf [Existence and regularity with measure data]}}  Let $\Omega$ be a bounded Lipschitz domain in $\rn$  and   let  $\mu$ be a signed Radon  measure with finite total variation on $\om$.    Assume that conditions \eqref{A3}--\eqref{A2''}  are in force.
 Then, there exists an approximable solution $u\in {\mathcal T}^{1,\Phi}_0(\Omega)$ to the Dirichlet problem~\eqref{eqmeas}.    If $\{u_k\}$ is the sequence in the~definition of~approximable solution then $\nabla u_k\to   \nabla u $ a.e. in $\Omega$.
Moreover, $u$ and $\nabla u$  fuflill property \eqref{umarc}.% and \eqref{gradmarc}.
\end{theorem}

\section{Special instances}\label{sec:instances}

In  this section we implement the results stated above in cases when the $N$-function $\Phi$ takes one of~the forms given by \eqref{Phi=p}--\eqref{trud1}. Model equations whose nonlinearities are  driven by these specific functions $\Phi$ are also exhibited.
\par In what follows, the relation $\phi_1 \approx \phi_2$ between two  functions $\phi_i :  I \to [0, \infty]$, $i=1,2$, where $I$ is either $\rn$ or $[0, \infty)$, means that there exist positive constants $c_1$ and $c_2$ such that $\phi_1(c_1 x) \leq \phi_2(x) \leq \phi_1(c_2x)$ for every $x \in I$. If these inequalities hold for $|x|$ larger than some positive constant $M$, we shall write that $\phi_1 \approx \phi_2$ {\it near infinity}.

\begin{ex}\rm\label{ex-plap} A prototypical  equation with a power growth in the gradient is the $p$-Laplace equation. In~a~slightly generalized form, involving a non-necessarily  smooth coefficient, the corresponding Dirichlet problem reads
\begin{equation}\label{eq:p}
\begin{cases}
-\dv \, (b(x)|\nabla u|^{p-2}\nabla u)= f &\quad \hbox{in\ \ $\Omega$}\\
u=0 &\quad \hbox{on\ \ $\partial\Omega$\,,}
\end{cases}
\end{equation}
where  $1<p<\infty$ and   $b \in L^{\infty}(\Omega)$ is   such that $b(x)\geq c$ for some positive constant $c$. Without loss of~generality, here, and in similar circumstances in the following examples, we assume for simplicity that $c=1$. Plainly, assumptions  \eqref{A2'} and \eqref{A2''} are now fulfilled with $\Phi$ obeying \eqref{Phi=p}, namely $\Phi(\xi)=|\xi|^p$.
Note that, with this choice of $\Phi$,   assumption \eqref{A2''} agrees with the classical growth condition
%\todo{This is correct only if we assume $b$ to be bounded from above?}
$$|a(x,\xi)|\leq c\big(|\xi|^{p-1} + g(x)\big) \quad \quad \hbox{for a.e. $x \in \om$ and every $\xi \in \rn$,}$$
for some function $g \in L^{p'}(\Omega)$ and some constant $c>0$.
%, $h=0$ and any constant $c_\Phi \in (0, 1)$.
Existence and regularity of weak and approximable solutions to problem \eqref{eq:p} are discussed below in items {\it A)} and {\it B)}, respectively.
\begin{itemize}
\item[{\it A)}] Theorem \ref{theo:boundex} implies  that problem \eqref{eq:p} has a unique {weak solution} $u$ in each of the following cases:
\begin{eqnarray}
\label{hpp1}
1<p< n&  \ \
\text{ and } \ \   &f \in L^{[\frac{np}{np+p-n}, p']}(\om),
\\
\label{hpp2}
p= n& \ \
\text{ and } \ \  &  {f \in L^{[1, n']}(\om)}\,,\\
 \label{hpp3}
p>n
& \ \
\text{ and } \ \   &f \in L^1(\om)\,.
\end{eqnarray}
Case \eqref{hpp1}   extends a standard result on the existence of weak solutions under the assumption that $f\in L^{\frac{np}{np+p-n}}(\om)$, since the latter space is strictly contained in $L^{[\frac{np}{np+p-n}, p']}(\om)$. As far as we know, the result in the borderline situation \eqref{hpp2} is new. The conclusion under \eqref{hpp3} is classical.
\item[{\it B)}] Assume now that $f \in L^1(\om)$ and $1<p\leq n$.  Theorem~\ref{theo:main-f} yields the existence and uniqueness of an {approximable solution} $u$ to problem \eqref{eq:p}. The existence of such a solution is guaranteed by Theorem~\ref{theo:main-mu} even if  $f$ is replaced   by a signed measure $\mu$ with finite total variation on $\om$.  In both cases, if $1<p<n$, then
\begin{equation}\label{hpp4}
u \in L^{\frac{n(p-1)}{n-p}, \infty}(\om) \qquad \hbox{and} \qquad |\nabla u| \in L^{\frac{n(p-1)}{n-1}, \infty}(\om) \,.
\end{equation}
In the limiting case when $p=n$, the  approximable solution in question   fulfills
\begin{equation}\label{hpp5}
u \in \exp L(\om) \qquad \hbox{and} \qquad |\nabla u| \in L^{\varrho(\cdot),\infty}(\om) \,,
\end{equation}
where $\varrho(t) \approx\frac{t^n}{\log t}$ near infinity. Property  \eqref{hpp4} is nowadays classical --
%\todo{Check reference\\
%IS: yes, it is in~\cite{BBGGPV}}
see \cite{BBGGPV}. Equation~\eqref{hpp5} is a special case of~\cite[Example~3.4]{CiMa}. In  \cite{DHM} it is shown that, indeed,   $|\nabla u| \in L^{n, \infty}(\om)$ when $p=n$. This stronger piece of information is derived via ad hoc sophisticated techniques, exploiting the fact that the differential operator has exactly an $n$-growth.
\end{itemize}\end{ex}

\begin{ex}\rm\label{ex-iso-Zyg} Consider next the   case when  problem \eqref{eq:main:f} has still an isotropic growth, but not necessarily of power type. A model with this regard is provided by the problem
\begin{equation}\label{eq:A}
\begin{cases}
-\dv \, \bigg(b(x)\displaystyle\frac{A(|\nabla u|)}{|\nabla u|^2}\nabla u\bigg)= f &\qquad \hbox{in $\Omega$}\\
u=0 &\qquad \hbox{on $\partial\Omega$\,,}
\end{cases}
\end{equation}
where $A$ is an $N$-function and  $b \in L^{\infty}(\Omega)$ is   such that $b(x)\geq 1$. Clearly, problem~\eqref{eq:A} reduces to~\eqref{eq:p} when $A(t)=t^p$ for some $p>1$.
Assumption \eqref{A2'} and \eqref{A2''} are satisfied with $\Phi$ given by \eqref{Phi=A}, i.e.~$\Phi (\xi) =A(|\xi|)$ for $\xi \in \rn$.
% , $h=0$ and any constant $c_\Phi \in (0, 1)$.
 In particular, owing to the first inequality in \eqref{AAtilde'}, assumption \eqref{A2''} is equivalent to
%the typical
 %growth condition
 $$|a(x,\xi)|\leq c\big(A(|\xi|)/|\xi| + g(x)\big)  \quad \quad \hbox{for a.e. $x \in \om$ and every $\xi\in \rn$,}$$
for some function $g \in L^{\wt{A}}(\Omega)$ and some constant $c$, which  agrees with  a growth condition typically imposed under the $\Delta _2$-condition on $A$.
Of course, here the expression $A(|\xi|)/|\xi| $ has to be understood as $0$ if $\xi=0$. Since $\Phi_\circ (t) = A(t)$ in the situation at hand, our conclusions about weak solutions and approximable solutions to problem \eqref{eq:A} can be derived from Theorems \ref{theo:boundex} and  \ref{theo:main-f}   just on replacing $\Phi_\circ$ by $A$ in all relevant occurrences.
\par
 For instance, consider the case when
\begin{equation}\label{powerlog}
A(t) \approx t^p (\log t)^\alpha \qquad {\hbox{near infinity,}}
\end{equation}
where either $p >1$ and $\alpha \in \mathbb R$, or $p=1$ and $\alpha >0$.
%{\color{blue} Then the Orlicz-Lorentz type space $L[A,r](\om)$ coincides with the so-called Lorentz-Zygmund space defined, for $\alpha\in \mathbb R$ and $\tfrac 1r = \tfrac 1\gamma - \tfrac 1\sigma$ for some $\gamma \in [0, \infty)$, as
%\begin{equation}\label{lorentzzug}
% L[A,r](\om)=  E[A,r](\om)=L^{[\gamma,\sigma]}(\log L)^\alpha (\om)\,,
%\end{equation} where the equality holds up to equivalent norms and for $u \in \mathcal M (\om)$\todo{parameters to be checked/changed}
%\begin{equation}\label{LZ}
%\|u\|_{L^{[\gamma,\sigma]}(\log L)^\alpha(\om)}
%=\bigg( \int _0^{|\om|}\big(s^{\frac 1\gamma} u^{**}(s)\big)^{\sigma}\log^\alpha \Big(1 +\frac {|\Omega|}s\Big)\, \frac {ds}s\bigg)^{\frac 1\sigma} .
%\end{equation}}
%We infer the following existence and regularity of weak and approximable solutions to problem \eqref{eq:A} with $A$ satisfying~\eqref{powerlog}.
\\
The  conclusions described below can be derived from our general results.
Equation \eqref{lorentzzug} is also exploited for such a derivation. In what follows, $E[\exp L^{\frac 1\alpha}, n](\Omega)$ denotes the space defined as in \eqref{E}, with $A(t) \approx e^{t^{1/\alpha}}$ near infinity.
\begin{itemize}
\item[{\it A)}] Theorem \ref{theo:boundex} tells us  that problem \eqref{eq:A}   admits a unique {weak solution}
 $u$ under any of the following assumptions:
\begin{eqnarray}
\label{hplog1}
\hbox{$p= 1$  and $\alpha >0$},\ &
 \text{and }  \ &
  f \in E[\exp L^{\frac 1\alpha}, n](\Omega),
%  \int _0^{|\om|} e^{\lambda (s^{\frac 1n} f^{**}(s))^{\frac 1\alpha}} \, ds < \infty \,\,\hbox{for every $\lambda >0$},
\\\label{hplog2}
\begin{cases}\hbox{either $1< p< n$, $\alpha \in \mathbb R$}\\ \hbox{or $p=n$, $\alpha \leq n-1$},\end{cases}  \ &
 \text{and }  \ & f \in L^{[\frac{np}{np+p-n}, p']}(\log L)^{-\frac \alpha{p-1}}(\om),
\\\label{hplog3}
\begin{cases}\hbox{either $p>n$},\\
\hbox{or $p=n$ and $\alpha > n-1$,}\end{cases}\ &
 \text{and }  \ & f \in L^1(\om)\,.
\end{eqnarray}
%Note that in deriving the condition on $f$ in \eqref{hplog2} we have also made use of \cite[Lemma 6.12, Chapter 4]{BS}.

\item[{\it B)}]  If $f \in L^1(\om)$, then Theorem~\ref{theo:main-f} provides us with the  existence and uniqueness of an {approximable solution} $u$ to problem   \eqref{eq:A}. When $f$ is replaced by a signed measure $\mu$ with finite total variation, Theorem~\ref{theo:main-mu} applies to ensure the existence of a solution of the same kind. Moreover, in  both cases:
  \begin{itemize}
\item[{\it i)}] if $1 \leq p <n$,
% either $p=1$ and $\alpha >0$, or $1<p<n$ and $\alpha \in \mathbb R$,
then
% \footnote{ IC:  that here we should have
% \[\vr(t)\approx t^{\frac{n(p-1)}{n- p}}(\log t)^{\frac {\alpha n}{n-p}}(\om) \qquad \hbox{and} \qquad \vr(t)\approx t^{\frac{n(p-1)}{n-1}}(\log t)^{\frac {\alpha n}{n-1}}(\om)\quad\text{near infinity}.\] }
 %\todo{AC: corrected $\rho$}
\begin{eqnarray}\label{hplog4bis}
 u \in L^{\vartheta(\cdot),\infty} (\om) \quad & \text{and} & \quad \nabla u \in L^{\varrho (\cdot),\infty}(\om)\,,\\
\nonumber
\text{where }\quad {\vartheta} (t) \approx t^{\frac{n(p-1)}{n-p}}(\log t)^{\frac{n\alpha}{n-p}} \quad & \hbox{and} &\quad{{\varrho} (t) \approx  t^{\frac{n(p-1)}{n-1}}(\log t)^{\frac{n\alpha}{n-1}}}\quad \hbox{near infinity;}\end{eqnarray}
\item[{\it ii)}] if $p=n$ and $\alpha <n-1$, then
\begin{eqnarray}\label{hplog5}
 u \in \exp L^{\frac{n-1}{n-1-\alpha}} (\om) \quad & \hbox{and}& \quad \nabla u \in L^{\varrho (\cdot),\infty}(\om) \\\nonumber
& \text{where}& \quad   \varrho (t) \approx t^{n}(\log t)^{\frac{\alpha n }{n-1}-1} \hbox{ near infinity;}
\end{eqnarray}
\item[{\it iii)}] if $p=n$ and $\alpha =n-1$, then
\begin{eqnarray}\label{hplog6}
 u \in \exp \exp L (\om) \quad &\hbox{and}&\quad \nabla u \in L^{\varrho (\cdot),\infty}(\om) \\\nonumber
& \text{where}& \quad\varrho (t) \approx t^{n}(\log t)^{n-1}(\log\log t)^{-1} \quad \hbox{near infinity.}\end{eqnarray}
\end{itemize}
Properties \eqref{hplog4bis}, \eqref{hplog5} and \eqref{hplog6} were established in  \cite[Example 3.4]{CiMa}, except for the case when $p=1$ in  \eqref{hplog4bis}, which is new. This case involves an $N$-function $A$ that does not satisfy the $\nabla _2$-condition near infinity, a situation that is not contemplated in \cite{CiMa}.\end{itemize}
\end{ex}

\begin{ex}\rm\label{ex-aniso-plap}
Pattern anisotropic problems
% displaying
% affected by
 %anisotropic nonlinearities
have the form
\begin{equation}\label{eq:pi}
\begin{cases}
\displaystyle -\sum _{i=1}^n \, \big(b_i(x)|u_{x_i}|^{p_i-2}u_{x_i}\big)_{x_i}= f &\qquad \hbox{in $\Omega$}\\
u=0 &\qquad \hbox{on $\partial\Omega$\,,}
\end{cases}
\end{equation}
where $u_{x_i}$ denotes partial derivative with respect to the variable $x_i$,  the functions  $b_i \in L^{\infty}(\Omega)$  are   such that $b_i(x)\geq 1$, and $p_i>1$ for $i=1, \dots , n$. Here, assumptions \eqref{A2'} and \eqref{A2''} are fulfilled with $\Phi$ as in
 \eqref{Phi=pi}, namely
$\Phi (\xi) = \sum _{i=1}^n |\xi_i|^{p_i}$ for $\xi \in \rn$.
 One has that
\begin{equation}\label{pbar}
\Phi_\circ (t) \approx t^{\overline p} \qquad \hbox{for $t\geq 0$,}
\end{equation}
where $\overline p$ denotes the harmonic mean of the exponents $p_i$. Namely,
\begin{equation}\label{harmonic}
\overline p = \frac 1{\frac 1n\sum _{i=1}^n \frac 1{p_i}}\,.
\end{equation}
Equation \eqref{pbar} is a special case of \eqref{Abar} below.
\\
% We infer the following existence and regularity of weak and approximable
%solutions to problem \eqref{eq:pi}.
Our results with regard to problem \eqref{eq:pi}  can be described as follows.
\begin{itemize}
\item[{\it A)}] Owing to Theorem \ref{theo:boundex}, a unique {weak solution} to problem \eqref{eq:pi} exists under  the same conditions as in \eqref{hpp1}--\eqref{hpp3}, with $p$ replaced with $\overline p$.
\item[{\it B)}]  When $f \in L^1(\om)$ and $1<\overline{p}\leq n$, Theorem~\ref{theo:main-f} yields the existence and uniqueness \mbox{of an }ap\-proximable solution $u$ to problem \eqref{eq:pi}. An  {approximable solution}   also  exists, owing to Theorem~\ref{theo:main-mu}, if a signed measure $\mu$ with finite total variation replaces  $f$ in problem  \eqref{eq:pi}.  Moreover, if  $1<\overline{p}<n$, then
\begin{equation}\label{hppi1}
u \in L^{\frac{n(\overline p-1)}{n-\overline p}, \infty}(\om) \qquad \hbox{and} \qquad u_{x_i} \in L^{\frac{p_in(\overline p-1)}{(n-1)\overline p}, \infty}(\om) \quad \hbox{for $i=1, \dots , n$}\,,
\end{equation}
whereas, if  $\overline p=n$, then
\begin{equation}\label{hppi2}
u \in \exp L(\om) \qquad \hbox{and} \qquad u_{x_i} \in  L^{\varrho_i (\cdot),\infty}(\om) \,,\quad\text{where }\quad \varrho_i (t) \approx\frac{t^{p_i}}{\log t}\ \text{near infinity}.
\end{equation}
Property \eqref{hppi1} extends and enhances a result of \cite{BGM}, proved only for $p_i \geq 2$, $i=1, \dots , n$, and yielding the weaker piece of information that $u_{x_i} \in L^q(\om)$ for every $q<\frac{p_in(\overline p-1)}{(n-1)\overline p}$.
\end{itemize}
\end{ex}

\begin{ex}\rm\label{ex-aniso-Zyg}
Problem \eqref{eq:pi} is a distinguished  member  of a more general class of problems taking the form
\begin{equation}\label{eq:Ai}
\begin{cases}
\displaystyle -\sum _{i=1}^n \, \bigg(b_i(x)\frac{A_i(|u_{x_i}|)}{|u_{x_i}|^2} u_{x_i}\bigg)_{x_i}= f &\qquad \hbox{in $\Omega$}\\
u=0 &\qquad \hbox{on $\partial\Omega$\,,}
\end{cases}
\end{equation}
where $A_i$ are  $N$-functions,  and   $b_i \in L^{\infty}(\Omega)$ are   such that $b_i(x)\geq 1$, for $i=1, \dots , n$. A choice of~the~function $\Phi$ that renders assumptions \eqref{A2'} and \eqref{A2''} true is now  \eqref{Phi=Ai}, i.e.
$
\Phi (\xi) = \sum _{i=1}^n A_i(|\xi_i|)$   for~$\xi \in \rn$. One can show that
$$\Phi _\circ (t) \approx \overline A (t) \qquad \hbox{near infinity,}$$
where $\overline A$ is the $N$-function obeying
\begin{equation}\label{Abar}
\overline{A} ^{\, -1}(\tau) = \bigg( \prod _{i=1}^n A_i^{-1}(\tau)\bigg)^{\frac 1n} \quad \hbox{ for $\tau \geq 0$,}
\end{equation}
see \cite[Equation 1.9]{Cfully}. Thus, our results about weak and approximable solutions to problem \eqref{eq:Ai} follow from Theorems \ref{theo:boundex}, \ref{theo:main-f}, and \ref{theo:main-mu}  on replacing $\Phi_\circ$ by $\overline A$ throughout.
\par
To give the flavor of the  conclusions  that can be derived from these theorems, let us test them on~the example given by choosing
\begin{equation}
\label{powerlogi}
A_i(t) \approx t^{p_i} (\log t)^{\alpha _i} \qquad \hbox{near infinity,}
\end{equation}
where either $p_i >1$ and $\alpha_i \in \mathbb R$, or $p_i=1$ and $\alpha_i >0$, for $i=1, 	\dots , n$.
Let $\overline p$
be given   by \eqref{harmonic}, and let $\overline \alpha$ be
defined as
$$\overline \alpha = \frac {\overline p}{n} \sum_{i=1}^n \frac {\alpha
_i}{p_i}.$$
One can verify via \eqref{Abar} that
$$\overline{A} (t) \approx
t^{\overline p} (\log  t)^{\overline \alpha}  \qquad \hbox{near infinity.}
$$
%We infer the following existence and regularity of weak and approximable solutions to problem \eqref{eq:Ai} with $A$ satisfying~\eqref{powerlogi}.
Then we have what follows.
\begin{itemize}
\item[{\it A)}] The existence and uniqueness of a {weak solution} to problem  \eqref{eq:Ai}, with $A_i$ given by \eqref{powerlogi}, depends on the exponents $p_i$ and $\alpha_i$ only through $\overline p$ and $\overline \alpha$, according to the same assumptions as in \eqref{hplog1}--\eqref{hplog3}, with $p$ and $\alpha$ replaced by $\overline p$ and $\overline \alpha$.

\item[{\it B)}] Theorem~\ref{theo:main-f} or Theorem~\ref{theo:main-mu} ensure that an {approximable solution} $u$ to  problem  \eqref{eq:Ai}, with $A_i$ given by \eqref{powerlogi}, exists whenever $f\in L^1(\om)$, or $f$ is replaced by a signed measure with finite total variantion, respectively. In the former case, the uniqueness of the solution is also assured.
In both cases:
\begin{itemize}
\item[{\it i)}]  if $1 \leq \overline p  <n$, then
\begin{align} \label{march1}
 u \in L^{\vartheta(\cdot),\infty} (\om) \quad & \text{and}  \quad  u_{x_i} \in L^{\varrho _i(\cdot),\infty}(\om)\quad \hbox{for $i=1, \dots , n$}\,, \\
\nonumber
\text{where }
{\vartheta} (t) \approx t^{\frac{n(\overline p-1)}{n-\overline p}}(\log t)^{\frac{n\overline \alpha}{n-\overline p}} \quad & \hbox{and}  \quad{ {\varrho}_i (t) \approx  t^{\frac{p_i n(\overline p-1)}{(n-1)\overline p }}(\log t)^{\frac{n(\alpha_i (\overline p -1) + \overline \alpha)}{(n-1)\overline p}}}\quad \hbox{near infinity;}
\end{align}

% if $1 \leq \overline p  <n$, then\footnote{\color{teal}IC: I believe that here we should have \[u \in L^{\vt(\cdot),\infty}(\Omega)\qquad\text{and} u \in L^{\vr(\cdot),\infty}(\Omega),\qquad\]
% \[\vr(t)\approx t^{\frac{n(\overline p-1)}{n-\overline p}}(\log t)^{\frac {\overline \alpha n}{n-\overline p}}(\om) \qquad \hbox{and} \qquad \vr(t)\approx t^{\frac{p_in(\overline p-1)}{(n-1)\overline p}}(\log t)^{\frac{p_i}{\overline{p}}\frac {\overline \alpha n}{n-1}}(\om)\quad\text{near infinity}.\]
% }
% \begin{equation}\label{hpplogi1}
% u \in L^{\frac{n(\overline p-1)}{n-\overline p}(\log L)^{\frac {\overline \alpha n}{n-\overline p}}{,\infty}}(\om) \qquad \hbox{and} \qquad u_{x_i} \in  {L^{\frac{p_in(\overline p-1)}{(n-1)\overline p},\infty}(\om)}
% \quad \hbox{for $i=1, \dots , n$}\,;
% \end{equation}
%where $\varrho _i (t) \approx t^{\frac{p_in(\overline p-1)}{(n-1)\overline p}}$\footnote{IC: why this layout?} near infinity.
\item[{\it ii)}] if  $\overline p=n$ and $\overline \alpha < n-1$, then
\begin{equation}\label{hpplogi2}
 {u \in \exp L^{\frac{n-1}{n-1-\overline \alpha}} (\om) }\qquad \hbox{and} \qquad u_{x_i} \in L^{\varrho_i (\cdot),\infty}(\om)
\quad \hbox{for $i=1, \dots , n$}\,;
\end{equation}
where $\varrho_i (t) \approx t^{p_i} (\log t)^{\frac{\alpha _i(n-1) + \overline \alpha}{n-1}-1}$ near infinity.

\item[{\it iii)}] if $\overline p=n$ and $\overline \alpha =n-1$, then
\begin{equation}\label{hplogi3}
 u \in \exp \exp L (\om) \quad \hbox{and} \quad u_{x_i} \in L^{\varrho_i (\cdot),\infty} (\om)
\quad \hbox{for $i=1, \dots , n$}\,
\end{equation}
where
$ {\varrho_i (t) \approx t^{p_i}(\log t)^{\alpha _i-1}(\log\log t)^{-1}} \quad \hbox{near infinity.}$\end{itemize}
\end{itemize}
\end{ex}

\begin{ex}\rm\label{ex-aniso-Trud} Assume that $\Omega\subset\R^2$, and consider any Dirichlet problem
\begin{equation}\label{eq:T}
\begin{cases}
-\dv \, a(x,\nabla u)= f &\qquad \hbox{in $\Omega$}\\
u=0 &\qquad \hbox{on $\partial\Omega$\,}
\end{cases}
\end{equation}
under assumptions
\eqref{A3}--\eqref{A2''}, with    $\Phi$ given by~\eqref{trud}, namely
$
\Phi (\xi) = |\xi_1 -\xi_2|^p + |\xi_1|^q\log (c+ |\xi _1|)^\alpha$  for~$\xi \in \mathbb R^2$, with $p>1$ and either $q\geq 1$ and $\alpha >0$, or $q=1$ and $\alpha >0$. Let $\Phi_2$ be the function associated with this $\Phi$ as in \eqref{sobconj}, with $n=2$.
One has that 
\begin{itemize}
\item[{\it i)}]  {if  $pq<p+q$, then  $\Phi_2(t)\approx s^\frac{2pq}{p+q-pq}\log^\frac{p\alpha }{p+q-pq}(t) $ near infinity,}
\item[{\it ii)}]  if  $pq=p+q$ and $p\alpha<p+q$, then  $\Phi_2(t)\approx \exp\big(t^\frac{2(p+q)}{p+q-p\alpha}\big)$ near infinity,
\item[{\it iii)}] if  $pq=p\alpha=p+q$, then  $\Phi_2(t)\approx \exp(\exp(t^2))$ near infinity,
\item[{\it iv)}] if either $pq>p+q$, or $pq=p+q$ and $\alpha>q$,  then  condition \eqref{intconv} holds,
%\footnote{IC: It is not completely the same range as in A), why?}
%$\int ^\infty \frac{t}{\Phi_\circ(t)}\, dt <\infty$.
\end{itemize}
see \cite[Section 1]{Cfully}.
Thus the following conclusions   hold.
\begin{itemize}
\item[{\it A)}] Owing to Theorem \ref{theo:boundex},   problem \eqref{eq:T}   admits a unique {weak solution}
 $u$ under any of the following assumptions:
\begin{eqnarray}
% \label{hplog1}
% \hbox{$p= 1$  and $\alpha >0$},\ &
%  \text{and }  \ &
%  {\color{blue} f \in E[\exp L^{\frac 1\alpha}, n](\Omega),}
%  \int _0^{|\om|} e^{\lambda (s^{\frac 1n} f^{**}(s))^{\frac 1\alpha}} \, ds < \infty \,\,\hbox{for every $\lambda >0$},
\label{hpT1}
& \begin{cases}\hbox{either $pq<p+q$,}\\ \hbox{or $pq=p+q$ and  $\alpha \leq q$},\end{cases}
 \ &\text{and }  \qquad
 f \in L^{[\frac{2pq}{3pq-p-q}, \frac{2pq}{2pq-p-q}]}(\log L)^{-\frac {\alpha p}{2pq-p-q}}(\om),
\\\label{hpT2}
&\begin{cases}\hbox{either $pq>p+q$},\\ \hbox{or $pq=p+q$ and $\alpha > q$,}\end{cases}
\ & \text{and }  \qquad
f \in L^1(\om)\,.
\end{eqnarray}
\item[{\it B)}]  Problem \eqref{eq:T} has an {approximable solution} $u$ if  either $f \in L^1(\om)$, or $f$ is replaced by a~measure~$\mu$ with finite total variation.  In the former case, the solution is also unique. These assertions are consequences of Theorems~\ref{theo:main-f} and~\ref{theo:main-mu}.
Also,
\begin{itemize}
\item[{\it i)}] if  $pq<p+q$,  then
\begin{flalign}\label{hpT3}
% u \in L^{\vartheta(\cdot),\infty} (\om) \quad & \ \ \, \text{where }\quad {\vartheta} (t) \approx t^{\frac{2pq-p-q}{p+q-pq}}(\log t)^{\frac{\alpha p}{p+q-pq}}
%\quad \hbox{near infinity},
%\\
 u \in L^{\vartheta(\cdot),\infty} (\om) \quad &  \ \ \, \text{where }\quad {\vartheta} (t) \approx t^{\frac{pq}{p+q-pq}-1}(\log t)^{\frac{\alpha p}{p+q-pq}}
\quad \hbox{near infinity},
\\
u_{x_1}  \in L^{\varrho_1(\cdot),\infty}(\om)\quad &\quad\text{ and}\qquad u_{x_1}-u_{x_2} \in L^{\varrho_2(\cdot),\infty}(\om)    % \,, \\
%\nonumber
%{\varrho}_1 (t) \approx  t^{\frac{2pq-p-q}{p }}(\log t)^{\alpha (2 -\frac{ 1}{p})} & \quad\text{ and}\qquad {\varrho}_2 (t) \approx  t^{\frac{2pq-p-q}{q }}(\log t)^{\frac{ \alpha }{q}}\
%\quad \hbox{near infinity;}
\\ \nonumber
\hbox{where} \,\,\,{\varrho}_1 (t) \approx  t^{q(2 -\frac{ 1}{p})-1}(\log t)^{\alpha (2 -\frac{ 1}{p})} &  \quad\text{ and}\qquad {\varrho}_2 (t) \approx  t^{q(2 -\frac{ 1}{p})-1}(\log t)^{\frac{ \alpha }{q}}\
\quad \hbox{near infinity;}
\end{flalign}
\item[{\it ii)}] if
 $pq=p+q$ and $\alpha < q$, then
\begin{flalign}\label{hpT4}
u \in \exp L^{\frac{q}{q-\alpha}}(\om)\,, \quad
u_{x_1}  \in L^{\varrho_1(\cdot),\infty}(\om)\,,& \quad  \text{and}\quad u_{x_1}-u_{x_2} \in L^{\varrho_2(\cdot),\infty}(\om) \,, \\
\nonumber
\text{where }\ {\varrho}_1 (t) \approx t^{q}(\log t)^{\frac \alpha q}  & \quad\text{and}\quad {\varrho}_2 (t) \approx   t^{p}(\log t)^{\frac{ \alpha }{q}-1}
\quad \hbox{near infinity;}
\end{flalign}
\item[{\it iii)}] if $pq=p+q$ and $\alpha = q$, then
\begin{flalign}\label{hplogi3bis}
\qquad\quad\quad u \in \exp \exp L (\om)\,, \quad
u_{x_1}  \in L^{\varrho_1(\cdot),\infty}(\om)\,, &\quad\text{and}\quad u_{x_1}-u_{x_2} \in L^{\varrho_2(\cdot),\infty}(\om) \,, \\
\nonumber
\text{with }\ {\varrho}_1 (t) \approx t^{q}(\log t)^{\alpha}(\log \log t)^{-1}   & \quad\text{and}\quad {\varrho}_2 (t) \approx    t^{p}(\log \log t)^{-1}
\quad \hbox{near infinity.}
\end{flalign}
\end{itemize}
\end{itemize}
\end{ex}

\begin{ex}\rm\label{ex-aniso-new} {Assume that $\Omega\subset\R^2$, and consider any Dirichlet problem as in \eqref{eq:T},
with   $\Phi$ now given by \eqref{trud1}, namely
$
\Phi (\xi) = |\xi_1 +3\xi_2|^p + e^{|2\xi_1-\xi_2|^\beta}-1$   for $\xi \in \mathbb R^2$,  where $p>1$ and $\beta>1$. An analogous argument as in \cite[Section 1]{Cfully} shows that $$\Phi_\circ (t) \approx t^{2p}\log^{-\frac p\beta}(1+t) \quad \hbox{near infinity.}$$
 Hence, condition \eqref{intconv} is in force. Theorem \ref{theo:main-f} then tells us that there exists a unique {weak solution}  to problem \eqref{eq:T} for every $f \in L^1(\om)$.}

\end{ex}

%%%%%%%%%%%%%%%%%%%%%%%%%%%%%%%%%%%%%%%%%%%%%%%%%%%%%%%%%%%%%
{
\section{Proofs of approximation theorems} \label{approxproof}

Here, we are concerned with  proofs of the results stated in Subsection \ref{subapprox}.

\begin{proof}[Proof of Proposition~\ref{prop:conv:mod-weak}]
By our assumption, there exists $\lambda_1>0$ such that $\smallint_\Omega
\Phi((U_k- U )/\lambda_1)\,dx\to 0$ as $k \to \infty$,
namely,   $\Phi((U_k- U )/\lambda_1)\to 0$
 in $L^1(\Omega)$. Hence, there exists a subsequence of $\{U_k\}$, still indexed
  by $k$, such that $U_k\to U$ a.e. in $\Omega$,
  and the sequence of functions
$\Phi((U_k- U )/\lambda_1)$ is pointwise bounded by a
function in $L^1(\Omega)$ independent of $k$. Given any function $V
\in L^{\wt{\Phi}}(\Omega; \rn)$, there exists $\lambda_2>0$ such that
$\wt{\Phi}(V/\lambda_2)\in L^1(\Omega)$. The definition of Young's conjugate implies that
$$
\frac{|V\cdot (U_k-U)|}{\lambda_1\lambda_2}  \leq \Phi\Big(\frac{U_k-
U}{\lambda_1}\Big)+\wt{\Phi}\Big(\frac{V}{\lambda_2}\Big)\qquad\hbox{a.e. in
$\Omega$}\,.
$$
Hence, equation \eqref{july37} follows, via the dominated
convergence theorem.
\end{proof}

\begin{proof}[Proof of Proposition~\ref{prop:modulardensity}] Fix any  $U\in L^\Phi(\Omega;\rn)$. Set, for ${\ell}\in\N$,
\[\Omega_\ell=\{x\in\Omega:\ |U(x)|\leq \ell\}.\]
By Tchebyshev inequality,
$|\Omega\setminus\Omega_\ell|\leq \|U\|_{L^1(\Omega;\rn)}/\ell$.
Next, define $U_\ell=U{\chi_{\Omega_\ell}}$, and notice that
$|U_\ell(x)|\leq|U(x)|$  and
$\Phi(U_\ell(x))\leq\Phi(U(x))$ for $x\in\Omega$.
Thus, if
$\lambda\geq \|U\|_{L^\Phi(\Omega;\rn)}/2$, then
%\todo[inline]{Andrea: there was an extra $\lambda>1$ in the denominator, that seems useless to me. The previous version is commented. I have deleted it. Please, check this. Also, $l$ should be changed into $\ell$ and $\alpha $ into $\lambda$}
 \begin{equation}
\label{approx1}
\lim _{\ell \to \infty}\int_\Omega\Phi\left(\frac{U_\ell-U}{2\lambda}\right)dx= \lim _{\ell\to \infty}
\int_{\Omega\setminus\Omega_\ell}\Phi\left(\frac{U}{2\lambda}\right)dx\
%leq \lim _{l \to \infty}
%\frac{1}{\lambda}\int_{\Omega\setminus\Omega_l}\Phi\left(\frac{U}{2\alpha}\right)dx
=0.\end{equation}
%Thus, for every $\lambda>1$ and
%$\alpha\geq \|U\|_{L^\Phi(\Omega;\rn)}/2$,
%\begin{equation}
%\label{approx1}
%\lim _{l \to \infty}\int_\Omega\Phi\left(\frac{U_l-U}{2\lambda\alpha}\right)dx= \lim _{l \to \infty}
%\int_{\Omega\setminus\Omega_l}\Phi\left(\frac{U}{2\lambda\alpha}\right)dx\leq \lim _{l \to \infty}
%\frac{1}{\lambda}\int_{\Omega\setminus\Omega_l}\Phi\left(\frac{U}{2\alpha}\right)dx
%=0.\end{equation}
Let $\wt{U}_\ell$ denote the representative of the function $U_l$, which is defined everywhere in $\om$ as the limit of its averages on balls at each Lebesgue point, and by $0$ elsewhere.
%
%The function {{$U_l\in \mathcal M(\om;\rn)$ and therefore it is}} defined almost
%everywhere. We choose its representant $\wt{U}_l$ which is defined
%everywhere in $\Omega$.
Fix any $\ell,k\in\N$, and {{set}}
$Q=[-\ell,\ell]^n$. We split $Q$ into a family of $N(k)$ cubes ${Q_i^k}$
of diameter $\tfrac 1k$ defined as follows.
 Consider a dyadic  decomposition of $Q$, and distribute the boundaries of the
dyadic cubes $Q_i^k$ in such a way that they are pairwise disjoint,
% neither
%open nor closed, but they are Borel sets
and
$Q=\cup_{i=1}^{N(k)}Q_i^k$.  Define $y_i={\rm
argmin}\,\Phi|_{\overline{Q_i^k}}$ {for }$i=1,\dots,N(k)$. On setting
$E_i^k=\wt{U}_\ell^{-1}(Q_i^k)$, we have that
$\Omega=\cup_{i=1}^{N(k)}E_i^k$. Since $Q_i^k$ is a Borel set and
{$\wt{U}_\ell\in \mathcal M(\om; \rn)$}, the set $E_i^k$ is measurable.
Therefore,  the family $\{E_i^k:\,i=1,\dots,N(k)\}$ is a partition of
$\Omega$ into pairwise disjoint measurable sets. Next,   define the function ${U}_{l,k} : \om \to \rn$ as
\[ {U}_{l,k}=\sum_{i=1}^{N(k)}y_i\chi_{E_i^k}.\]
We have that $\lim _{k\to \infty}{U}_{\ell,k}(x)=\wt{U}_\ell(x)$ for every $x\in\Omega$. Indeed,
${U}_{\ell,k}(x)=y_i$ for every $x\in E_i^k$, whence  $|y_i-\wt{U}_\ell(x)|\leq{\rm
diam}\,{Q_i^k}\leq\tfrac 1k$ for every such $x$. As a consequence,
$\lim_{k\to \infty}{U}_{\ell,k}(x)={U}_\ell(x)$   for a.e.
$x\in\Omega$. On the other hand,
$\Phi({U}_{\ell,k}(x)/\lambda)=\Phi(y_i/\lambda)\leq
\Phi({U}_\ell(x)/\lambda)$ for  every $\ell , k \in \mathbb N$, and   $x\in E_i^k$. Hence, owing to Jensen's inequality,
$$\int_\Omega \Phi\left(\frac{{U}_{\ell,k}-{U}_\ell}{2\lambda}\right)dx\leq\frac{1}{2}\int_\Omega \Phi\left(\frac{{U}_{\ell,k}}{\lambda}\right)dx+\frac{1}{2}\int_\Omega \Phi\left(\frac{{U}_\ell}{\lambda}\right)dx\leq \int_\Omega \Phi\left(\frac{{U}_\ell}{\lambda}\right)dx$$
for  every $\ell , k \in \mathbb N$.
Therefore, thanks to the dominated convergence theorem,
\begin{equation}
\label{approx2}\lim _{k\to \infty} \int_\Omega
\Phi\left(\frac{{U}_{\ell,k}-{U}_\ell}{2\lambda}\right)dx=0
\end{equation}
for every $\ell \in \mathbb N$.
By the convexity of $\Phi$,
\begin{equation}\label{feb1}\int_\Omega \Phi\left(\frac{{U}_{\ell,k}-{U}}{4\lambda}\right)dx\leq\frac{1}{2}
\int_\Omega \Phi\left(\frac{{U}_{\ell,k}-{U}_\ell}{2\lambda}\right)dx+\frac{1}{2}\int_\Omega \Phi\left(\frac{{U}_\ell-{U}}{2\lambda}\right)dx
\end{equation}
for  every $\ell , k \in \mathbb N$.
Owing to equations~\eqref{approx1} and~\eqref{approx2}, the left-hand side of \eqref{feb1} tends to $0$ as $k\to \infty$.   A diagonal argument then
completes the proof.
\end{proof}

With Proposition \ref{prop:modulardensity} at our disposal, we are ready to prove Proposition \ref{prop:approx}. The proof to be presented is based on ideas of that
of~\cite[Theorem~2.2]{pgisazg1}.

\begin{proof}[Proof of Proposition \ref{prop:approx}]
Assume, for the time being, that $\Omega$ is starshaped with respect
to the ball $B_r(0)$, centered at $0$ and with radius $r$. This
means that $\om$ is starshaped with respect to every point in~$B_r(0)$. Let  $k\in\N$ be so large that  $\tfrac 1k \in (0, \tfrac r4)$, and  set $\gamma_k
=1-\tfrac{2}{rk}<1$. For any such $k$, we define
the~set
\begin{equation}
\label{july30} \Omega_k =\gamma_k \Omega + \tfrac 1k B_1(0).
%\subset
%\subset \Omega.
\end{equation}
Our choice of $k$ and $\gamma _k$ ensures that $\Omega_k  \subset
\subset \Omega$.
Let $m \in \mathbb N$, and let $U \in\mathcal{M}(\rn; \R^m)$ be such that $U=0$ in $\rn \setminus \om$. Define
$U_k  : \om \to \mathbb R^m$ as
\begin{equation}
\label{xid}U_k(x) = \int_\rn \rho_k( x-y)U (y/\gamma_k)\,dy \qquad
\hbox{for $x \in \om$,}
\end{equation}
where $\rho_k(x)=\rho(kx)k^n$ is a standard smoothing kernel on
$\rn$,   i.e. $\rho $ is a nonnegative radially decreasing function,
$\rho\in C^\infty(\rn)$, $\mathrm{supp}\,\rho\subset\subset B_1(0)$ and
$\smallint _\rn \rho(x)dx = 1$. Since $U (y/\gamma_k) =0$ if $y\notin \gamma_k
\om$, %by \eqref{july30}
one has that $U_k\in C_0^\infty(\om; \Rm)$.
Moreover, if $U \in L^\infty (\om; \Rm)$, then
\begin{equation}
\label{july43} \|U_k \|_{L^\infty(\om;\Rm)}\leq
\|U\|_{L^\infty(\om;\Rm)}.
\end{equation}
\\ We claim that, if  $m=n$, then
\begin{equation}
\label{unifMxid}\iO \Phi(U_k )\, dx\leq \iO \Phi(U)\, dx
\end{equation}
for $k$ as above. Indeed,
 \[
\begin{split}
\iO \Phi(U_k(x))\, dx &= \int _\rn  \   {\Phi\left(\int_\rn \rho_k(
x-y)U (y/\gamma_k)\,dy\right)}\, dx \leq
 \int_\rn \int_{\rn} \rho_k(x-y) {{\Phi}\left( U (y/\gamma_k )  \right)} \,dy\,dx\\
 & =
 \int_\rn {{\Phi}\left( U (y/\gamma_k )  \right)} \int_{\rn} \rho_k(x-y)  \,dx\,dy = \gamma_k^n \int_{\rn} {{\Phi}\left( U (z) \right)}
dz =  \gamma_k^n \int_{\om} {{\Phi}\left( U (z) \right)} dz\\
 &\leq \int_{\om} {{\Phi}\left( U (z) \right)} \,
dz\,,\end{split}
\]
where the first equality holds since  $U_k =0$ in $\rn \setminus
\om$ and $\Phi(0)=0$, the inequality follows from Jensen's
inequality, and the third equality is due to the fact that
$\smallint_{\rn} \rho_k(x-y)  \,dx=1$ for every $y \in \rn$.
\\
%{\bf Step 2.} Let $\om$ be as in step 1, and let  $u\in W^{1}_0L^\Phi(\Omega)$.
Assume now that $u\in W^{1}_0L^\Phi(\Omega)$. As observed above, the function $u_k$, defined as in
\eqref{xid}, belongs to~$C_0^\infty (\om)$. Moreover,
 since the continuation of $u$  to $\rn$ by $0$ outside $\om$ is weakly differentiable in $\rn$, the function $\om  \ni x \mapsto u(x/\gamma_k)$ is weakly differentiable in $\om$. Thus,
 \begin{equation}
 \label{july31}
 (\nabla u)_{k} = \nabla u_k \quad \hbox{in $\om$,}
 \end{equation}
where $(\nabla u)_{k}$ is defined as in \eqref{xid}, with $U=\nabla u$.
We shall show that there exists $\lambda >0$ such that
\begin{equation}
\label{july32} \lim _{k \to \infty}\int_\Omega
\Phi\left(\frac{\nabla u_k - \nabla u}{\lambda}\right) dx = 0\,.
\end{equation}
Owing to \eqref{july31}, equation \eqref{july32} will follow if we
prove that
\begin{equation}
\label{july33} \lim _{k \to \infty}\int_\Omega
\Phi\left(\frac{(\nabla u)_k - \nabla u}{\lambda}\right) dx =
0\qquad\text{ for some }\ \lambda >0.
\end{equation}
Fix any $\sigma>0$. By Propositions~\ref{prop:modulardensity}, there
exists a simple function
 $V : \om \to \rn$ such that
    \begin{equation}\label{IE:aw15}
    \int_\Omega \Phi\left( \frac{\nabla u -V }{\frac{1}{3}\lambda} \right) dx <
    \sigma.
    \end{equation}
    The convexity of $\Phi$ ensures that
    \begin{align}\label{IE:aw14}
    &\int_\Omega
     \Phi\left( \frac{ (\nabla u)_{k } - \nabla u }{ \lambda }\right) \,dx  = \int_\Omega
    \Phi\left( \frac{ (\nabla u)_{k } -V_{k }
    +  V_{k } -V + V - \nabla u}{ \lambda }\right) \,dx\\
    &\  \leq
    \frac{1}{3} \int_\Omega \Phi\left(
    \frac{ (\nabla u)_{k } -  V_k   }{ \frac{1}{3} \lambda } \right) \,dx
    + \frac{1}{3} \int_\Omega \Phi\left(
    \frac{  V_{k } -V }{
    \frac{1}{3} \lambda } \right) \,dx   + \frac{1}{3} \int_\Omega \Phi\left(  \frac{ V - \nabla u }{\frac{1}{3}\lambda} \right) \,dx .\nonumber
    \end{align}
     By \eqref{unifMxid} and \eqref{IE:aw15},
    \begin{equation}\label{IE:aw16}
    \int_\Omega \Phi\left(
    \frac{ (\nabla u)_k  - V_k   }{ \frac{1}{3} \lambda } \right) \,dx = \int_\Omega \Phi\left(
    \frac{ (\nabla u -V  )_k   }{ \frac{1}{3} \lambda } \right) \,dx <
    \sigma.
    \end{equation}
On the other hand, owing to Jensen's  inequality and  Fubini's
theorem
    \begin{align}\label{IE:aw17}
    \int_\Omega\Phi\bigg(
    \frac{  V_{k } -V}{
    \frac{1}{3} \lambda } \bigg) \,dx
    & = \int_\Omega \Phi\bigg( \frac{3}{ \lambda} \int_{B_1(0)} \rho (y) \big(V((x-  y/k)/\gamma_k)
  - V(x)\big)\,dy \bigg) \,dx
    \\ &
    \leq
    \int_{B_1(0)} \rho (y)  \int_\Omega \Phi\bigg( \frac{3}{\lambda}
\big(V((x- y/k)/\gamma_k)
  - V(x)\big)\bigg) \,dx   \,dy.\nonumber
    \end{align}
Therefore
$$\lim _{k \to \infty}  \Phi\bigg( \frac{3}{\lambda}
\big(V((x-y/k)/\gamma_k)
  - V(x)\big)\bigg) = 0 \quad \hbox{for {a.e.} $x \in \om$ and every $y \in B_1(0)$.}$$
Moreover,
$$  \Phi\bigg( \frac{3}{\lambda}
\big(V((x-y/k)/\gamma_k)
  - V(x)\big)\bigg)\leq C$$
for some constant $C$, and for every $x\in \om$, $y \in B_1(0)$ and
$k$ such that $\tfrac 1k \in (0, \tfrac  r4)$. Hence, by the~dominated convergence theorem,
$$\lim _{k \to \infty}
 \int_\Omega\Phi\bigg( \frac{3}{\lambda}
\big(V((x-y/k)/\gamma_k)
  - V(x)\big)\bigg)\,dx =0\qquad\text{for every }\ y \in B_1(0).$$
Furthermore,
$$
 \int_\Omega \Phi\bigg( \frac{3}{\lambda}
\big(V((x-y/k)/\gamma_k)
  - V(x)\big)\bigg) \,dx \leq C|\om|$$
for every  $y \in B_1(0)$ and every $k$ such that $\tfrac 1k \in (0,
\tfrac  r4)$.
Consequently, the rightmost side   of~\eqref{IE:aw17} converges
to zero as $k \to \infty$, thanks to the dominated convergence
theorem again, whence
\begin{equation}\label{july35}
\lim _{k \to \infty}  \int_\Omega\Phi\left(
    \frac{  V_{k} -V}{
    \frac{1}{3} \lambda } \right) \,dx  =0\,.
\end{equation}
Inequality \eqref{july33} follows from \eqref{IE:aw14},
\eqref{IE:aw15}, \eqref{IE:aw16} and \eqref{july35}, owing to the
arbitrariness of $\sigma$. This completes the proof in the case when $\om$ is  a starshaped
domain.
\\
%{\bf Step 3.}
Assume now that $\Omega$ is any bounded Lipschitz domain in~$\rn$.
Then, there exists a finite family of open sets $\omega_1, \dots
\omega _J$ and a corresponding family of balls $B_1, \dots B _J$,
with radii $r_1, \dots r_J$, such that $\Omega=\cup
_{k=1}^J\omega_j$, and every set $\omega_j$ is starshaped with
respect to the ball $B_j$. Let us  introduce a partition of unity
$\theta_j$ subordinated to the family $\{\omega _j\}$.
%$0\le\theta_i\le 1,\, \theta_i\in C^\infty_c(\Omega_i), \,{\rm supp} \,\theta_i=\Omega_i, \sum_{i\in I}\theta_i(\x)=1$
Any function $u \in W^1_0L^\Phi(\om)$ admits the decomposition
    \begin{equation}
    \label{july44}
    u(x) = \sum_{j=1}^J \theta_j (x) u(x) \quad \hbox{for $x \in \om$.}
    \end{equation}
Since $\nabla u \in L^\Phi(\Omega ; \rn)$ and $u \in L^\infty(\Omega)$,
one has that  $\nabla (\theta_j u) = (u \nabla \theta_j +
\theta_j\nabla u) \in L^\Phi(\Omega{;\rn})$. Therefore, $\theta _j u \in
W^1_0L^\Phi (\omega _j)$. Property \eqref{angela1grad} then follows
on applying to each function  $\theta_j u$ the result for  domains starshaped with respect to
balls.
\\ Inequality \eqref{july36} is a consequence of inequality \eqref{july43} and of the representation  formula \eqref{july44}\,.
\\ As far as property \eqref{sep30} is concerned, choose any $\lambda >0$ such that
\begin{align}\label{sep31}
\lim _{k \to \infty}\int _\om \Phi\bigg(\frac {\nabla u_k - \nabla
u}{\lambda}\bigg)\, dx =0\,.
\end{align}
By inequality \eqref{anisopoinc},
\begin{align}\label{sep32}
\int _\om \Phi_\circ \bigg(\frac {\kappa_1|u_k - u|}{|\om|^{\frac
1n}\lambda}\bigg)\, dx \leq  \int _\om \Phi\bigg(\frac {\nabla u_k -
\nabla u}{\lambda}\bigg)\, dx
\end{align}
for every $k\in \mathbb N$. From \eqref{sep31} and an application of Jensen's inequality to the integral on the left-hand side of inequality \eqref{sep32} we infer that $u_k \to u$ in $L^1(\om)$. Hence,  equation \eqref{sep30} follows, on  taking a~subsequence if necessary.
\end{proof}}

%%%%%%%%%%%%%%%%%%%%%%%%%%%%%%%%%%%%%%%%%%%%%%%%%%%%%%%%%%%%%

\section{Weak solutions: proof of Theorem~\ref{theo:boundex} }

The present section is split into subsections, corresponding to subsequent steps towards a proof of Theorem~\ref{theo:boundex}.
%
%
%As a preliminary step in view of the proof of Theorem
%\ref{theo:boundex}, we consider a family of  approximating problems
%satisfying isotropic ellipticity and growth conditions.
\subsection{Regularized problems}
We begin by constructing a sequence of   problems approximating \eqref{eq:main:f}, and whose principal part
satisfies isotropic ellipticity and growth conditions.
\\
Let $A: [0, \infty) \to [0,\infty)$
 be a strictly convex $N$-function such that $A\in C^1([0, \infty))$. In particular, $A'(0)=0$. Hence, the function
$$\rn \ni \xi \mapsto A(|\xi|) \in [0, \infty)$$
is a continuously differentiable radially increasing $n$-dimensional
$N$-function, whose gradient
%vanishes at $0$ and
agrees with
$A'(|\xi|)\frac{\xi}{|\xi|}$ for $\xi \in \rn$, with the convention   that the latter expression has to be interpreted as $0$ when $\xi=0$. The equality case in
Young's inequality yields
\begin{equation}\label{FYeq}
tA'(t) =    {A}( t)+\wt{A} (A'(t)) \quad \hbox{for $t \geq 0$.}
\end{equation}
Moreover, since $A$ is strictly convex,
\begin{equation}\label{msmon}
\bigg(A'(|\xi|)\frac{\xi}{|\xi|} -
A'(|\eta|)\frac{\eta}{|\eta|}\bigg)\cdot(\xi-\eta)>0\qquad \hbox{for
every $\xi \neq \eta$.}
\end{equation}
Given $\ep \in (0,1)$  we  define $a^\ep : \om \times \rn \to
\mathbb R$ by
\begin{equation}\label{Atheta-e}
a^\ep(x,\xi)=a(x,\xi)+ \ep A'(|\xi|)\frac{\xi}{|\xi|} \qquad
\hbox{for $x\in \om$ and $\xi \in \rn$, }
\end{equation}
and consider the problem
\begin{equation}\label{feb10}%\label{eq:reg-bound-e}
\left\{\begin{array}{ll}
-\dv \, a^\ep(x,\nabla u^\ep) = f & \ \mathrm{ in}\  \Omega \\
u^\ep =0 &\ \mathrm{  on} \ \partial\Omega\,.
\end{array}\right.
\end{equation}
We shall show that  the function  $a^\ep(x, \cdot)$ satisfies isotropic ellipticity
and growth conditions,   that allow  to make use  of
an existence theory available in the literature. A priori estimates for $u^\ep$, independent of $\ep \in (0,1)$, will then be derived.
{
\begin{proposition}{\rm {\bf [Existence of solutions to regularized problems]}}\label{prop:reg-bound-e} Let $\Omega$ be a bounded Lipschitz domain in $\rn$. Assume that   $a:\Omega\times\rn\to\rn$  is a Carath\'eodory function satisfying assumptions~\eqref{A3}--\eqref{A2''} for some $n$-dimensional $N$-function $\Phi$.  Let $A(t)$ be any continuously differentiable strictly convex  $N$-function in $[0, \infty)$ that grows essentially faster than $t^q$ near infinity for some $q>n$, and such that
\begin{equation}\label{gamma>Phi}
A(|\xi|) \geq \Phi(\xi) \qquad \hbox{for $\xi \in \rn$.}
% \hbox{$\mathcal A$\, dominates \, $\Phi$\, globally.}
\end{equation}
%Assume, in addition, that $A$ grows essentially faster than another $N$-function $B$ fulfilling condition \eqref{convA}.
%such that
%\begin{equation}
%\label{intconviso} \int^\infty\bigg(\frac t{B(t)}\bigg)^{\frac
%1{n-1}}\, dt < \infty\,.
%\end{equation}
Let $\ep \in (0,1)$ and let $a^\ep$ be defined  as in~\eqref{Atheta-e}. If $f \in L^1(\om)$,
then   there exists a  weak solution  $u^\ep \in   W_{0}^{1}\mathcal L^A(\Omega)\cap L^\infty(\om)$ to problem
\eqref{feb10}.
\end{proposition}}

The following function spaces will  come into play in the proof of Proposition \ref{prop:reg-bound-e}.
 Let us denote by~$\mathcal W^1_0 L^A(\Omega)$  the closure of $C^\infty_0(\Omega)$ in
$W^1L^A(\Omega)$ with respect to the weak topology  $\sigma \big(L^A\times L^A, E^{\widetilde{A}}\times
E^{\widetilde{A}}\big)$.
\iffalse

$\sigma \big(L^A\times L^A, E^{\widetilde{A}}\times
E^{\widetilde{A}}\big)$-closure of $C^\infty_0(\Omega)$ in
$W^1L^A(\Omega)$. %The space $\mathcal W^1_0 E^A(\Omega)$ is the closure of $\mathcal {D}(\Omega)$ in $W^1L^A(\Omega)$ with respect to the norm \eqref{sobolev-norm}, and then it is a Banach space (because it is a closed subspace of $W^1E^A(\Omega)$).
Here, $\sigma\big(L^A\times L^A, E^{\widetilde{A}}\times
E^{\widetilde{A}}\big)$ stands for the weak topology in
$\big(L^A\times L^A, E^{\widetilde{A}}\times
E^{\widetilde{A}}\big)$.

\fi
One has that
\begin{equation}\label{A5}
\mathcal W^1_0 L^A(\Omega) \subset{W}^1_0 L^A(\Omega)\,,
\end{equation}
see \cite{Gossez}.
Moreover, we shall consider the space of distributions defined as
\begin{equation}\label{W-1E}
{\mathcal W}^{-1}E^{\wt{A}}(\Omega)=\bigg\{ f\in \mathcal{D}'(\Omega):
f=f_0-\sum_{i=1}^n \frac{\partial f_i}{\partial x_i} ,\; \; f_i\in
E^{\wt{A}}(\Omega),\, i=0, \dots n \bigg\}\,.
\end{equation}

\begin{proof}[Proof of Proposition \ref{prop:reg-bound-e}]
We begin by showing that, under condition \eqref{gamma>Phi}, the function $a^\ep$ fulfills the~assumptions required  in~\cite[Section~5]{Gossez2}.
% We begin by showing that, under condition \eqref{gamma>Phi}, the function $a^\ep$ fulfills assumptions of  \cite[Theorem~5.1]{Mustonen}.
% \todo{Separability the spaces $Y_0$ and $Z_0$ is also needed. This seems to be true, due to the definition of the norm in $W^{-1}$ and the separability of $E^{\widetilde A}$. Alternatively, cite also \cite{Gossez2}}
 Besides being  a Carath\'eodory's function,
 those assumptions   on $a^\ep$ amount to a
 monotonicity condition that immediately follows from \eqref{msmon} and~\eqref{A3}, and to~an~estimate of the form
\begin{equation}\label{june1}
\left|a^\ep(x,\xi)\right|  \leq c \wt{A}^{-1}\left(c
{A}\left(\left|c\,\xi\right|\right)\right)+ c  \wt{A}^{-1}(c\, h(x))\qquad
\hbox{for  a.e. $x\in\Omega$ and for $\xi \in \rn$,}
\end{equation}
for some positive constant $c$.
% and $h\in E^{\wt{A}}(\Omega)$.
To
verify inequality \eqref{june1}, observe that, by inequality \eqref{young},
\begin{equation} \label{*}a^\ep(x,\xi)\cdot \xi\leq  {A}\left(\left|\frac{2}{c_\Phi}\xi\right|\right)+\wt{A}\left(\left|\frac{c_\Phi}{2}a^\ep(x,\xi)\right|\right)
\end{equation}
% \eqref{FYeq}, and the fact that $c_\Phi \in(0,1]$ and $\ep\in(0,1]$,
%  \begin{equation} \label{*}a^\ep(x,\xi)\cdot \xi\leq  {A}\left(\left|\frac{2}{c_\Phi}\xi\right|\right)+\wt{A}\left(\left|\frac{c_\Phi}{2}a^\ep(x,\xi)\right|\right)\leq  {A}\left(\left|\frac{2}{c_\Phi}\xi\right|\right)+c_\Phi \wt{ {A}}\left(\left|\frac{1}{2}a^\ep(x,\xi)\right|\right)\,
% \end{equation}
for  a.e. $x\in\Omega$ and for $\xi \in \rn$. Inequality \eqref{gamma>Phi} implies that $\widetilde A(|\xi|) \leq \widetilde \Phi (\xi)$ for $\xi \in \rn$. Hence, via
inequalities~\eqref{A2'} and  \eqref{A2''},
\begin{align} \label{**}
 a^\ep(x,\xi)\cdot \xi & \geq \Phi \left( \xi\right)+   \ep  {A}\left(\left|\xi\right|\right)+\ep \wt{A} \left(A'(|\xi|) \right) \geq  \wt{\Phi}\left(c_\Phi a(x,\xi) \right) + \wt{A} \left(\ep A' (|\xi|)\right) -h(x)
 \\
 \nonumber  & \geq 2 \left(\frac{1}{2}\wt{A} \left(c_\Phi \left|a(x,\xi)\right|\right) +\frac{1}{2}\wt{A} \left(c_\Phi \ep A'(|\xi|) \right) \right) -h(x) \geq 2  \wt{A}\left(\frac{c_\Phi}{2} \left|a^\ep(x,\xi)\right| \right)-h(x)\,
 \end{align}
%  \begin{equation} \label{**}
%  \begin{split}
%  a^\ep(x,\xi)\cdot \xi &\geq \Phi \left( \xi\right)+   \ep  {A}\left(\left|\xi\right|\right)+\ep \wt{A} \left(A'(|\xi|) \right) \geq c_\Phi \wt{\Phi}\left(a(x,\xi) \right) + \wt{A} \left(\ep A' (|\xi|)\right) -h(x) \\
%  &\geq 2 c_\Phi\left(\frac{1}{2}\wt{A} \left(\left|a(x,\xi)\right|\right) +\frac{1}{2}\wt{A} \left( \ep A'(|\xi|) \right) \right) -h(x) \geq 2 c_\Phi \wt{A}\left(\frac{1}{2} \left|a^\ep(x,\xi)\right| \right)-h(x)\,
%  \end{split}
%  \end{equation}
for  a.e. $x\in\Omega$ and for $\xi \in \rn$
Combining inequalities \eqref{*} and \eqref{**} tells us that
$$  \wt{A} \left(\frac{c_\Phi}{2} \left|a^\ep(x,\xi)\right| \right) \leq  {A}\left(\left|\frac{2}{c_\Phi}\xi\right|\right)+h(x) \,$$
for  a.e. $x\in\Omega$ and for $\xi \in \rn$.
Therefore, thanks to the monotonicity of the function $\wt{A}^{-1}$, we obtain that
\begin{align}
\left|a^\ep(x,\xi)\right| &\leq \frac{2}{c_\Phi}\wt{A}^{-1}\left({A}\left(\left|\frac{2}{c_\Phi}\xi\right|\right)+ h(x)\right)
\leq
\frac{2}{c_\Phi}\wt{A}^{-1}\left(2{A}\left(\left|\frac{2}{c_\Phi}\xi\right|\right)\right) + \frac{2}{c_\Phi}\wt{A}^{-1}\left( 2h(x)\right)
\end{align}
% \[\begin{split}
% \left|a^\ep(x,\xi)\right| &\leq 2\wt{A}^{-1}\left(\frac{1}{c_\Phi}{A}\left(\left|\frac{2}{c_\Phi}\xi\right|\right)+\frac{1}{c_\Phi}h(x)\right)\\
% &\leq
% 2\wt{A}^{-1}\left(\frac{2}{c_\Phi}{A}\left(\left|\frac{2}{c_\Phi}\xi\right|\right)\right)+2\wt{A}^{-1}\left(\frac{2}{c_\Phi}h(x)\right)
% \end{split}\]
for  a.e. $x\in\Omega$ and for $\xi \in \rn$.
Hence, \eqref{june1} follows.
\\ {
%Let $B$ be the $N$-function given by
%$$B(t) = A(t^2) \quad \hbox{for $t \geq 0$.}$$
%Since $A$ fulfills condition \eqref{convA}, the same condition is also satisfied by $B$.
Now, since $q>n$, we have that $q'<n'$.   Then there exists a function $F \in L^ {q'}(\om;\mathbb R^n)$, with $F=(F_1,\dots , F_n)$, such that
\begin{equation}\label{div1}
{\rm div} \,F = f - f_\om \qquad \hbox{in $\om$,}
\end{equation}
where $f_\om = \tfrac 1{|\om|}\smallint_\om f(x)\, dx$, the mean value of $f$ over $\om$.
This  follows, for instance, from the use of~the~Bogowskii operator, and the boundedness of the latter  from $L^1(\om)$ into $L^{q'}(\om)$ -- see \cite{Bogovskii}.
Inasmuch as $A(t)$ grows essentially faster than $t^q$ near infinity, the function  $t^{q'}$ grows essentially faster than $\widetilde A (t)$ near infinity.
Thus, $L^ {q'}(\om) \subset E^{\widetilde A}(\om )$, and hence $f$ is a distribution of the form
$f = f_\om-\sum_{i=1}^n \frac{\partial F_i}{\partial x_i}$ with $F_i \in  E^{\widetilde A}(\om )$ for $i=1, \dots, n$.
Therefore, $f \in \mathcal W^{{-1}}E^{\widetilde A}(\om)$.
As a consequence, the results in~\cite[Section~5]{Gossez2} ensure that   there exists    a function  $u^\ep \in {\mathcal
W}_{0}^{1}L^A(\Omega)$ such that
%fulfilling the definition of~weak solution to problem
 %\eqref{eq:reg-bound-e} for any test function $\vp\in {\mathcal W}_0^{1}L^{A}(\Omega)$. Namely, %\color{black} \todo{This should be included in the result of Gossez {\color{teal}\\ It is!}}
 $a^\ep(x,\nabla u^\ep) \in L^{\widetilde A}(\om)$} and
\begin{equation} \label{weak-reg-bound-ebis}
\int_\Omega a^\ep(x,\nabla u^\ep)\cdot \nabla \vp\, dx= \int_\Omega
f\,\vp\, dx
\end{equation}
for every $\vp\in {\mathcal W}_0^{1}L^{A}(\Omega)$. By \eqref{A5}, $u^\ep \in W_{0}^{1}L^A(\Omega)$.
Moreover, an inspection of the proof of~\cite[Sec\-tion~5]{Gossez2} reveals that $\smallint_{\om}A(\nabla u^\ep)\, dx < \infty$, whence $u^\ep \in W_{0}^{1}\mathcal L^A(\Omega)$.
Since   the function $A(t)$ grows faster than $t^q$ near infinity, one has that $W^1_0L^A(\om) \to W^{1,q}_0(\om) \to L^\infty (\om)$, and hence $u^\ep \in L^\infty(\Omega)$ as~well.
%
%  fulfills the same condition. By the Sobolev inequality \eqref{B-Wter}, applied in the isotropic case when $\Phi (\xi) =A(|\xi|) = \Phi_\circ (|\xi|)$, we have that $u^\ep \in L^\infty(\Omega)$ as well.
\\ It remains to show that equation \eqref{weak-reg-bound-ebis}  holds not only for $\vp\in {\mathcal W}_0^{1}L^{A}(\Omega)$, but also
for every $\vp \in W^1_0L^A(\om)$, a space containing
 $W^1_0\mathcal L^A(\om)$. Fix any such $\vp$ and observe that, by the embedding mentioned above, one has in fact that $\vp \in L^\infty (\om)$ as~well. An application of Proposition \ref{prop:approx}  (in the isotropic case)   ensures that there exists a sequence  $\{\vp_k\} \subset C^\infty_0(\om)$ such that $\vp _k \to \vp$ a.e. in $\om$, $\|\vp_k\|_{L^\infty(\om)} \leq C \|\vp\|_{L^\infty(\om)}$ for some constant $C$ and for every $k \in \mathbb N$, and $\nabla \vp_k \to \nabla \vp$ modularly in $L^A(\om)$. Since $a^\ep(x,\nabla u^\ep) \in L^{\widetilde A}(\om)$, by Proposition \ref{prop:conv:mod-weak} and the dominated convergence theorem one can pass to the limit in equation \eqref{weak-reg-bound-ebis} applied with $\vp$ replaced by $\vp_k$, and infer that equation \eqref{weak-reg-bound-ebis} holds for $\vp$ as~well. This fact amounts to saying that $u^\ep$ is actually a weak solution to problem \eqref{feb10}.
%
 %Inasmuch as equation \eqref{weak-reg-bound-ebis} holds for every $\vp \in W^1_0L^A(\om) \cap L^\infty (\om)=W^1_0L^A(\om)$, the function $u^\ep$ is actually a weak solution to problem \eqref{feb10}.
%\\ The assumption  on the right-hand side $f$ requested in \cite[Theorem~5.1]{Mustonen} is exactly \eqref{june2}. The existence of the solution $u^\ep$ thus follows from that theorem.
%\\ The fact that $u^\ep \in L^\infty (\om)$ under assumption \eqref{intconviso} is a consequence of the Sobolev inequality \eqref{B-Wter}.
\end{proof}

A priori bounds for the solution $u^\ep$ to problem \eqref{feb10}, independent of $\ep \in (0, 1)$,  are established in~Proposition~\ref{prop:uniform1} below. They
are critical  in obtaining a weak solution to problem~\eqref{eq:main:f}  as the limit of $u^\ep$ as $\ep \to 0^+$.

\begin{proposition}{\rm {\bf [Uniform estimates in approximating problems]}}\label{prop:uniform1}
 Let $\om$, $a$, $\Phi$ and $A$ be as in Proposition \ref{prop:reg-bound-e}.
Suppose that   $f$ satisfies either of assumptions \eqref{hpf} and
\eqref{hpfbis}.
Given $\ep \in (0, 1)$, let  $u^\ep$ be a weak solution  to problem \eqref{feb10} exhibited  in Proposition \ref{prop:reg-bound-e}. Then:
\begin{itemize}
\item[(i)] the family $\{u^\ep\}$ is uniformly bounded in $W^1_0L^\Phi(\Omega)$,
\item[(ii)] the family $\{\ep \wt{A}(A'(|{\nabla} u^\ep|))\}$ is uniformly bounded in $L^1(\Omega)$,
\item[(iii)] the family $\{a(x,\nabla u^\ep)\}$ is uniformly bounded in $L^{\wt{\Phi}}(\Omega ;\rn)$.
\end{itemize}
\end{proposition}

\begin{proof} We shall make use of a comparison principle estanblished in \cite{Ci-sym}, that links the solution $u^\ep$ to~the~solution  $v$ to the  symmetrized problem
\begin{equation}\label{symmetrized}
\left\{\begin{array}{ll}-\dv \left(\displaystyle \frac{\Phi_\Diamond(|\nabla v|)}{|\nabla v|^2}\nabla v\right)=
f^\bigstar (x)&\text{in }\Omega^\bigstar\\
v=0&\text{on }\partial\Omega^\bigstar\,,
\end{array}\right.
\end{equation}
where $\Phi_\Diamond$ is defined in \eqref{phidiamond}, $\om
^\bigstar$ denotes the open ball centered at the origin such that $|\om
^\bigstar|=|\om|$, and $f^\bigstar$ stands for the radially
decreasing symmetral of $f$. Recall that $f^\bigstar (x)= f^*(\omega _n |x|^n)$ for $x \in \Omega ^\bigstar$, where $\omega _n$ denotes  Lebesgue measure of the unit ball in $\rn$.
%{\color{blue} Needed comment that
%under assumptions \eqref{hpf} or \eqref{hpfbis} we can
%apply~\cite[Theorem~1]{Ci-sym}.}
%\todo{Correct reference and explanation of decreasing symmetral
%}
According to~\cite[Theorem~3.1]{Ci-sym}, our alternate assumptions \eqref{hpf} or \eqref{hpfbis} on $f$ ensure that
 problem \eqref{symmetrized} actually admits a weak solution $v$, given by
\begin{equation}\label{june23}
v(x) = \int^{|\Omega|}_{\omega _n |x|^n}\frac{1}{n\omega
_n^{1/n}s^{1/n^\prime}}\Psi _\Diamond
^{-1}\bigg(\frac{s^{1/n}}{n\omega _n^{1/n}}f^{**}(s)\bigg)\,ds \quad
{\rm for} \,\,\, x \in \Omega ^\bigstar \,,
\end{equation}
where $\Psi_\Diamond$ is the function defined as in \eqref{Psi}.
Indeed, by \eqref{june23},
\begin{equation}\label{june24}
|\nabla v(x)| = \Psi _\Diamond ^{-1}\bigg(\frac{(\omega
_n|x|^n)^{1/n}}{n\omega _n^{1/n}}f^{**}(\omega _n|x|^n)\bigg) \quad
{\rm for \,\, a.e. } \,\,\, x \in \Omega ^\bigstar \,.
\end{equation}
Thus
\begin{equation}\label{june25}
\int _{\Omega ^\bigstar} G(|\nabla v|)\,dx = \int_0^{|\om|}G
\bigg(\Psi _\Diamond ^{-1}\bigg(\frac{s^{1/n}}{n\omega
_n^{1/n}}f^{**}(s)\bigg)\bigg)\, ds\,
\end{equation}
for every continuous function $G : [0, \infty) \to [0, \infty)$.
Equation \eqref{june25}, with $G = \Phi _\Diamond$, combined with property {\emph {(v)}} of Lemma \ref{lem:aux-aniso-est}
and \eqref{equivdiamond}, tells us that
\begin{equation}\label{june27}
\int _{\Omega ^\bigstar} \Phi _\Diamond(|\nabla v|)\,dx \leq
\int_0^{|\om|}\widetilde{\Phi_\circ}\big(c s^{1/n} f^{**}(s)\big)\,
ds
\end{equation}
for some constant $c$ depending on $n$. If  \eqref{hpf}  is in
force, then  the last integral converges, owing to~the~very
definition of the space $E[\widetilde{\Phi_\circ}, n](\om)$. Suppose
that, instead, \eqref{hpfbis} holds. Then, owing to the inequality
$$f^{**}(s) \leq \frac 1s \int_0^{|\Omega|}f^*(r)\, dr = \frac 1s\|f\|_{L^1(\om)} \quad  \hbox{for $s\in (0, |\om|)$,}$$
the convergence of the integral on the right-hand side of inequality
\eqref{june27} is a consequence of the fact that
\begin{equation}\label{oct1}
\int _0 \widetilde{\Phi _\circ}\big(\lambda
s^{-\frac 1{n'}}\big)\, ds < \infty \quad \hbox{for every $\lambda
>0$.}
\end{equation}
Indeed,  a change of variables in the integral in \eqref{oct1} tells
us that the latter condition can be rewritten as
\begin{equation}\label{oct1'}\int ^\infty \frac{\widetilde{\Phi _\circ} (t)}{t^{n' +1}}\, dt < \infty\,,
\end{equation}
%The latter is in its (IC: this was a repetition)
which turns to be equivalent -- see \cite[Lemma
4.1]{cianchi_ibero} -- to condition \eqref{intconv} appearing in \eqref{hpfbis}. Altogether, we have shown that
\begin{equation}
\label{june28} \int _{\Omega ^\bigstar} \Phi _\Diamond(|\nabla
v|)\,dx < \infty\,
\end{equation}
under  either assumption \eqref{hpf} or \eqref{hpfbis}. This implies that  $v \in W^1_0\mathcal L^{\Phi _\Diamond}(\om ^\bigstar)$, and hence it is indeed a~weak solution to problem \eqref{symmetrized}.
\\
 On making use of the solution
   $u^\ep$ as a test function in the weak formulation of problem \eqref{feb10},
%\eqref{weak-reg-bound-e},
and recalling assumption \eqref{A2'} we deduce that
\begin{equation}
\label{en:eq:imp-e}
 \int_{\Omega } \Phi(\nabla u ^\ep) \,dx+\int_{\Omega }\ep A(|\nabla u ^\ep|)\, dx+\int_{\Omega }\ep \wt{A}(A'(|\nabla u ^\ep| ))\,dx   \leq  \int_{\Omega } f u ^\ep\, dx .
\end{equation}
In particular, inequality \eqref{en:eq:imp-e} ensures that $u^\ep \in W^1_0\mathcal L^\Phi(\om)$, and hence \cite[Theorem 3.1]{Ci-sym}  can be exploited to infer that
\begin{equation}
\label{june29} (u^\ep)^*(s) \leq v^*(s) \quad \hbox{for $s \in (0,
|\om|)$.}
\end{equation}
We now distinguish between  the cases when  either assumption \eqref{hpf} or   \eqref{hpfbis} holds.
 \\ Assume first that  condition \eqref{hpf} is
in force.
Let us replace, if necessary, $\Phi_\circ$ in the definition of
$\widehat{\Phi_\circ}$ in~\eqref{Bphi} by another Young function
$\Phi_\bullet$ fulfilling condition \eqref{conv0} and such that  $\Phi_\bullet(t) = \Phi_\circ(t)$ if $t
\geq 1$. For instance, one can
define $\Phi_\bullet$ in such a way that it is linear in $[0,1]$.
Therefore, there exists  a constant $t_1>0$ such that $\widetilde
{\Phi_\bullet}(t) = \widetilde {\Phi_\circ}(t)$ if $t\geq t_1$.
Denote by $\widehat {\Phi_\bullet}$ the function  defined as in
\eqref{Bphi} and \eqref{2,-2}, with $\Phi_\circ$ replaced by
$\Phi_\bullet$.
 Let $\lambda$ be a positive number to be fixed later. By
inequality \eqref{holderint}, with $\Phi_\circ$ replaced by
$\Phi_\bullet$,
\begin{align}
\label{june30} \int_{\Omega } f u ^\ep\, dx & \leq C
\bigg(\int_0^{|\om|}\widetilde {\Phi_\bullet}\big(\lambda s^{\frac
1n}f^{**}(s)\big)\, ds + \int_0^{|\om|} \widehat {\Phi_\bullet}
\big({\tfrac 1 \lambda}  s^{-\frac 1n} (u^\ep)^*(s)\big)\, ds\bigg)\,.
\end{align}
Choose  $\lambda = { {\kappa_3}}/c_1$, where $\kappa_3$ and $c_1$ are the constants appearing in inequalities
\eqref{optint} and \eqref{equivdiamond}, respectively.
The following chain holds:
\begin{align}
\label{june31} \int_0^{|\om|} &\widehat {\Phi_\bullet} \bigg(\frac 1\lambda
 s^{-\frac 1n} (u^\ep)^*(s)\bigg)\, ds\  \leq
  \int_0^{|\om|} \widehat {\Phi_\bullet} \bigg({ \frac 1 \lambda} s^{-\frac 1n} v^*(s)\bigg)\, ds
 \\ \nonumber & \leq
  \int _{\om ^\bigstar}\Phi _\bullet \bigg( {\frac {\kappa_3} \lambda}|\nabla v|\bigg)\, dx
 %\\ \nonumber
\leq
\int_0^{|\om|}\Phi _\bullet \bigg({ \frac {\kappa_3} \lambda}\bigg(\Psi
_\Diamond ^{-1}\bigg(\frac{s^{1/n}}{n\omega
_n^{1/n}}f^{**}(s)\bigg)\bigg)\bigg)\, ds
\\ \nonumber & \leq
\int_{s_0}^{|\om|}\Phi _\bullet \bigg({ \frac {\kappa_3} \lambda} \bigg(\Psi
_\Diamond ^{-1}\bigg(\frac{s^{1/n}}{n\omega
_n^{1/n}}f^{**}(s)\bigg)\bigg)\bigg)\, ds
%\\ \nonumber & \quad
+
\int_0^{|\om|}\Phi _\circ \bigg({ \frac {\kappa_3} \lambda} \bigg(\Psi
_\Diamond ^{-1}\bigg(\frac{s^{1/n}}{n\omega
_n^{1/n}}f^{**}(s)\bigg)\bigg)\bigg)\, ds
\\ \nonumber & \leq
|\om|\Phi _\circ (1) + \int_0^{|\om|}\Phi _\Diamond
\bigg(\frac{{ {\kappa_3}}}{c_1 \lambda} \bigg(\Psi _\Diamond
^{-1}\bigg(\frac{s^{1/n}}{n\omega
_n^{1/n}}f^{**}(s)\bigg)\bigg)\bigg)\, ds
\\ \nonumber &
 =
|\om|\Phi _\circ (1)
  +
\int_0^{|\om|}\Phi _\Diamond \bigg(\Psi _\Diamond
^{-1}\bigg(\frac{s^{1/n}}{n\omega _n^{1/n}}f^{**}(s)\bigg)\bigg)\,
ds
\\ \nonumber &
\leq
|\om|\Phi _\circ (1)
  +
\int_0^{|\om|}\widetilde{\Phi _\Diamond
}\bigg(\frac{2s^{1/n}}{n\omega _n^{1/n}}f^{**}(s)\bigg) \, ds
%\\ \nonumber &
 \leq
|\om|\Phi _\circ (1)
 +
\int_0^{|\om|}\widetilde{\Phi _\circ
}\bigg(\frac{2s^{1/n}}{c_1n\omega _n^{1/n}}f^{**}(s)\bigg) \, ds\,.
\end{align}
Note that the first inequality is due to \eqref{june29}, the second
to \eqref{optint}  (with $\Phi_\circ$ replaced by $\Phi_\bullet$),
the third by \eqref{june25}, the fourth by the definition of
$\Phi_\bullet$, where $s_0 \in [0,|\om|]$ is chosen in such a way
that
$$s_0 = \inf \bigg\{s \in [0, |\om|]: {\frac {\kappa_3} \lambda} \bigg(\Psi _\Diamond
^{-1}\bigg(\frac{s^{1/n}}{n\omega _n^{1/n}}f^{**}(s)\bigg)\bigg)
\leq 1 \bigg\},$$ the fifth by \eqref{equivdiamond},  the
equality holds owing to the very choice of $\lambda$,   the sixth
inequality   is a consequence of property \emph{(v)} of Lemma  \ref{lem:aux-aniso-est}, and the last one
follows via \eqref{equivdiamond} again.
\\ On the other hand,
\begin{align}\label{june32}
\int_0^{|\om|}\widetilde {\Phi_\bullet}\big(\lambda s^{\frac
1n}f^{**}(s)\big)\, ds & = \int_0^{|\om|}\widetilde
{\Phi_\bullet}\bigg(\frac {{\kappa_3}} {c_1} s^{\frac 1n}f^{**}(s)\bigg)\,
ds \\ \nonumber & \leq \int_{s_1}^{|\om|}\widetilde
{\Phi_\bullet}\bigg(\frac {{\kappa_3}}{c_1} s^{\frac 1n}f^{**}(s)\bigg)\,
ds + \int_0^{|\om|}\widetilde {\Phi_\circ}\bigg(\frac {{\kappa_3}}{c_1}
s^{\frac 1n}f^{**}(s)\bigg)\, ds
\\ \nonumber & \leq |\om| \widetilde {\Phi_\circ} (t_1) + \int_0^{|\om|}\widetilde {\Phi_\circ}\bigg(\frac {{\kappa_3}}{c_1} s^{\frac 1n}f^{**}(s)\bigg)\, ds\,,
\end{align}
where
$$s_1 = \inf \bigg\{s \in [0, |\om|]: \frac {{\kappa_3}}{c_1} s^{\frac 1n}f^{**}(s) \leq t_1 \bigg\}.$$
The rightmost sides in inequalities \eqref{june31} and
\eqref{june32} are finite, owing to  assumption \eqref{hpf}, and
only depend on $f$, $n$ and $\Phi$. From inequalities
\eqref{en:eq:imp-e} and \eqref{june30} one thus deduces that there
exists a constant $C$, depending on these data, such that
\begin{equation}\label{june34}
\int_{\Omega } \Phi(\nabla u ^\ep) \,dx+\int_{\Omega }\ep A(|\nabla
u ^\ep|)\, dx+\int_{\Omega }\ep \wt{A}(A'(|\nabla u ^\ep |))\,dx
\leq C
\end{equation}
for $\ep \in (0,1)$. Assertions {\it (i)--(ii)} follow from
\eqref{june34}.
 Assertion {\it (iii)} follows on coupling inequality~\eqref{june34}
with assumption~\eqref{A2''}.
% Assertion {\it (iii)} follows on coupling {\it (i)}
% with assumption{\color{teal}s \eqref{A2''} and~\eqref{A2'}.}
\\  Assume next that condition \eqref{hpfbis}
holds. Then $W_0^1L^\Phi(\Omega)\to L^\infty(\Omega)$,
and from equations
\eqref{en:eq:imp-e}, \eqref{june29},  \eqref{B-Wter}, \eqref{june24}
and \eqref{equivdiamond} we obtain that
\begin{align}\label{june50}
\int_{\Omega } &\Phi(\nabla u ^\ep) dx   +\int_{\Omega }\ep
A(|\nabla u ^\ep|)\, dx+\int_{\Omega }\ep \wt{A}(A'(|\nabla u ^\ep
|))\,dx   \leq  \|f\|_{L^1(\om)}\|u^\ep\|_{L^\infty(\om)}  \\
\nonumber & \leq  \|f\|_{L^1(\om)}\|v\|_{L^\infty(\om^\bigstar)}
\leq C \|f\|_{L^1(\om)}\|\nabla v\|_{L^{\Phi_\circ}(\om^\bigstar)}
\\ \nonumber & \leq  C \|f\|_{L^1(\om)} \bigg\|\Psi _\Diamond
^{-1}\bigg(\frac{s^{1/n}}{n\omega
_n^{1/n}}f^{**}(s)\bigg)\bigg\|_{L^{\Phi_\circ}(0,|\om|)}
\leq  C \|f\|_{L^1(\om)} \bigg\|\Psi _\Diamond
^{-1}\bigg(\frac{s^{1/n}}{n\omega
_n^{1/n}}f^{**}(s)\bigg)\bigg\|_{L^{\Phi_\Diamond}(0,|\om|)}
\end{align}
for $\ep \in (0,1)$, and for some constants $C$ and $C'$ depending on
$n$,  $\Phi_\circ$ and $|\om|$. We claim that the last norm on the
rightmost side of inequality \eqref{june50} is finite, since $f \in
L^1(\om)$. This is a consequence of the fact that  $s^{1/n} f^{**}(s)\leq s^{-1/n'} \|f\|_{L^1(\om)}$ for $s \in (0, |\om|)$, of property \emph{(v)} of Lemma \ref{lem:aux-aniso-est}, of equation
\eqref{equivdiamond}, and of  \eqref{oct1}, which is equivalent to \eqref{intconv}.
%Altogether, the rightmost side of inequality
%\eqref{june50} is finite, whence
Therefore, inequality \eqref{june34} holds
also in this case. One can then conclude as above.
\end{proof}

\subsection{A Minty--Browder--type result}

The following proposition provides us with  an anisotropic version of a  classical result, known as the Minty-Brown monotonicity trick.  It will be applied later, in~the~identification of  limits of certain nonlinear expressions in an approximation process.
\begin{proposition}{\rm {\bf [A monotonicity trick]}}\label{prop:mon-e}
 Let $\om$ be a measurable set in $\rn$  with $|\om|< \infty $. Assume that the  Carath\'eodory function  $a: \om \times \rn \to \mathbb R$
satisfies condition \eqref{A2''} for some $N$-function
$\Phi$.  Suppose that there exist functions
\begin{equation}\label{june37}
{Y}\in L^{\wt{\Phi}}(\Omega;\rn)\quad\text{ and }\quad U\in
L^{\Phi}(\Omega;\rn)
\end{equation} such that
\begin{equation}
\label{anty-mon} \int_\Omega \big(Y -a(x,V)\big)\cdot(U
-V)\,dx\geq0\qquad\hbox{for every $V\in L^\infty(\Omega;\rn)$}.
\end{equation}
Then
\begin{equation}\label{june35}
a(x,U(x))=Y(x)\qquad\hbox{for a.e. $x \in  \Omega$.}
\end{equation}
\end{proposition}
\begin{proof} Define the increasing family $\{\om _j\}$ of invading subsets of $\om$ as $\Omega_j=\{x\in\Omega:\ |U(x)|\leq j\}$
%\begin{equation*}
%%\label{omK}
%\Omega_j=\{x\in\Omega:\ |U(x)|\leq j\}
%\end{equation*}
for~$j \in \mathbb N$. Fix any $j, k \in \mathbb N$ with $j<k$. An
application of inequality \eqref{anty-mon}, with
$V=U \chi_{\Omega_k}+\sigma Z\chi_{\Omega_j}$
for any  $\sigma \in (0,1)$ and any function $Z\in
L^\infty(\Omega;\rn)$, yields
\begin{equation*}
%\label{Ak-po-mon}
\int_\Omega( Y - a(x,U\chi_{\Omega_k}  +\sigma Z\chi_{\Omega_j})
)\cdot( U-U\chi_{\Omega_k} -\sigma Z\chi_{\Omega_j})\, dx\geq  0.
\end{equation*}
%Here, $\chi_E$ stands for the characteristic function of the set
%$E$.
The last inequality is equivalent to
\begin{equation}
\label{A-po-mon}  \int_{\Omega\setminus\Omega_k} \big(Y
-a(x,0)\big)\cdot U\,dx + \sigma \int_{ \Omega_j} (a(x,U+\sigma Z)-Y
)\cdot Z\,  dx\geq  0.
\end{equation}
The first integral on the left-hand side of inequality
\eqref{A-po-mon} tends to zero as $k\to\infty$. Indeed, assumption \eqref{A2''}
implies  that $\big(Y -a(x,0)\big) \cdot  U \in L^1(\om)$, and hence the
convergence follows owing to assumption \eqref{june37} and
H\"older's inequality~\eqref{holder}. Thus, passing to the limit as
$k\to\infty$ in inequality \eqref{A-po-mon} and dividing by
$\sigma$ the resultant  inequality tells us that
\begin{equation*} \int_{ \Omega_j} (a(x,U+\sigma Z)-Y )\cdot Z\,  dx\geq  0.
\end{equation*}
Clearly,
\begin{equation}
\label{A-poz} \lim _{\sigma\to 0^+} a(x,U+\sigma U) = a(x,U)\quad
\hbox{for a.e. $x \in \Omega_j$.}
\end{equation}
Moreover, by \eqref{A2''},
\begin{equation}
\label{june38} \sup_{\sigma\in(0,1)} \int_{ \Omega_j}
\wt{\Phi}\left( c_\Phi a(x,U+\sigma Z)\right)dx\leq    \int_{
\Omega_j} \sup_{\sigma \in(0,1)} {\Phi}\left(U+\sigma
Z\right)dx+\int_{ \Omega_j} h(x)\,dx.\color{black}
\end{equation}
The integral on the right-hand side of \eqref{june38} is finite,
since the function $\sup_{\sigma \in(0,1)}(U+\sigma Z)$, and hence
also the function $\sup_{\sigma \in(0,1)} {\Phi}\left(U+\sigma
Z\right)$, is
 bounded in $\Omega_j$. By
Theorem~\ref{theo:delaVP}, the family of functions   $\{a(x,U+\sigma
Z)\}_{\sigma \in (0,1)}$ is       uniformly integrable in
$\Omega_j$. Hence, owing to Theorem~\ref{theo:VitConv},
\[\lim _{\sigma \to 0^+} a(x,U+\sigma Z) = a(x,U)\quad \hbox{in $L^1(\Omega_j{;}\rn)$}.\]
Thus,
\begin{equation*} \lim _{j \to \infty}\int_{ \Omega_j} (a(x,U+\sigma Z)-Y )\cdot Z \, dx=  \int_{ \Omega_j} (a(x,U)-Y )\cdot Z \, dx.
\end{equation*}
Consequently,
\begin{equation*}  \int_{ \Omega_j} (a(x,U)-Y )\cdot Z\,  dx \geq 0
\end{equation*}
for every $Z\in L^\infty(\Omega{;}\rn)$. The choice of
\[Z=\left\{\begin{array}{ll}-\frac{a(x,U)-Y }{|a(x,U)-Y |}&\ \text{if}\quad a(x,U)-Y \neq 0\\
0&\ \text{if}\quad a(x,U)-Y = 0,
\end{array}\right.\]
ensures that
\begin{equation*}  \int_{ \Omega_j} |a(x,U)-Y |\, dx\leq 0,
\end{equation*}
whence \[a(x,U(x))=Y (x)\quad \hbox{for a.e. $x\in \Omega _j$}.\]
Equation \eqref{june35} follows, owing to the arbitrariness of $j$.
\end{proof}

\subsection{Proof of existence of weak solutions}

We are now ready to accomplish the proofs of Theorem~\ref{theo:boundex} and of Proposition \ref{boundedsol}.
\begin{proof}[Proof of Theorem \ref{theo:boundex}]  Let $A$ be an $N$-function as in Propositions \ref{prop:reg-bound-e} and~\ref{prop:uniform1}, and let $\{u^\ep\} \subset  W_{0}^{1}\mathcal L^A(\Omega)\cap L^\infty(\om)$ be the family of solutions to problems
\eqref{feb10}
for $\ep \in (0,1)$. By property {\it (i)} of
Proposition~\ref{prop:uniform1},  this family is bounded in
$W^{1}_0L^\Phi(\Omega)$, and hence in $W^{1,1}_0(\Omega)$.
Therefore, it is compact in $L^1(\Omega)$, and consequently there exists a
function $u \in L^1(\om)$ and a sequence $\{u^{\ep_k}\}$ such that
$u^{\ep_k}\to u$ in $L^1(\om)$ and a.e. in $\om$.
%{\color{blue} By
%Proposition~\ref{thm_dual},   both   $L^\Phi (\om{;\rn})$ and
%$L^{\wt{\Phi}} (\om{;\rn})$ are duals of some normed spaces}.
Property {\it (i)} of Proposition~\ref{prop:uniform1} and
Theorem~\ref{theo:Banach-Alaoglu} then  ensure   that the family of
functions $\{\nabla u^{\ep_k}\}$ is weakly-$*$ compact in $L^\Phi
(\om{;\rn})$. Since $u^{\ep_k}\to u$ in $L^1(\om)$, we have that $u$ is
weakly differentiable, and its gradient agrees with the  weak-$*$
limit of $\{\nabla u^{\ep_k}\}$  in $L^\Phi (\om; \rn)$.
Similarly, property {\it (iii)} of~Proposition~\ref{prop:uniform1} and
Theorem~\ref{theo:Banach-Alaoglu} again imply   that the family of
functions $\{a(x, \nabla u^{\ep_k})\}$ is weakly-$*$ compact in $L^{\widetilde \Phi}
(\om{;\rn})$.
%{\color{magenta} Should we use that $u^\ep$ is bounded in $W^{1,\Phi}_0(\om)$, and hence weakly$*$ compact in $W^{1,\Phi}_0(\om)$be sure that the limit is a gradient?}
Finally, property {\it (i)} of
Proposition~\ref{prop:uniform1} implies, via Theorems
\ref{theo:delaVP} and \ref{theo:dunf-pet}, that the family $\{\nabla
u^\ep\}$ is  weakly compact in $L^1(\om{;}\rn)$.
% =(E^{\wt{\Phi}})^*$ for $\ep \in (0,1)$. The duality holds due to
% The de la Vall\'e-Poussin Theorem
% (Theorem~\ref{theo:delaVP}), Dunford-Pettis Theorem
% (Theorem~\ref{theo:dunf-pet}), and the fact that $\Phi$ is an
% $N$-function (according to Definition~\ref{def:Nf}) imply that for
% each the sequence $( u^\theta)_{\theta\in(0,1)}$  is equiintegrable
% in $W^{1,1}_0(\Omega)$. Therefore,
Altogether, there exists a decreasing sequence $\{\ep_k\}$,
fulfilling   $\ep_k \to 0^+$, and  functions $u\in W^{1}_0
L^{\Phi}(\Omega)$ and $Y \in L^{\wt{\Phi}}(\Omega{;}\rn)$ such that
\begin{alignat}{4}
\label{andrea31}
u^{\ep_k} &\to  u\quad &&\text{in $L^1(\Omega)$ and a.e. in $\Omega$},\\
\label{conv:untLi-e}
u^{\ep_k} &\xrightharpoonup{} u\quad &&\text{weakly in $W^{1,1}(\Omega)$},\\
\label{limDut-e}
\nabla u^{\ep_k} &\xrightharpoonup* \nabla u\quad && \text{weakly-$*$ in } L^\Phi(\Omega{;}\rn),\\
\label{limaDut-e}
a (x,\nabla u^{\ep_k}) &\xrightharpoonup* Y \quad && \text{weakly-$*$ in } L^{\wt{\Phi}}(\Omega{;}\rn).
\end{alignat}
By the weak formulation of problem \eqref{feb10}
%~\eqref{weak-reg-bound-e} of
with $\ep= \ep_k$,
\begin{equation}\label{june45}
\int_\Omega a(x,\nabla u ^{\ep_k})\cdot \nabla \vp + {\ep_k}
A'(|\nabla u ^{\ep_k}|) \frac{\nabla u ^{\ep_k}}{|\nabla u
^{\ep_k}|}\cdot \nabla \vp\, dx = \int_\Omega f\,\vp\, dx
\end{equation}
for every   $\vp\in  W_0^{1}\mathcal L^{A}(\Omega)$. Notice that any such $\vp$ is automatically bounded by the classical Sobolev embedding, since our assumptions on $A$ imply that $A(t)\geq t^q$ near infinity for some $q>n$. 
%$\vp\in {\mathcal W}_0^{1}L^{A}(\Omega)$.
We begin by
observing that
\begin{equation} \label{thetamvanish}
\lim_{k \to \infty}\int_\Omega {\ep_k} A'(|\nabla u
^{\ep_k}|)\frac{\nabla u ^{\ep_k}}{|\nabla u ^{\ep_k}|} \cdot\nabla
{\vp}\,dx =0\qquad\text{for every}\quad {\vp}\in C_0^\infty(\Omega).
\end{equation}
To verify this assertion, consider, for fixed $j\in \mathbb N$, the set
\[\Omega^{\ep_k}_{j}=\{x\in \Omega:\quad |\nabla u ^{\ep_k}|\leq j\}.\]
Plainly,
\begin{multline}\label{june40}
\int_\Omega{\ep_k}A'(|\nabla u ^{\ep_k}|)\frac{\nabla u
^{\ep_k}}{|\nabla u ^{\ep_k}|} \cdot\nabla  {\vp}\,dx
\\ =\int_{\Omega^{\ep_k}_{j}}{\ep_k}A'(|\nabla u ^{\ep_k}|)\frac{\nabla u ^{\ep_k}}{|\nabla u ^{\ep_k}|} \cdot\nabla  {\vp}\,dx +\int_{\Omega \setminus\Omega^{\ep_k}_{j}}{\ep_k} A'(|\nabla u ^{\ep_k}|)\frac{\nabla u ^{\ep_k}}{|\nabla u ^{\ep_k}|} \cdot\nabla  {\vp}\,dx .
\end{multline}
Inasmuch as $A'$ is a non-decreasing function,
\begin{equation}\label{june41}
\limsup_{k \to
\infty}\bigg|\int_{\Omega^{\ep_k}_{j}}{\ep_k}A'(|\nabla u
^{\ep_k}|)\frac{\nabla u ^{\ep_k}}{|\nabla u ^{\ep_k}|} \cdot\nabla
{\vp}\,dx\bigg| \leq |\Omega|\|\nabla
{\vp}\|_{L^\infty(\Omega)}A'(j) \lim_{k \to \infty}
{\ep_k}=0.
\end{equation}
On the other hand, since the sequence
  $\{|\nabla u^{\ep_k}|\}$ is uniformly integrable in
$L^1(\Omega)$, there exists a constant $C$, independent of $k$, such
that
\begin{equation}
\label{qtrc:meas-e}\sup_{k \in \mathbb
N}|\Omega\setminus\Omega^{\ep_k}_{j}|\leq  \frac Cj.
\end{equation}
Furthermore, since ${\widetilde A}$ is an $N$-function, one has that
 $\wt{A}(\lambda t)\leq\lambda
\wt{A}(t)$, provided that $t \geq 0$ and $\lambda \in (0,1)$.
Thereby, $\wt{A}({\ep_k}A'(|\nabla u ^{\ep_k}|)) \leq {\ep_k}
\wt{A}(A'(|\nabla u ^{\ep_k}|))$, and hence, by property
\textit{(ii)} of Proposition \ref{prop:uniform1}, the sequence
$\{{\ep_k} A' (|\nabla u ^{\ep_k}|)\}$ is uniformly bounded in
$L^{\wt{A}}(\om)$. Thanks to  Theorem~\ref{theo:delaVP}, the sequence
$\{{\ep_k}A'(|\nabla u ^{\ep_k}|)\}$ is uniformly integrable in
$\om$. Coupling this piece of information with \eqref{qtrc:meas-e}
implies that
\begin{multline}\label{june42}
\limsup _{j \to \infty}\bigg(\sup _{k \in \mathbb
N}\bigg|\int_{\Omega\setminus\Omega^{\ep_k}_{j}}{\ep_k}A'(|\nabla u
^{\ep_k}|)\frac{\nabla u ^{\ep_k}}{|\nabla u ^{\ep_k}|} \cdot\nabla
{\vp}\,dx\bigg|\bigg) \\ \leq \|\nabla {\vp}\|_{L^\infty(\Omega)}
\lim _{j \to \infty}\bigg(\sup_{k \in \mathbb
N}\int_{\Omega\setminus\Omega^{\ep_k}_{j}}{\ep_k}|A'(|\nabla
u^{\ep_k}|)|\,dx\bigg) =0.
\end{multline}
Equation \eqref{thetamvanish} follows from \eqref{june40},
\eqref{june41} and \eqref{june42}.
\\ Thanks to \eqref{limaDut-e} and \eqref{thetamvanish}, choosing $\varphi \in C^\infty_0(\om)$ in \eqref{june45} and passing to the limit as $k \to \infty$  yield
\begin{equation} \label{lim:theta=0-e}
\int_\Omega Y \cdot \nabla  {\vp}\, d x= \int_\Omega f {\vp}\, dx\,.
\end{equation}
Since
%the function $A(t)$ fulfills condition \eqref{gamma>Phi}, and grows essentially faster than $|\xi|^q$ for some $q>n$, we have that
$u^{\ep_k}\in W^{1}_0
L^\Phi(\Omega)\cap L^\infty(\Omega)$,
%for every $k\in \mathbb N$,
for
each $k \in \mathbb N$ the function $u^{\ep_k}$ can  be
approximated by a sequence of~functions from $C^\infty_0(\om)$ as in
Proposition~\ref{prop:approx}. On  making use of equation
\eqref{lim:theta=0-e} with $\vp$ replaced by~the~functions approximating $u^{\ep_k}$, passing to the limit in the approximating sequence, and recalling that
$Y \in L^{\wt{\Phi}}(\Omega {;} \rn)$ and that   the sequence of
approximating functions is uniformly bounded in $L^\infty(\om)$ by
$C\|u^{\ep_k}\|_{L^\infty(\om)}$ we infer that
\begin{equation} \label{lim:theta=0-e2}
\int_\Omega Y \cdot \nabla  u^{\ep_k}\, dx= \int_\Omega fu^{\ep_k}
\, dx
\end{equation}
for every $k \in \mathbb N$. Inasmuch as $u^{\ep _k}$ belongs to {\color{blue} $
W^{1}_0 \mathcal L^A(\Omega)\cap L^\infty(\om)$}, it can be used as a test function
in~the~weak formulation of problem \eqref{feb10}
%~\eqref{weak-reg-bound-e} of
with $\ep= \ep_k$.
 Therefore,
\begin{equation} \label{weak-reg-bound-e2}
\int_\Omega a(x,\nabla u ^{\ep_k})\cdot \nabla u ^{\ep_k} + {\ep_k}
A'(|\nabla u ^{\ep_k}|) |\nabla u ^{\ep_k}|\, dx = \int_\Omega f u
^{\ep_k}\, dx
\end{equation}
for every $k \in \mathbb N$. Since the second term in the integral
on the left-hand side of \eqref{weak-reg-bound-e2} is nonnegative,
equations \eqref{lim:theta=0-e2}, \eqref{weak-reg-bound-e2} and
\eqref{limDut-e} imply that
\begin{equation}\label{limsup<aln-e}
\limsup_{k \to \infty}\int_\Omega  a(x,\nabla u ^{\ep_k})\cdot\nabla
u ^{\ep_k}\,dx \leq \int_\Omega Y \cdot \nabla u  \,dx .
\end{equation}
Now, given any function $V \in L^\infty(\om {;} \rn)$, we have, by
assumption \eqref{A3},
\begin{align}
\label{june46} 0 & \leq \int_{\om} \big(a(x, V) - a(x, \nabla
u^{\ep_k})\big)\cdot (V- \nabla u^{\ep_k})\, dx
\\ \nonumber & \leq  \int_{\om} a(x, V)\cdot V \, dx - \int_{\om} a(x, V)\cdot \nabla u^{\ep_k}\, dx - \int_{\om} a(x, \nabla u^{\ep_k}) \cdot V\, dx + \int_{\om} a(x, \nabla u^{\ep_k}) \cdot \nabla u^{\ep_k}\, dx\,.
\end{align}
Passing to the limit as $k \to \infty$ on the rightmost side of
\eqref{june46}, and making use of \eqref{conv:untLi-e},
\eqref{limaDut-e} and~\eqref{limsup<aln-e} imply that
\begin{equation}
\label{mono:int-e} \int_\Omega (a(x,V)-Y )\cdot (V-\nabla u )\,dx
\geq 0.
\end{equation}
Therefore, we are in a position to apply
Proposition~\ref{prop:mon-e}, with  $U=\nabla u$, and deduce that
\begin{equation}\label{lim=ca-e}
a(x,\nabla u (x)) = Y (x) \quad \hbox{for a.e.  $x \in \om$.}
\end{equation}
Hence, in particular, $a(x,\nabla u) \in L^{\wt{\Phi}}(\om  {;} \rn)$.
Fix any test function $\vp \in C^\infty_0(\om)$. On passing to the
limit as $k \to \infty$ in equation \eqref{june45}, and exploiting
\eqref{limaDut-e}, \eqref{lim=ca-e} and \eqref{thetamvanish} one
concludes that
\begin{equation} \label{june47}
\int_\Omega a(x,\nabla u)\cdot \nabla \vp\, dx= \int_\Omega f \vp\,
dx
\end{equation}
for every $\vp \in C^\infty_0(\om)$. Equation \eqref{june47}
continues to hold for any test function   $\vp\in W^{1}_0\mathcal L^\Phi(\Omega)\cap L^\infty(\Omega)$ as in the definition of weak solution to problem \eqref{eq:main:f}. Actually, let
$\{\vp_k\}\subset C^\infty_0(\om)$ be a sequence approximating $\vp$
as in Proposition~\ref{prop:approx}. Then, from
 equation \eqref{june47} with $\vp$ replaced by $\vp_k$, we have that
 \[
\int_\Omega a (x,\nabla u)\cdot \nabla  {\vp}\, dx=\lim_{k \to
\infty}\int_\Omega a(x,\nabla u)\cdot \nabla  {\vp_k}\, dx= \lim_{k
\to \infty}\int_\Omega f\, {\vp_k}\, dx=\int_\Omega f\,{\vp}\, dx,
\]
 where the first equality holds by properties \eqref{angela1grad} and \eqref{july37},
 and the last equality,   since $\|\varphi _k\|_{L^\infty (\om)} \leq C\|\varphi \|_{L^\infty (\om)}$ for some constant $C=C(n)$ and every $k \in \mathbb N$.
\\ Finally, we have that
\begin{equation}
\label{warsaw1}\int_\Omega \Phi(\nabla u)\,dx<\infty.
\end{equation}
Indeed, since $\Phi$ is an $n$-dimensional $N$-function, inequality \eqref{warsaw1} follows, via semicontinuity, from the~convergence  in \eqref{andrea31} and estimate   \eqref{en:eq:imp-e}, whose right-hand side  is uniformly bounded as  $\ep \to 0^+$. Equation \eqref{warsaw1} ensures that, in fact, $u \in W^1_0\mathcal L^\Phi(\om)$.
\\
The uniqueness of the solution $u$ can be established along  the same lines as in the case of
approximable solutions -- see Step~6 of the proof of
Theorem~\ref{theo:main-f} in
Section~\ref{ssec:proof-main}. We shall not reproduce it here, for~brevity.
 \end{proof}

%\todo{Add one sentence of explanation}

\smallskip

\begin{proof}[Proof of Proposition~\ref{boundedsol}] Let $u^{\ep_k}$ be as in the proof of Theorem  \ref{theo:boundex}. By property ~\eqref{andrea31}, one has that $u^{\ep_k}\to u$ a.e. in $\Omega$. Moreover, inequality~\eqref{june29} implies that $\| u^{\ep_k}\|_{L^\infty(\Omega)}\leq \| v\|_{L^\infty(\Omega)}$. The norm $ \| v\|_{L^\infty(\Omega)}$ can be estimated on making use of  equation~\eqref{june23}. Thanks to  assumption ~\eqref{boundcond}, passing to the limit  as $k\to \infty$ in the resultant estimate yields inequality \eqref{boundcond1}.
\end{proof}

\section{Approximable solutions: proof of Theorems~\ref{theo:main-f} and~\ref{theo:main-mu}}\label{sec:main-proof}

Proofs of  Theorems~\ref{theo:main-f} and~\ref{theo:main-mu} are presented in Subsection \ref{ssec:proof-main} below.   Their outline is reminiscent of that of the diverse contributions  on approximable solutions mentioned above, in particular of~\cite{BBGGPV}. However, some of the specific steps require substantially new ingredients, due to the nonstandard functional setting  at hand. This is especially apparent in some fundamental a priori bounds  that are the subject of the next subsection.

\subsection{A priori estimates}\label{ssec-apriori}

%
%Uniform integrability
%
%
%{\color{blue} Remark: in fact we prove the estimate below under
%assumptions that $f\in L^1$ and we can use $u$ as a test function.}

A fundamental step in the proof of Theorem~\ref{theo:main-f} amounts to an a
priori anisotropic gradient bound for the solution  $u_k$ to the
approximating problem \eqref{prob:trunc} by the $L^1$ norm of~$f_k$.
Of course, we need  such estimate to be independent of $k$. This
is a consequence of the following proposition.
% Proposition \ref{prop:aniso-est}  below.
% that is
% stated for the solution to problem \eqref{eq:main:f} with an
% arbitrary {\color{magenta}smooth}\footnote{IC: smooth or not smooth?} right-hand side $f$.

\begin{proposition}{\rm {\bf [A gradient estimate by the $L^1$ norm of the datum]}}\label{prop:aniso-est}
%Suppose that $f$ and $\Phi$ satisfy either~\eqref{hpf} or~\eqref{hpfbis} and that
 Let $\Omega$ be an open set in $\rn$ with $|\om|<\infty$. Assume that assumptions \eqref{A3}--\eqref{A2''} hold for some  $N$-function $\Phi$.
Let $\Theta$ be the function associated with $\Phi$ as in \eqref{andrea3}.  Assume $f \in L^1(\om)$ and that there exists a weak solution $u$ to~problem~\eqref{eq:main:f}. Then
\begin{equation}\label{andrea12}
\int_\Omega \Theta (\nabla u)dx\leq c
|\Omega|^{1/n}\|f\|_{L^1(\Omega)}
\end{equation}
for some constant $c=c(n)$.
\end{proposition}

\begin{proof}
 Standard properties of truncations of weakly differentiable functions ensure that, since  $u \in W^1_0\mathcal L^{{\Phi}}(\om)$,  the function $T_\tau (u-T_t(u))$ is weakly differentiable for every $t, \tau >0$, and belongs to  $W^1_0\mathcal L^\Phi(\om)\cap L^\infty (\om)$.  Thus, the function $T_\tau (u-T_t(u))$  can be used as a test function in equation~\eqref{weak-bound-e}.
%Moreover, owing to assumption \eqref{warsaw1},  one has that $\smallint _\om \Phi(\nabla T_\tau (u-T_t(u)))\, dx < \infty$. Thus, the function $T_\tau (u-T_t(u))$  can be used as a test function
%
%Since we assume either~\eqref{hpf} or~\eqref{hpfbis} and that~\eqref{warsaw1} holds,
 %the assumptions of \cite[Theorem~1]{Ci-sym} are satisfied.
 This choice of test functions is the point of departure to derive  \cite[Inequalities (5.5) and (5.6)]{Ci-sym}, which tell us that
\begin{equation}
\label{andrea7} \frac{1}{ -\mu'_{u}(t)}\leq
\frac{1}{n\omega_n^{1/n}\mu_{u}^{1/n'}(t)}
{\Psi^{-1}_\Diamond\left(\frac{-\frac{d}{dt}\int_{\{|u|>t\}}\Phi(\nabla
u)dx}{\omega_n^{1/n}\mu_{u}^{1/n'}(t)}\right)} \quad \hbox{for a.e.
$t>0$.}
\end{equation}
Here, $\mu_u$ is the distribution function of $u$ defined as in \eqref{muu}.
%, and
%$\omega_n$ denotes the Lebesgue measure of a unit ball in~$\rn$.
%\\
Multiplying through inequality~\eqref{andrea7} by~$-\tfrac{d}{dt}\smallint_{\{|u|>t\}}\Phi(\nabla
u)\,dx$   results in
\begin{multline}
\label{andrea8} \frac{-\frac{d}{dt}\int_{\{|u|>t\}}\Phi(\nabla
u)dx}{ -\mu'_{u}(t)}  \leq
\frac{-\frac{d}{dt}\int_{\{|u|>t\}}\Phi(\nabla
u)dx}{n\omega_n^{1/n}\mu_{u}^{1/n'}(t)}
{\Psi^{-1}_\Diamond\left(\frac{-\frac{d}{dt}\int_{\{|u|>t\}}\Phi(\nabla
u)dx}{n\omega_n^{1/n}\mu_{u}^{1/n'}(t)}\right)} \quad \hbox{for a.e.
$t>0$.}
\end{multline}
Now, Lemma~\ref{lem:aux-aniso-est} {\it (i)} ensures that the function $\PD\circ
\Theta_\Diamond^{-1}$ is convex. Thereby, an application of Jensen's inequality and
Lemma~\ref{lem:aux-aniso-est} {\it (ii)}
 yield
\begin{align}\label{andrea8'}
\PD\circ
\Theta_\Diamond^{-1}\left(\frac{\frac{1}{h}\int_{\{t<|{u}|<t+h\}}\Theta(
\nabla {u} )dx}{\frac{1}{h}(-\mu_{u}(t+h)+\mu_{u}(t))}\right) &\leq
\frac{\frac{1}{h}\int_{\{t<|{u}|<t+h\}}\PD\circ
\Theta_\Diamond^{-1}(\Theta( \nabla {u})
)dx}{\frac{1}{h}(-\mu_{u}(t+h)+\mu_{u}(t))}\\ \nonumber & =
\frac{\frac{1}{h}\int_{\{t<|{u}|<t+h\}}\Phi( \nabla
{u})dx}{\frac{1}{h}(-\mu_{u}(t+h)+\mu_{u}(t))} \quad \hbox{for $t, h
>0$.}
\end{align}
Passing to the limit as $h \to 0^+$ in \eqref{andrea8'} tells us
that
\begin{equation}\label{PDMDin}
\PD\circ
\Theta_\Diamond^{-1}\left(\frac{-\frac{d}{dt}\int_{\{|{u}|>t\}}\Theta(
\nabla {u})dx}{-\mu_{u}'(t)}\right)\leq
\frac{-\frac{d}{dt}\int_{\{|{u}|>t\}}\Phi( \nabla
{u})dx}{-\mu_{u_k}'(t)} \quad \hbox{for a.e. $t>0$.}
\end{equation}
On the other hand, \cite[Inequality (5.5)]{Ci-sym} implies that
\begin{equation}
\label{andrea9}
 -\frac{d}{dt}\int_{\{|{u}|>t\}}\Phi( \nabla {u})dx \leq \int_0^{\mu_{u}(t)}f^*(s)ds \quad \hbox{for a.e. $t>0$.}
\end{equation}
From \eqref{PDMDin}, \eqref{andrea8}, Lemma~\ref{lem:aux-aniso-est}
{\it (iii)}, and \eqref{andrea9} one deduces that
\begin{align}
\label{andrea10}
 \Theta_\Diamond^{-1}\left(\frac{-\frac{d}{dt}\int_{\{|{u}|>t\}}\Theta( \nabla {u})dx}{-\mu_{u}'(t)}\right)&\leq \PD^{-1}\left(\frac{-\frac{d}{dt}\int_{\{|{u}|>t\}}\Phi( \nabla {u})dx}{-\mu_{u}'(t)}\right)\\ \nonumber & \leq \PD^{-1}\left(\frac{-\frac{d}{dt}\int_{\{|u|>t\}}\Phi(\nabla u)dx}{n\omega_n^{1/n}\mu_{u}^{1/n'}(t)} \Psi^{-1}_\Diamond\left(\frac{-\frac{d}{dt}\int_{\{|u|>t\}}\Phi(\nabla u)dx}{n\omega_n^{1/n}\mu_{u}^{1/n'}(t)}\right)\right)
 \\ \nonumber & = \Psi^{-1}_\Diamond\left(\frac{-\frac{d}{dt}\int_{\{|u|>t\}}\Phi(\nabla u)dx}{n\omega_n^{1/n}\mu_{u}^{1/n'}(t)}\right)
% \\ \nonumber &
\leq \Psi^{-1}_\Diamond\left(\frac{\int_0^{\mu_{u}(t)}f^*(s)ds}{n\omega_n^{1/n}\mu_{u}^{1/n'}(t)}\right)%\qquad  \hbox{for a.e. $t>0$.}
\end{align}
for a.e. $t>0$.
Hence,
\begin{align}
\label{andrea11} -\frac{d}{dt}\int_{\{|{u}|>t\}}\Theta( \nabla
{u})dx \leq -\mu_{u}'(t) \Theta_\Diamond \circ
\Psi^{-1}_\Diamond\left(\frac{\int_0^{\mu_{u}(t)}f^*(s)ds}{n\omega_n^{1/n}\mu_{u}^{1/n'}(t)}\right)\qquad
\hbox{for a.e. $t>0$.}
\end{align}
Now, notice that
\begin{align}
\int_{\{|{u}|>t\}}\Theta( \nabla {u})\,dx  & =  \int _\Omega
\chi_{\{\nabla u = 0\}}\Theta( \nabla {u})\,dx + \int _\Omega
\chi_{\{\nabla u \neq 0\}}\Theta( \nabla {u})\,dx
\\ \nonumber & = \int _\Omega \chi_{\{\nabla u \neq 0\}}\frac{\Theta( \nabla {u})}{|\nabla u|}|\nabla u| \,dx=
\int _t^\infty \int_{\{|{u}|=\tau\}}\frac{\Theta( \nabla
{u})}{|\nabla u|}\,d\mathcal H^{n-1}\, d\tau \qquad \hbox{for $t>0$,}
\end{align}
where the second equality holds since $\Theta(0)=0$, and the last
one by the coarea formula for Sobolev functions. Therefore, the
function
$ [0, \infty) \ni t \mapsto \smallint_{\{|{u}|>t\}}\Theta( \nabla {u})dx$
is absolutely continuous. Combining this fact with inequalities
\eqref{andrea11}   and Lemma~\ref{lem:aux-aniso-est} {\it (iv)} ensures
that
\begin{align*}\int_{\Omega}\Theta( \nabla {u})dx&=\int_0^\infty \left(-\frac{d}{dt}\int_{\{|{u}|>t\}}\Theta( \nabla {u})dx\right)dt
%\\ \nonumber &
\leq \int_0^\infty (-\mu_{u}'(t))\Theta_\Diamond\left(\Psi_\Diamond^{-1}\left(\frac{\int_0^{\mu_u(t)}f^*(s)ds}{n\omega_n^{1/n}\mu_{u_k}^{1/n'}(t)}\right)\right)\,dt
\\ \nonumber
&
\leq \int_0^{|\Omega|}  \Theta_\Diamond\left(\Psi_\Diamond^{-1}\left(\frac{\int_0^{r}f^*(s)ds}{n\omega_n^{1/n}r^{1/n'} }\right)\right)dr
% \\ \nonumber
% &
\leq \int_0^{|\Omega|} \frac{2}{n\omega_n^{1/n}r^{1/n'} }\int_0^{r}f^*(s)\,ds\, dr
 \\ \nonumber
 &\leq  \frac{2}{n\omega_n^{1/n}}\|f\|_{L^1(\Omega)}\int _0^{|\Omega|} r^{-1/n'}\, dr = 2\omega_n^{-1/n}|\Omega|^{1/n}\|f\|_{L^1(\Omega)}\,.
 \end{align*}
 Inequality \eqref{andrea12} is thus established.
\end{proof}

The next two propositions provide us with superlevel set estimates for functions  $u \in {\mathcal T}_0^{1,\Phi}(\Omega)$  and for their gradients $\nabla u$ depending of the decay of the  integrals of $\Phi (\nabla u)$ over the sublevel sets of $u$.

\begin{proposition}{\rm {\bf [Superlevel set estimate for $u$]}} \label{prop:pre-est} Let  $\Omega$ be an open set in $\rn$ with $|\om|<\infty$.  Let
$\Phi$ be an $N$-function fulfilling conditions
\eqref{conv0} and  \eqref{intdiv}. Assume that $u\in {\mathcal T}_0^{1,\Phi}(\Omega)$ and
there exist constants $K>0$ and $t_0\geq 0$ such that~
\begin{equation}
\label{podpoziomice} \int_{\{|u|<t\}}\Phi(\nabla u)\,dx\leq K t
\qquad\text{for }t>t_0.
\end{equation}
Then %there exists a constant $c=c(n)$ such that
\begin{equation}\label{|u|>r-inf}
|\{|u|\geq t\}|  \leq  \frac{Kt}{\Phi _n\big(\kappa_2 t^{\frac 1{n'}}K^{-\frac
1n}\big)} \qquad \hbox{for $t >t_0$,}
\end{equation}
where  $\Phi_n$   and $\kappa_2$ are the Young function and the constant appearing in the Sobolev
inequality~\eqref{B-W}.
\\
If condition \eqref{conv0} is not satisfied, then an analogous
statement holds, provided that $\Phi_n$ is defined as~in~\eqref{sobconj}--\eqref{H1}, with $\Phi$ modified near $0$ in such
a way that \eqref{conv0} is fulfilled. In this case, the constant
$\kappa_2$ in  \eqref{|u|>r-inf} has to be replaced by another constant depending also  on $\Phi$. Furthermore, in \eqref{|u|>r-inf} the
constant $t_0$ has to be replaced by another constant depending also
on $\Phi$, and the constant $K$ has  to~be~replaced by another constant depending
on the constant $K$ appearing in \eqref{podpoziomice}, on $\Phi$ and on $|\om|$.
\\ In any case, irrespective of whether \eqref{conv0} holds or does not, for every $\ep >0$, there exists $\overline t = \overline t (\ep, K, t_0, n, \Phi)$ such that
\begin{equation}\label{warsaw2}
|\{|u|\geq t\}| < \ep \quad \hbox{if $\ t>\overline t$.}
\end{equation}
\end{proposition}
\begin{proof}
%If necessary for~\eqref{int0PD}, instead of~$\Phi$ consider $\Phi_\Diamond^0(t)=t\PD(1)\mathds{1}_{[0,1]}(t)+\PD(t)\mathds{1}_{(1,\infty)}(t)$.
Assume first that assumption \eqref{conv0} is in force. Thanks to
the definition of $T_t$ and to property~\eqref{gengrad},
\[\int_\Omega \Phi \left(\nabla  T_t(u) \right)dx=\int_{\{|u|<t\}} \Phi \left(\nabla u  \right)dx\qquad\text{and}\qquad \{|T_t(u)|\geq t\}= \{|T_t(u)| = t\} = \{|u|\geq t\}\]
for $t>0$.
%Let $\kappa_2$ be the constant appearing in the Orlicz--Sobolev inequality \eqref{B-W}.
We have that
\begin{align}\label{sep35}
|\{|u|&\geq  t\}|\Phi_n\left(\frac{t}{\kappa_2\left(\int_{\{|u|<t\}}
\Phi(\nabla u)dy\right)^\frac{1}{n}}\right)\leq
\int_{\{|u|\geq t\}} \Phi_n\left(\frac{|T_t(u)|}{\kappa_2\left(\int_{\{|u|<t\}} \Phi(\nabla  u )dy\right)^\frac{1}{n}}\right)dx\\ \nonumber
& \leq \int_{\Omega}
\Phi_n\left(\frac{|T_t(u)|}{\kappa_2\left(\int_{\{|u|<t\}}
\Phi(\nabla  u )dy\right)^\frac{1}{n}}\right)dx \leq \int_{\Omega}
\Phi_n\left(\frac{|T_t(u)|}{\kappa_2\left(\int_{\Omega} \Phi(\nabla
T_t( u ))dy\right)^\frac{1}{n}}\right)dx.
\end{align}
 By  inequality \eqref{B-W} applied to $T_t(u)$,
\begin{equation}
\label{4.11} \int_\Omega
\Phi_n\left(\frac{|T_t(u)|}{\kappa_2\left(\int_\Omega \Phi(\nabla
T_t(u))dy\right)^\frac{1}{n}}\right)dx\leq \int_\Omega \Phi
\left(\nabla  T_t(u) \right)dx= \int_{\{|u|<t\}} \Phi \left(\nabla
u\right)dx.
\end{equation}
Combining inequalities \eqref{sep35}, \eqref{4.11} and
\eqref{podpoziomice} yields
\[|\{|u|\geq t\}|\Phi_n\bigg(\frac{t}{\kappa_2 (Kt)^\frac{1}{n}}\bigg)\leq Kt \quad \hbox{for $t>t_0$},\]
an  equivalent formulation of ~\eqref{|u|>r-inf}.
\\ Assume next that condition \eqref{conv0} fails.  Consider the {$n$-dimensional Young function} $\overline \Phi : \rn \to [0, \infty)$ defined as
\begin{align}\label{sep36}
\overline \Phi (\xi) = \begin{cases} \Xi (\xi) & \quad \hbox{if $\xi
\in \{\Phi \leq 1\}$,}
\\
\Phi (\xi) &  \quad \hbox{if $\xi \in \{\Phi > 1\}$,}
\end{cases}
\end{align}
where  $\Xi$ is the (unique)  function, which vanishes at $0$, is
linear along each half-line issued from $0$, and agrees with $\Phi$
on $\{\Phi =1\}$. Clearly, $\Phi (\xi) \leq \overline \Phi(\xi)$ for
$\xi \in \rn$, and condition \eqref{conv0} is satisfied if $\Phi$ is
replaced by $\overline \Phi$.
One has that
\begin{align}\label{sep37}
\int _{\{|u|<t\}} \overline \Phi (\nabla u)\, dx & \leq \int
_{\{|u|<t, \Phi(\nabla u) >1\}} \Phi (\nabla u)\, dx +
\int _{\{|u|<t, \Phi(\nabla u) \leq 1\}} \overline \Phi (\nabla u)\, dx \\
\nonumber & \leq \int _{\{|u|<t, \Phi(\nabla u) >1\}} \Phi (\nabla
u)\, dx + |\{|u|<t\}| \leq t(K+ |\om|)\,,
\end{align}
if $t>\max\{t_0, 1\}$. Therefore, the function $u$ satisfies
assumption \eqref{podpoziomice} with $\Phi$ replaced by $\overline
\Phi$, $K$ replaced by $K+ |\om|$, and $t_0$ replaced by $\max\{t_0,
1\}$. Consequently, inequality \eqref{|u|>r-inf} holds with $\Phi_n$
replaced by $(\overline \Phi) _n$, $K$ replaced by $K+ |\om|$, and
$t_0$ replaced by $\max\{t_0, 1\}$.
\\ Finally, in the light of    \eqref{|u|>r-inf},  inequality \eqref{warsaw2}  will follow if we show that
\begin{equation}\label{sep41}
\lim _{t \to \infty}\frac{\Phi_n(t^{\frac 1{n'}})}{t} = \infty\,.
\end{equation}
By  the definitions of $\Phi_n$ and $\Phi_\circ$, %i.e. \eqref{sobconj} and \eqref{H1},
 equation \eqref{sep41} is
equivalent to
\begin{equation}\label{sep43}
\lim _{t \to \infty}\frac{\Phi_\circ(t)}{\int _0^t\big(\frac
\tau{\Phi_\circ(\tau)}\big)^{\frac 1{n-1}}d\tau } = \infty\,.
\end{equation}
On the other hand, since $\Phi_\circ$ is an $N$-function, there
exist constants $c>0$ and $\widehat t >0$ such that
\begin{align}\label{sep44}
\int _0^t\bigg(\frac \tau{\Phi_\circ(\tau)}\bigg)^{\frac 1{n-1}}
d\tau \leq c + t \quad \hbox{if $t >\widehat t$, }
\end{align}
whence  equation \eqref{sep43} follows, owing to the behavior of $N$-functions near infinity. \iffalse

\medskip

Let us consider the case {\it (b)}. For this we define\[
F_n(t)=t^{n'}\int_0^\infty
\frac{\wt{(\PD)}(s)}{s^{1+n'}}ds\qquad\text{and}\qquad
G_n(t)=\frac{t}{F_n^{-1}(t)}.
\]
As in the previous case, by the anisotropic Sobolev-Poincar\'e
inequality, this time under the condition~\eqref{intPhi}$_2$,
applied to $T_t(u)$ and  the property of truncation~\footnote{I
don't have the reference.}, we can estimate
\[\|T_t u \|_{L^\infty(\Omega)}\leq c\, G_n\left(\int_{\Omega} \Phi(\nabla T_t u)dy\right)=c\, G_n\left(\int_{\{|u|<t\}} \Phi(\nabla u)dy\right)\leq c\,G_n(Kt).\]
Note that for $t> F_n(cK)/K$ we have $\|T_t u
\|_{L^\infty(\Omega)}<t$, which implies~\eqref{|u|>r-fin}.

\fi
\end{proof}

 \medskip

\begin{proposition}{\rm {\bf [Superlevel set estimate for $\Phi(\nabla u)$]}} \label{prop:grad-est}
Let  $\Omega$ be an open set in $\rn$ with $|\om|<\infty$.  Let
$\Phi$ be an $N$-function fulfilling conditions
\eqref{conv0} and  \eqref{intdiv}. Assume that $u\in {\mathcal T}_0^{1,\Phi}(\Omega)$ and
fulfills inequality \eqref{podpoziomice} for some
 constants $K>0$ and $t_0\geq 0$. Then there exist  constants
$c_1=c_1(n,K)$ and $s_0 = s_0 (t_0, \Phi, n, K)$   such that
\begin{equation}\label{grad-u-inf}|\{ \Phi(\nabla u) >s \}|\leq c_1  \frac{\Phi_n^{-1}(s)^{n'}}{s}\quad\text{for }\ s > s_0\,.
\end{equation}
If condition \eqref{conv0} is not satisfied, then an analogous
statement holds, provided that $\Phi_n$ is defined as~in~\eqref{sobconj}--\eqref{H1}, with $\Phi$ modified near $0$ in such
as way that \eqref{conv0} is fulfilled. In this case, the~constant
$c_1$ in  \eqref{grad-u-inf}   depends also on $\Phi$. \iffalse

\begin{itemize}
\item[(a)] If~\eqref{intPhi}$_1$ holds, then there exists a constant $K_1=K_1(n,K)$  such that
\begin{equation}\label{grad-u-inf}|\{ \Phi(\nabla u) >s \}|\vt_n^2(s)\leq K_1  \quad\text{for }\ s \geq 0,
\end{equation}
where $\vt_n^2$ is given by~\eqref{vt1vt2}.
\item[(b)] If~\eqref{intPhi}$_2$ holds, then there exists a constant $K_2=(r_0,r_1,n,K)$ such that
\begin{equation}\label{grad-u-fin}|\{ \Phi(\nabla u) >s \}|s\leq   {Kr_2} \quad\text{for }\  s \geq 0.
\end{equation}
\end{itemize}
\fi
\end{proposition}

\begin{proof}
% If necessary for~\eqref{int0PD}, instead of~$\Phi$ consider $\Phi_\Diamond^0(t)=t\PD(1)\mathds{1}_{[0,1]}(t)+\PD(t)\mathds{1}_{(1,\infty)}(t)$.
Inequality \eqref{podpoziomice} implies that
\begin{align}\label{sep39}
|\{\Phi(\nabla u)>s,\, |u| < t\}|\leq\frac{1}{s}\int_{\{\Phi (\nabla
u)>s,\, |u|< t \}} \Phi(\nabla u)\,dx\leq K\frac{t}{s}\quad\text{for
}t>t_0\text{ and }s> 0.
\end{align}
On the other hand,
\begin{equation}
\label{B>s} |\{\Phi(\nabla u )>s \}|\leq |\{ |u|\geq
t\}|+|\{\Phi(\nabla u )>s,\, |u|<t\}| \quad \hbox{for $t>0$ and
$s>0$.}
\end{equation}
From \eqref{sep39} and  \eqref{|u|>r-inf} one deduces that
\[|\{\Phi(\nabla u )>s \}|\leq   \frac{Kt}{\Phi_n(ct^\frac{1}{n'}/K^\frac{1}{n})}+K\frac{t}{s}\qquad\text{for }t>t_0 \text{ and } s>0.\]
Choosing $t=(K^{1/n}\Phi_n^{-1}(s)/c)^{n'}$ in this inequality
yields
\[|\{\Phi(\nabla u)>s \}|\leq 2\left(\frac{K}{c}\right)^{n'} \frac{(\Phi_n^{-1}(s))^{n'}}{s} \quad \hbox{for $s >  \Phi_n(ct_0^{1/n'}/K^{1/n})$,} \]
whence \eqref{grad-u-inf} follows.
\\ If condition \eqref{conv0} is not fulfilled,
 the conclusion follows on modifying the function $\Phi$ near $0$, via an~argument analogous to that  of the proof of Proposition \ref{prop:pre-est}.
\end{proof}

\subsection{Proof of existence of approximable solutions}\label{ssec:proof-main}

The proofs of the common parts of the statements of Theorems \ref{theo:main-f} and  \ref{theo:main-mu}  are  very similar. We shall provide details on the former,  and just briefly comment on the minor variants  needed for the latter.

\begin{proof}[Proof of Theorem~\ref{theo:main-f}] For clarity of presentation, we split the proof into steps.

\smallskip
\par\noindent \emph{Step 1.} \emph{Approximating problems with smooth data.}
\\ Let $\{f_k\} \subset L^\infty(\Omega)$ be a sequence  such that
\begin{equation}\label{lim-fk-to-f}
f_k\to f\qquad\text{in}\quad L^1(\Omega)\qquad\text{and}\qquad
\|f_k\|_{L^1(\Omega)}\leq 2 \|f\|_{L^1(\Omega)}.\end{equation} By
Theorem~\ref{theo:boundex}, there exists a (unique) weak solution
$u_k\in  W_0^1\mathcal L^\Phi(\Omega)$ to problem \eqref{prob:trunc}. In
particular, the very definition of weak solution tells us that
%
%
%the
%problem \begin{equation}\label{prob:trunc} \left\{\begin{array}{cl}
%-\dv a(x,\nabla u_k)= f_k &\qquad \mathrm{ in}\qquad  \Omega,\\
%u_k(x)=0 &\qquad \mathrm{  on}\qquad \partial\Omega,
%\end{array}\right.
%\end{equation}  is a direct consequence of Theorem~\ref{theo:boundex} with $f=f_k$. We have
\begin{equation}
\label{uk-weak} \int_\Omega a(x,\nabla u_k)\cdot \nabla
\vp\,dx=\int_\Omega f_k\,\vp\,dx
\end{equation}
for every $\vp\in
W^1_0\mathcal L^\Phi(\Omega)\cap L^\infty(\Omega)$.

\smallskip
\par\noindent \emph{Step 2.} \emph{A priori estimates.}\\  The following inequality holds for every $k \in \mathbb N$ and for every $t>0$:
\begin{equation}
\int_\Omega \Phi( \nabla T_t(u_k) )\,dx\leq 2 t \|f\|_{L^1(\Omega)}\label{Phiapriori}.
% \\
% {\color{red}
% \int_\Omega \widetilde{\Phi}(c_\phi a(x,\nabla T_t(u_k)) )\,dx}&\leq&
%  {\color{red}\left(2 t
% \|f\|_{L^1(\Omega)}+\|h\|_{L^1(\Omega)}\right),\label{Phi*apriori}}
\end{equation}
% \todo[inline]{Are we using this inequality?}
% where $c_\Phi$ is the constant appearing in \eqref{A2''}.
% \\
Inequality  \eqref{Phiapriori} is a
consequence of the following chain, that relies upon  assumption
%\eqref{A2'} and
\eqref{A2''}  and on the use of the test function
% * <cianchi@unifi.it> 2018-11-28T18:58:22.655Z:
%
% ^.
% * <cianchi@unifi.it> 2018-11-28T18:58:16.491Z:
%
% > \eqref{A2''}
%
% ^.
$\varphi =  T_t(u_k)$ in equation \eqref{uk-weak}:
\begin{align*}
%  {\color{red}\int_\Omega \wt{\Phi}(c_\Phi a(x, \nabla T_t(u_k)))\,dx}&
%\leq
\int_\Omega  \Phi( \nabla T_t(u_k) )\,dx &  \leq \int_\Omega
a(x,\nabla T_t(u_k))\nabla T_t(u_k)\,dx
\\ &
=\int_\Omega a(x, \nabla  u_k )\nabla T_t(u_k)\,dx =
\int_\Omega f_k  T_t(u_k)\,dx  \leq 2 t \|f\|_{L^1(\Omega)}.
\end{align*}

\smallskip
\par\noindent \emph{Step 3.} \emph{Almost everywhere convergence of functions.}
\\  There exists  a function $u\in \mathcal M(\om)$ such that (up to subsequences)
\begin{equation}
\label{conv:ae:uk-to-u}
 u_k \to  u \quad \text{a.e. in }  \Omega .
\end{equation}
\iffalse

Fix any $t>0$. Since $\Phi$ ia an $N$-function, inequality
\eqref{Phiapriori} ensures that the sequence $\nabla T_t(u_k) $ is
bounded in $L^1(\om)$. Hence, the sequence $T_t(u_k)$ is bounded in
$W^{1,1}_0(\om)$ and since the latter space is compactly embedded
into $L^1(\om)$, there exist a function $u_t \in L^1(\om)$ and a
subsequence $\{u_k\}$, still denoted by $\{u_k\}$, such that
$T_t(u_k) \to u_t$ in $L^1(\om)$ and a.e. in $\om$. If $t_1 < t_2$,
then $u_{t_1} = u_{t_2}$ in $\{u_{t_1} \leq t_1\}$. Thus, there
exists a subsequence of $\{u_k\}$, still denoted by $\{u_k\}$, and
$u\in \mathcal M(\om)$, such that \eqref{conv:ae:uk-to-u} holds.
{\color{magenta} Why is $u$ finite a.e. in $\om$?}

\fi   Indeed, let $t, \tau >0$. Then
\begin{equation}\label{sep40}
|\{|u_k -u_m|>\tau \}| \leq |\{|u_k|>t\}| + |\{|u_m|>t \}| +
|\{|T_t(u_k) -T_t(u_m)|>\tau \}|
\end{equation}
for $k,m \in \mathbb N$. Fix any $\ep >0$. Inequality
\eqref{Phiapriori} ensures, via inequality \eqref{sep40} of
Proposition \ref{prop:pre-est}, that
\begin{equation}\label{sep45}
|\{|u_k|>t\}| + |\{|u_m|>t \}| < \ep
\end{equation}
for every $k, m \in \mathbb N$, provided $t$ is sufficiently large.
Moreover,  inequality \eqref{Phiapriori} again ensures that the~sequence $\nabla T_t(u_k) $ is bounded in $L^1(\om)$. Hence, the
sequence $T_t(u_k)$ is bounded in $W^{1,1}_0(\om)$ and since the
latter space is compactly embedded into $L^1(\om)$, there exists a
subsequence, still denoted by $\{u_k\}$, such that $T_t(u_k)$
converges to some function in $L^1(\om)$. In particular, it is a
Cauchy sequence in measure, and hence
\begin{equation}\label{sep46}
|\{|T_t(u_k) -T_t(u_m)|>\tau \}|  < \ep
\end{equation}
if $k$ and $m$ are large enough. From inequalites
\eqref{sep40}--\eqref{sep46} we infer that (up to subsequences)
$\{u_k\}$ is a~Cauchy sequence in measure, whence
\eqref{conv:ae:uk-to-u} follows.

\smallskip
\par\noindent \emph{Step  4.} \emph{$\{\nabla u_k\}$  is a Cauchy sequence in measure.} \\
An application of Proposition \ref{prop:aniso-est} with  $f$ and $u$
replaced by $f_k$ and  $u_k$ yields, via \eqref{lim-fk-to-f},
\begin{equation}\label{july51}
\int_\Omega \Theta (\nabla u_k)\,dx\leq c
|\Omega|^{1/n}\|f\|_{L^1(\Omega)}
\end{equation}
for some constant $c=c(n)$ and every $k \in \mathbb N$. Here,
$\Theta$ is the function given by \eqref{andrea3}. Define the
function  $\Theta_- :\rp\to\rp$ by
\begin{equation}\label{M-m}  \Theta_-(s)=\inf_{|\xi|=s}\Theta(\xi).
\end{equation}
Namely, $\Theta_-$ is the largest radially symmetric minorant of
$\Theta$. Note that $\Theta_-$ is a strictly increasing function
vanishing at $0$.
 Let $\ep >0$. Given any $t, \tau, s >0$, one has that
\begin{multline}\label{grad-uk-um}
 |\{\Theta_-(|\nabla u_k-\nabla u_m|)>t\}| \leq  |\{\Theta_-(|\nabla u_k|)>\tau\}|+|\{\Theta_-(|u_m|)>\tau\}|+|\{| u_k - u_m |>s\}| \\  +|\{|  u_k - u_m |\leq s,\,\Theta_-(|\nabla u_k|)\leq\tau,\,\Theta_-(|\nabla
u_m|)\leq \tau,\,\Theta_-(|\nabla u_k-\nabla u_m|)>t\}|\,.
\end{multline}
Owing to inequality \eqref{july51},
\begin{align}\label{sep50}
t|\{\Theta_-(|\nabla u_k|)>t\}|\leq \int_\Omega \Theta_-(|\nabla
u_k|)dx\leq \int_\Omega \Theta(\nabla u_k)dx\leq
c {\color{blue} |\Omega|^{1/n}}\|f\|_{L^1(\Omega)}
\end{align}
for $k \in \mathbb N$. Thus,
\begin{equation}\label{grad-uk'n'um-est}
|\{\Theta_-(|\nabla u_k|)>\tau\}| + |\{\Theta_-(|\nabla
u_m|)>\tau\}|<\ve
\end{equation}
for every $k, m \in \mathbb N$, provided that $\tau $ is large
enough. Next, set
\begin{equation}\label{G}
G=\{|  u_k - u_m |\leq s,\,\Theta_-(|\nabla
u_k|)\leq\tau,\,\Theta_-(|\nabla u_m|)\leq \tau,\,\Theta_-(|\nabla
u_k-\nabla u_m|)>t\}\,,
\end{equation}
and define
\[
S=\{(\xi,\eta)\in\rn\times\rn:\ |\xi|\leq\tau,\ |\eta|\leq \tau,\
|\xi-\eta|\geq t\}\,,
\]
 a compact set. Consider the function
$\psi:\Omega\to[0,\infty)$ given by
\[\psi(x)=\inf_{ (\xi,\eta)\in S}\left[\left(a(x,\xi)-a(x,\eta)\right)\cdot(\xi-\eta)\right].\]
The monotonicity assumption \eqref{A3} and the continuity of the
function $\xi\mapsto a(x,\xi)$ for a.e. $x\in\Omega$ on~the~compact set $S$ ensure that $\psi\geq 0$ in $\Omega$ and
$|\{\psi(x)=0\}|=0$. Moreover,
\begin{align}
\label{psi-a-priori}
\int_G\psi(x)\,dx&\leq \int_G \left(a(x,\nabla u_k)-a(x,\nabla u_m)\right)\cdot (\nabla u_k-\nabla u_m)\,dx\\ \nonumber
&\leq \int_{\{|u_k-u_m|\leq s\}} \left(a(x,\nabla u_k)-a(x,\nabla u_m)\right) \cdot (\nabla u_k-\nabla u_m)\,dx\\ \nonumber
&= \int_{\Omega} \left(a(x,\nabla u_k)-a(x,\nabla u_m)\right)\cdot (\nabla T_s( u_k- u_m))\,dx\\ \nonumber
%&= \lim_{\delta\to 0}\int_{\Omega} \left(a(x,\nabla u_k)-a(x,\nabla u_m)\right)(\nabla T_s( u_k- u_m))_\delta\,dx\\
%&= \lim_{\delta\to 0}\int_{\Omega} \left(f_k-f_m \right) (T_s( u_k- u_m))_\delta \,dx\\
&= \int_{\Omega} \left(f_k-f_m \right) T_s( u_k- u_m) \,dx\leq
4s\|f\|_{L^1(\Omega)},
\end{align}
where the last but one equality follows on making use of the test
function  $T_s(u_k-u_m) $ in~\eqref{prob:trunc} and in~the~corresponding equation with $k$ replaced by $m$, and subtracting the
resultant equations. Inequality~\eqref{psi-a-priori} and the
properties of the function $\psi$ ensure that, if $s$ is chosen
sufficiently small, then
\begin{equation}\label{G-small}
|\{|  u_k - u_m |\leq s,\,\Theta_-(|\nabla
u_k|)\leq\tau,\,\Theta_-(|\nabla u_m|)\leq \tau,\,\Theta_-(|\nabla
u_k-\nabla u_m|)>t\}|< \ep\,.
\end{equation}
On the other hand,  since $\{u_k\}$ is a Cauchy sequence in measure,
\begin{equation}\label{sep52}
|\{| u_k - u_m |>s\}| < \ep\,,
\end{equation}
if $k$ and $m$ are  sufficiently large. From inequalities \eqref{grad-uk-um},
\eqref{grad-uk'n'um-est}, \eqref{G-small},  and \eqref{sep52}, we infer that $\{\nabla u_k\}$ is a Cauchy sequence in measure.

\smallskip
\par\noindent \emph{Step  5.} \emph{Almost everywhere convergence of gradients.}
\\
  Our aim here  is to show that the function $u$ obtained in Step~3  belongs to the class $\mathcal T ^{1,\Phi}_0(\om)$, and that $\nabla u_k \to \nabla u$ $\hbox{a.e. in $\om$}$
(up to subsequences), where $\nabla u$ denotes the \lq\lq
generalized gradient\rq\rq{}  of $u$ in the sense of the~function $Z_u$
appearing in \eqref{gengrad}.
\\ Since $\{\nabla u_k\}$ is a Cauchy sequence in measure, there exist a subsequence (still indexed by $k$) and a~function
$W\in\mathcal{M}(\Omega; \rn)$ such that
\begin{equation}\label{lim-grad-uk-W}
\nabla u_k\to  W \qquad\text{a.e. in }\Omega.
\end{equation}
We have  to show that
\begin{equation}\label{sep53}
\nabla u = W\,
\end{equation}
and
\begin{equation}\label{sep63}
\chi_{\{|u|<t\}}  W \in L^\Phi(\om ;\rn) \quad \hbox{for every $t>0$.}
\end{equation}
To this purpose, observe that estimate \eqref{Phiapriori} ensures
that, for each fixed $t>0$, the sequence $\{\nabla T_t(u_k)\}$ is
bounded in $L^\Phi(\om ;\rn)$. Hence, by Theorem \ref{theo:dunf-pet}, the
sequence $\{\nabla T_t(u_k)\}$ is compact in $L^\Phi(\om ;\rn)$ with
respect to the weak-$*$ convergence. Since $T_t(u_k) \to T_t(u)$ in
$L^1(\om)$, the function $T_t(u)$ is weakly differentiable, and its
gradient agrees with the weak-$*$ limit of $\{\nabla T_t(u_k)\}$.
\\ Thus, for each fixed $t>0$, there exists a subsequence of $\{u_k\}$, still indexed by $k$, such that
\begin{equation}\label{sep60}
\lim _{k \to \infty} \nabla T_t(u_k)= \lim _{k \to \infty}
\chi_{\{|u_k|<t\}} \nabla u_k = \chi_{\{|u|<t\}} W \quad \hbox{a.e.
in $\om$,}
\end{equation}
and
\begin{equation}\label{sep61}
\lim _{k \to \infty} \nabla T_t(u_k) = \nabla T_t(u) \quad
\hbox{weakly-$*$ in $L^\Phi(\om ;\rn)$.}
\end{equation}
Therefore,  $\nabla T_t(u) = \chi_{\{|u|<t\}} W \quad \hbox{a.e. in $\om$,}$
whence equations \eqref{sep53}  and \eqref{sep63} follow, owing to
\eqref{gengrad}.

\smallskip
\par\noindent \emph{Step   6.} \emph{Uniqueness of the solution.}
\\
Suppose  that $u$ and $\bu$ are approximable solutions to
problem~\eqref{eq:main:f}.   Thus,    there exist sequences
$\{f_k\}$ and $\{\overline f_k\}$ in $L^\infty (\Omega)$, such that
$f_k\to f$ and $\overline f_k\to f$ in $L^1(\Omega)$ and weak
solutions $u_k$ to~\eqref{prob:trunc} and $\bu_k$ to
\begin{equation}\label{prob:trunc-b}
\left\{\begin{array}{cl}
-\dv \, a(x,\nabla \overline u_k) = \overline f_k &\qquad \mathrm{ in}\qquad  \Omega\\
\overline u_k(x)=0 &\qquad \mathrm{  on}\qquad \partial\Omega,
\end{array}\right.
\end{equation}  such that  $u_k\to u$ and $\overline u_k\to \overline u$ a.e. in $\om$.
\\ Fix any $t>0$,  make use of $\vp=T_t(u_k-\overline u_k)$ as a test function in~\eqref{prob:trunc} and~\eqref{prob:trunc-b}, and subtract the resultant equations to obtain
\begin{equation}
\label{diff:u-bu} \int_{\{|u_k- \overline u_k|\leq t\}}(a(x,\nabla
u_k)-a(x,\nabla  \overline u_k))\cdot( \nabla u_k-\nabla \overline
u_k)\, dx=\int_\Omega (f_k- \overline f_k)\,T_t(u_k-\overline u_k)\,
dx
\end{equation}
for every $k\in\N$. The right-hand side of \eqref{diff:u-bu}  tends
to $0$ as $k \to \infty$, since $|T_t(u_k-\overline u_k)|\leq t$.
% and for $k\to\infty$ we have $
%f_k-\bar{f_k}\to 0$ in $L^1(\Omega)$\footnote{IC: weak convergence is enough!}.
As shown in Steps 3-5, one has that $u , \overline u \in \mathcal
T^{1,\Phi}_0(\om)$, and   $\{\nabla u_k\}$ and $\{\nabla \overline
u_k\}$ converge (up to subsequences) a.e.~in~$\om$ to the
generalized gradients $\nabla u$ and $\nabla \overline u$,
respectively. Thus, by assumption \eqref{A3} and Fatou's lemma,
passing to the limit in \eqref{diff:u-bu}  tells us that
\[
\int_{\{|u -\overline u |\leq t\}}(a(x,\nabla u )-a(x,\nabla \bu
))\cdot( \nabla u -\nabla  \overline u )\,dx=0.\] Consequently, by  \eqref{A3} again,
$\nabla u =\nabla  \overline u$ a.e. in $\{|u -\overline u |\leq
t\}$ for every $t>0$, whence
\begin{equation} \label{nau=nabu}
\nabla u =\nabla \overline u\quad\text{ a.e. in }\Omega.
\end{equation}
Fix any $t,\tau>0$.
%Since $\mathcal T^{1,\Phi}_0(\om) \subset \mathcal T^{1,1}_0(\om)$, we have that $T_{\tau}(u-T_t(\overline u)) \in W^{1,1}_0(\om)$.
Inequality  \eqref{anisopoinc}, applied to the
function   $T_{\tau}(u-T_t(\overline u))$,
%with
%$c= \kappa_1|\Omega|^{-\frac{1}{n}}$,
and equation
\eqref{nau=nabu} tell us that
\begin{equation}\label{438}
\int_{\om }\Phi_\circ(c|T_{\tau}(u-T_t(\overline u))|)\, dx \le
\left( \int_{\{t<|u|<t+\tau\} }\Phi( \nabla u)\, dx+
\int_{\{t-\tau<|u|<t\} }\Phi (\nabla u)\, dx \right ),
\end{equation}
where $c= \kappa_1|\Omega|^{-\frac{1}{n}}$, and $\kappa _1$ is the constant appearing in \eqref{anisopoinc}.
We claim that, for each $\tau>0$, the right--hand side of \eqref{438}
converges to 0 as $t\rightarrow\infty$. To verify this claim, choose
the test function $\varphi =T_{\tau}(u_k-T_t(u_k))$ in equation
\eqref{uk-weak}  and deduce that
\begin{equation}\label{439}
\int_{\{t<|u_k|<t+\tau\}}\Phi(\nabla u_k)\, dx\le
\int_{\{t<|u_k|<t+\tau\}} a(x, \nabla u_k) \cdot \nabla u_k\, dx\le
\tau \int_{\{|u_k|>t\}}|f_k|\, dx.
\end{equation}
Passing to the limit as $k\rightarrow\infty$ in \eqref{439}
%,
%and making use of Fatou's lemma on the leftmost side and of the
%dominated convergence theorem on the rightmost side yield
yields, by Fatou's lemma,
\begin{equation}\label{440}
\int_{\{t<|u|<t+\tau\}}\Phi(\nabla u)\, dx\le\tau
\int_{\{|u|>t\}}|f|\, dx.
\end{equation}
Thereby, the first integral on the right-hand side of \eqref{438}
approaches 0 as $t\rightarrow \infty$. An analogous argument
%involving the test function $\Phi=T_{t-\tau}(u_k-\T(u_k))$
implies that also the last integral in \eqref{438} tends to 0 as
$t\rightarrow \infty$.  On the other hand,
$$
\lim_{t \to \infty} T_{\tau}(u-\T(\overline u))   =
T_{\tau}(u-\overline u)  \quad \hbox{a.e. in } \om\,.
$$
From \eqref{438}, via Fatou's lemma, we thus infer that
\begin{equation}\label{441}
\int_{\Omega}\Phi_\circ (c |T_{\tau}(u-\overline u)|)\,dx=0
\end{equation}
for~every $\tau>0$. Since $\Phi_\circ$ vanishes only at $0$,
equation \eqref{441}  ensures that  $T_{\tau}(u-\overline u)=0$ a.e.
in $\om$ for every $\tau >0$, whence
$u = \overline u$ a.e. in $\om$.

\smallskip
\par\noindent \emph{Step 7.} \emph{Property \eqref{umarc}  holds. }
\par\noindent
Choosing $t=0$ in inequality \eqref{440} tells us that $u$ satisfies assumption \eqref{podpoziomice} of~Proposition~\ref{prop:pre-est} with $K= \|f\|_{L^1(\om)}$. By Propositions~\ref{prop:pre-est}
and~\ref{prop:grad-est}, the solution $u$ fulfills inequalities
\eqref{|u|>r-inf} and \eqref{grad-u-inf}. These inequalities in turn
imply \eqref{umarc}.
\end{proof}

\begin{proof}[Proof of Theorem \ref{theo:main-mu}]
The proof follows exactly along the same lines as Steps 1--5 and 7 of the proof of Theorem \ref{theo:main-f}. One has just to begin with a sequence $\{f_k\} \subset L^\infty(\om)$, which is weakly-$\ast$ convergent to $\mu$ in the space of measures, and such that $\|f_k\|_{L^1(\om)} \leq 2  \|\mu\|(\Omega)$. Such a sequence can be defined, for~instance, as in \eqref{xid}, with $U(y)dy$ replaced by $d\mu (y)$. Of course, the quantity $\|f\|_{L^1(\om)}$ has then to~be replaced by $ \|\mu\|(\Omega)$ throughout.
\\ Let us just point out that the proof of uniqueness, namely of Step~6 of  Theorem~\ref{theo:main-f}, fails in the present situation since, for instance, it is not guaranteed that the right-hand side of equation~\eqref{diff:u-bu} approaches~$0$ as $k \to \infty$.
\end{proof}

\tocless\section{Compliance with Ethical Standards}\label{conflicts}

\smallskip
\par\noindent
{\bf Funding}. This research was partly funded by:
\\ (i)
Italian Ministry of University and Research (MIUR), Research Project Prin 2015 \lq\lq{}Partial differential equations and related analytic-geometric inequalities\rq\rq{},   number 2015HY8JCC;
\\ (ii) GNAMPA   of the Italian INdAM - National Institute of High Mathematics (grant number not~available);
\\ (iii) NCN Polish  grant,  number 2016/23/D/ST1/01072.

\smallskip
\par\noindent
{\bf Conflict of Interest}. The authors declare that they have no conflict of interest.

\bigskip

\let\oldaddcontentsline\addcontentsline% Store \addcontentsline
\renewcommand{\addcontentsline}[3]{}% Make \addcontentsline a no-op

\let\addcontentsline\oldaddcontentsline% Restore \addcontentsline

\end{document}